\documentclass[10pt]{amsart}
\usepackage[all]{xy}
\usepackage{amsmath}
\usepackage{amsfonts}
\usepackage{amssymb}
\usepackage{amscd}
\usepackage{amsthm}
\usepackage{latexsym}
\usepackage{amsbsy}
\usepackage{graphicx}
\usepackage{epstopdf}
\newtheorem{lemma}{Lemma}[section]

\newtheorem{theorem}{Theorem}[section]
\newtheorem{corollary}{Corollary}[section]

\begin{document}

\title{
Modular forms in the spectral action 
of Bianchi IX gravitational instantons
}

\author{
Wentao Fan,
Farzad Fathizadeh,
Matilde Marcolli
}

\maketitle

\begin{center}
{
\it Division of Physics, Mathematics and Astronomy \\
California Institute of Technology \\
1200 E. California Blvd., Pasadena, CA 91125, U.S.A. 
}
\end{center}

\vskip 0.2cm

\begin{center}
{\it E-mail:} wffan@caltech.edu, farzadf@caltech.edu, matilde@caltech.edu
\end{center}

\vskip 0.3 cm

\begin{abstract}

Following our result on rationality of 
the spectral action for Bianchi type-IX cosmological 
models, which suggests the existence of a rich arithmetic
structure in the action, we study the arithmetic and 
modular properties of the action  for an especially interesting 
family of metrics, namely $SU(2)$-invariant Bianchi IX gravitational 
instantons.  A variant of the previous rationality result is first presented 
for a time-dependent conformal perturbation of the 
triaxial Bianchi IX metric which serves as a general 
form of Bianchi IX gravitational instantons.  By 
exploiting a parametrization of the gravitational instantons 
in terms of theta functions with 
characteristics, we then show  that each Seeley-de Witt 
coefficient $\tilde a_{2n}$ appearing in the expansion of their 
spectral action  gives rise to vector-valued and ordinary 
modular functions of weight 2. We investigate the modular 
properties of the first three terms $\tilde a_0$, $\tilde a_2$ 
and $\tilde a_4$ by preforming direct calculations. This 
guides us to provide  spectral arguments for general terms 
$\tilde a_{2n}$, which are  based on isospectrality of the involved 
Dirac operators.  Special attention is paid to the relation of the new    
modular functions  with well-known modular forms by 
studying explicitly two different cases of metrics with pairs of rational 
parameters that, in a particular sense, 
belong to two different general 
families. In the first case, it is shown that  
by running a summation over a finite orbit of the rational 
parameters, up to multiplication by the 
cusp form $\Delta$ of weight 12, 
each term $\tilde a_{2n}$ lands in the 1-dimensional 
linear space of modular forms of weight 14.  In the second case, 
it is proved that, after a summation over a finite orbit of the 
parameters and multiplication by $G_4^4$, where $G_4$ is the Eisenstein 
series of weight 4,  each $\tilde a_{2n}$ lands 
in the 1-dimensional linear space of cusp forms of weight 18.

\end{abstract}

\date{}




\newpage

\hrulefill

\tableofcontents

\hrulefill

\vskip 1cm

\section{Introduction}

\smallskip

The main result of this paper shows that the
spectral action of Bianchi IX gravitational instantons
is arithmetic, in the sense that its asymptotic expansion
can be expressed in terms of rational combinations of (vector valued)
modular forms, which in turn can be explicitly related to classical
modular forms of weight $14$ and $18$.

\smallskip

The rationality of the spectral action for a general
triaxial Bianchi type-IX metric with an $SU(2)$-symmetry, obtained in
\cite{FanFatMar1}, suggested the existence of a rich arithmetic
structure in the Seeley-de Witt coefficients  associated with the
square of the Dirac operator of these cosmological models.
Here the rationality assertion means
that each coefficient is the time integral of an
expression presented by a several variable polynomial with
{\it rational} coefficients evaluated on the expansion factors
and their derivatives, up to a certain order with respect to time.
An earlier rationality result of a similar nature was obtained in
\cite{FatGhoKha} for the Robertson-Walker metrics, proving
a conjecture formulated in \cite{ChaConRW}.

\smallskip

The present article is intended to obtain a deeper understanding of the
arithmetic properties of the spectral action, and to shed light on the role
of the rational coefficients appearing in the expansion of the spectral action
for Bianchi IX metrics.

\smallskip

By imposing the condition of self-duality of the Weyl
tensor and by employing a time-dependent conformal
factor to obtain an Einstein metric from the Bianchi IX
models, an especially interesting family of metrics called
Bianchi IX gravitational instantons have
been explored and well studied in the literature, see for example
\cite{Tod, Oku, Hit, Man2, BabKor} and references therein.
Interestingly, as explained in great detail in the latter,
the differential equations for finding
these metrics reduce to  well understood equations such
as the Halphen system and the Painlev\'e VI equation
with particular rational parameters. In \cite{BabKor},
following the work carried out in \cite{Tod, Hit}, these equations
are solved by using the $\tau$-function of the Schlesinger system
formulated in terms of theta functions  \cite{KitKor}, and an
explicit parametrization of the Bianchi IX gravitational instantons is given
in terms of theta functions with characteristics.
Considered along with our rationality result about
the spectral action \cite{FanFatMar1}, this parametrization of
the gravitational instantons will be the main ingredient in our construction
of the modular expression for the terms appearing in the expansion of the
spectral action in the energy scale. We will also describe an explicit
connection between the modular functions that arise in
the spectral action and well-known classical modular forms.

\smallskip

In the next section, we introduce and clarify our notation, and
we briefly review all the necessary background material on
the spectral action that we need to use throughout the paper.
While most of the material we need to recall is standard, we prefer to
maintain the paper as readable and self-contained as possible.
The reader already familiar with these notions and notations
can skip this section. We start by an explanation of the spectral
action functional \cite{ChaConSAP, ConAction}, which is based on the
Dirac operator and provides a modified Euclidean gravity model.
Implications of this action as a source of new early universe models and
inflationary mechanisms in cosmology has been studied in recent years
\cite{BallMar, KolMar, Mar, MarPieEarly, MarPieTeh2012, MarPieTehCosmic,
NelOchSal, NelSak1, NelSak2, EstMar}. We recall briefly some basic
facts about the Dirac operator, the heat kernel method and pseudo-differential
calculus, and how they can be employed to
compute the terms in the asymptotic expansion of the spectral action
in the energy scale. The terms that appear in the expansion include the
Einstein-Hilbert action and other modified gravity terms such as the Weyl
curvature and Gauss-Bonnet terms. Indeed the latter expressions
appear as the first few terms in the expansion. As a general problem it is
highly desirable to achieve an understanding of the full expansion.
In \cite{FanFatMar1}, we devised an efficient method
for computing the terms appearing in the expansion by using the
Wodzicki noncommutative residue \cite{Wod1, Wod2},
a powerful tool that, in addition to be very important for convenient
calculations, yields an elegant proof of the rationality result presented
in \cite{FanFatMar1}. After briefly recalling, for comparison purposes,
the traditional method of computing these coefficients,
we end Section \ref{PreliminariesSec} by describing
the final formulation of our noncommutative residue method, which prepares
the ground for deriving a variant of the rationality result for the specific family
of metrics that serve as a general form for the Bianchi IX gravitational
instantons.

\smallskip

Section \ref{RationalitySec} is devoted to the explicit computation
of the Dirac operator $\tilde D$ of a general time-dependent
conformal perturbation of the triaxial Bianchi IX metric. We also
prove a rationality statement for its spectral action. The general
form of the Bianchi IX gravitational instantons, which we
mentioned earlier, is
\begin{equation} \label{ConformalBianchiIXMetricEQ}
d\tilde s^2 = F \, ds^2 = F \left (  w_1 w_2 w_3 \, d\mu^2 +
\frac{w_2 w_3}{w_1} \sigma_1^2 +
\frac{w_3 w_1}{w_2} \sigma_2^2+
\frac{w_1 w_2}{w_3} \sigma_3^2 \right ),
\end{equation}
where the conformal factor $F$ is a function of the cosmic time
$\mu$. Here, the metric $ds^2$ is the Bianchi IX model with
general cosmic expansion factors $w_1$, $w_2$ and $w_3$,
and $SU(2)$-invariant 1-forms $\sigma_1$, $\sigma_2$ and
$\sigma_3$. We found it convenient to recall from \cite{FanFatMar1}
the expression for the Dirac operator $D$ of the Bianchi IX
model $ds^2$ and to use it for the presentation of the Dirac operator
$\tilde D$ of the conformally equivalent metric $d \tilde s^2 =$
$F \, ds^2$. We then obtain a rationality statement for the
general terms $\tilde a_{2n}$ appearing in the small time
asymptotic expansion of the heat kernel,
\begin{equation} \label{SmallTimeExpEq}
\text{Trace}\left ( \exp (-t \tilde{D}^2)   \right )
\sim
t^{-2} \sum_{n=0}^\infty \tilde{a}_{2n} t^n \qquad (t \to 0^+).
\end{equation}
The section ends by a presentation of explicit expressions for
$\tilde a_0$ and $\tilde a_2$, while
the lengthy expression for $\tilde a_4$ is recorded in Appendix \ref{fulla_4appendix}.
As general notation, a tilde is used on top of any
symbol that represents an object associated with the conformally
perturbed metric $d \tilde s^2 =$ $F \, ds^2$. We will follow
this notational convention throughout the paper.

\smallskip

In Section \ref{InstantonsSec} we recall briefly another result
that we need to use essentially in the rest of the paper:
the derivation of explicit formulas for the Bianchi IX
gravitational instantons obtained in \cite{Tod, Hit, BabKor}.
One starts by imposing the self-duality condition on the Weyl
tensor of the Bianchi IX model and employing a conformal factor
to obtain an Einstein metric.  These conditions
reduce to the Halphen system and the Painlev\'e VI equation
with particular rational parameters, which can
then be solved in terms of elliptic modular functions and theta
functions. We have written explicitly, in separate subsections, the
parametrization of the solutions in
terms of theta functions with characteristics given in \cite{BabKor}:
one subsection for a two-parametric family with non-vanishing
cosmological constants and the other for a one-parametric family
whose cosmological constants vanish.

\smallskip

One of our main goals is to understand the modular properties of
the Seeley-de Witt coefficients $\tilde a_{2n}$
appearing in \eqref{SmallTimeExpEq} when the solutions of the
gravitational instantons in terms of the theta functions are
substituted in the metric \eqref{ConformalBianchiIXMetricEQ}.
That is, by extending the real time $\mu$ to the right half-plane
$\Re (\mu) > 0$ in the complex plane, to the extent that the
argument of the theta functions that belongs to the upper
half-plane $\mathbb{H}$ can be replaced by $i \mu$, we  explore the
changes in the $\tilde a_{2n}$ under modular transformations
on $\mathbb{H}$. Therefore, in Section \ref{ArithmeticsofInstantonsSec}
we deal with the modular properties of the theta functions and their
derivatives that appear in the terms $\tilde a_0,$ $\tilde a_2$ and
$\tilde a_4$. Using these properties in the explicit expressions
\eqref{a_0Eq}, \eqref{a_2Eq} and the one recorded in Appendix
\ref{fulla_4appendix}, we show by direct calculations in
Section \ref{a_0a_2a_4Sec} that, under the
linear fractional transformations $T_1(i \mu) = i \mu +1$ and $S(i \mu) = i/ \mu$,
the terms $\tilde a_0,$ $\tilde a_2$ and $\tilde a_4$ satisfy interesting
modular properties. What is especially interesting here is that
this behavior reveals modular transformation properties that are encoded in
the parameters of the metric and that are of the same nature as those of
vector-valued modular forms considered in the Eichler--Zagier theory
of Jacobi forms \cite{EicZag}.

\smallskip

It is then natural to expect that these surprising modular transformation
properties obtained by direct calculations for the coefficients $\tilde a_0,$
$\tilde a_2$ and $\tilde a_4$, will continue to hold for all $\tilde a_{2n}$.
Clearly  it is impossible to resort to direct calculations to investigate
similar properties for the general terms. However, the properties seen
in the first few terms, which are reflected in the parameters of the
metric,  indicate that there is a deeper relation between the Dirac
operators associated with the metrics whose parameters transform
to each other. Indeed, in Section \ref{a_2nSec} we study the
corresponding Dirac operators and we prove that they are all isospectral.
By taking advantage of this fact and of uniqueness of the
coefficients in the asymptotic expansion, we prove that indeed
all of the Seeley-de Witt coefficients $\tilde a_{2n}$ of the Bianchi
IX gravitational instantons satisfy modular transformation properties.

\smallskip

 In Section \ref{ModularFormsSec}, we first prove a periodicity
property for the terms $\tilde a_{2n}$ with respect to the parameters
of the two parametric family of Bianchi IX gravitational instantons.
Combining this with their modular transformation properties, we show
that, for rational parameters, each term $\tilde a_{2n}$ defines a
vector-valued modular function of weight $2$ with values in a
finite-dimensional representation of the modular group 
$PSL_2(\mathbb{Z})$. Recall that the modular group  is generated
by the matrices corresponding to the linear fractional transformations
$T_1(i \mu) = i \mu +1$ and $S(i \mu) = i/ \mu$
acting on the upper half-plane, $i \mu \in \mathbb{H}$.
We then observe that, by running a summation over a finite
orbit of the modular transformations on the rational parameters,
each $\tilde a_{2n}$ gives rise to a modular function of weight $2$ with
respect to $PSL_2(\mathbb{Z})$.  
This type of modular functions are sometimes called quasi-modular,
weakly modular or meromorphic modular forms to indicate that they
are allowed to possess poles, which is indeed the case for any modular
function of weight $2$.

\smallskip

In the second part of Section \ref{ModularFormsSec}, we find an intimate 
connection between the modular functions that arise from the Seeley-de Witt 
coefficients $\tilde a_{2n}$ and well-known modular forms. 
 We consider two pairs of
rational parameters that belong to two different general families.
In the first case, we prove that such modular
functions have only  simple poles at infinity, hence multiplication by
the cusp form of weight $12$, $\Delta(q) =$
$q \prod_{n=1}^\infty (1-q^n)^{24}$, $q=\exp(2 \pi i z)$, $z \in \mathbb{H}$, lands them in the
1-dimensional linear space of modular forms of weight $14$,
generated by a single Eisenstein series. In the second case, we show that
the only poles of the resulting modular functions are of order 4 and located at the
point $\rho = e^{2 \pi i/3}$. Moreover, by showing that  they vanish at infinity, 
we prove that multiplication by $G_4^4$, where
$G_4(z)=\sum^*_{m, n \in \mathbb{Z}} 1/(mz+n)^4$, $z \in \mathbb{H}$, is the Eisenstein series of weight 4,
sends each resulting modular function in this case to the 1-dimensional linear space of cusp
forms of weight $18$. This illuminates the intimate connection between
the spectral action for Bianchi IX metrics and well-known modular forms.

\smallskip

Our main results and conclusions are summarized in the last Section.
The appendices contain proofs of some statements and lengthy expressions
that are provided for the sake of completeness.

\smallskip

\section{Dirac operator, spectral action and pseudo-differential calculus}
\label{PreliminariesSec}

\smallskip 

This section is devoted to a short explanation about the 
spectral action functional \cite{ChaConSAP, ConAction}, a brief summary of 
the Dirac operator \cite{LawMic, FriBook}, 
the heat kernel method that employs pseudo-differential calculus 
for computing Seeley-de Witt coefficients \cite{GilBook1}, and an efficient method 
that we devised in \cite{FanFatMar1} for expressing these coefficients as 
noncommutative residues of Laplacians. This method is remarkably 
efficient from a computational point of view and provided 
an elegant proof of the rationality result in \cite{FanFatMar1}, a variant 
of which has more direct relevance to the subject of the present paper,  as 
explained in Section \ref{RationalitySec}.

\smallskip

A noncommutative geometric space is described
by a spectral triple, which consists of an involutive algebra
$\mathcal{A}$ represented by bounded operators on a Hilbert
space $\mathcal{H}$, and an unbounded self-adjoint operator
$D$ on $\mathcal{H}$ \cite{ConBook}. The metric information is encoded
in the operator $D$, which is assumed to
satisfy the main regularity properties of the Dirac operator.
The spectral action \cite{ChaConSAP, ConAction} for 
a spectral triple $(\mathcal{A}, \mathcal{H}, D)$ is an action functional 
that depends on the spectrum of the Dirac operator $D$. It employs a 
cutoff function $f$ defined on the real line to consider  
\[
\text{Trace} \big ( f (D/ \Lambda)\big ),
\] 
where $\Lambda$ is the 
energy scale.  Its asymptotic expansion in the energy scale is usually, 
depending on the nature of the arising poles, 
of the form \cite{ConMarBook}
\begin{equation} \label{SpectActionExpEq}
\text{Trace} \big ( f (D/ \Lambda)\big )
\sim
\sum_{\beta \in \Pi}
f_\beta \Lambda^\beta \int\!\!\!\!\!\!- \, |D|^{-\beta}+ f(0) \zeta_D(0) + \cdots. 
\end{equation}
The summation in the latter runs over the points $\beta$ where 
the poles of the spectral zeta function of the Dirac operator 
$D$, $\zeta_D(s)$, and the poles of other associated zeta functions are 
located.  The set of such points is called the {\it  dimension spectrum}  
of the spectral triple \cite{ConMosLocal}. For classical manifolds, the 
poles of the spectral zeta functions are located at certain points 
on the real line. However, in general, the dimension spectrum of a 
noncommutative geometric space can contain points in the complex 
plane that do not belong to the real line.  See for instance \cite{LapFra} 
for examples of geometric zeta functions whose poles are not necessarily 
located on the real line.

\smallskip

The main commutative example of a spectral triple
is given by a spin$^c$ manifold $M$ \cite{ConReconstruct}, where
the algebra of smooth functions $\mathcal{A}=C^\infty(M)$
acts on the $L^2$-spinors of $M$, and $D$ is
the Dirac operator. In this case,  the coefficients in
the expansion of the spectral action are determined
by the Seeley-de Witt coefficients associated with $D^2$, which
are local invariants of the geometry \cite{GilBook1}. That is, 
up to considerations that merely rely on momenta of the cutoff 
function $f$, the terms of the expansion \eqref{SpectActionExpEq} are 
determined by the coefficients appearing in a small time asymptotic 
expansion of the form 
\begin{equation} \label{HeatExpEq}
\text{Trace}\left ( \exp  ( -t D^2  )\right )   
\sim 
t^{- \text{dim} (M)/2} \sum_{n=0}^\infty a_{2n}(D^2)\, t^n \qquad ({t \to 0^+}). 
\end{equation} 
Chapter 1 of the book \cite{ConMarBook} contains a detailed mathematical discussion 
of the asymptotic expansions related to the spectral action.

\smallskip

\subsection{Dirac operator and its pseudo-differential symbol}
\label{DiracOpSubSec}

Given a spin bundle $S$ and a spin connection $\nabla^S$ on a Riemannian 
manifold  $M$ of dimension $m$, the Dirac operator $D$ acting on smooth sections 
of $S$ is defined by the following composition. It is given by 
composing the three maps 
\[
C^\infty(S) \xrightarrow{\nabla^S} C^\infty(T^*M \otimes S) 
\xrightarrow{\#} C^\infty(TM \otimes S) 
\xrightarrow{c} C^\infty(S),  
\]
where the second arrow is essentially the musical isomorphism $\#$
identifying the cotangent and tangent bundles, and the 
third arrow is obtained by considering the Clifford action of 
the tangent bundle $TM$ on $S$.  The spin connection $\nabla^S$ 
is a connection on the $Spin(m)$-principal bundle associated with $S$, which is obtained by 
lifting the Levi-Civita connection $\nabla$ seen as a connection on the 
$SO(m)$-principal bundle of $TM$. Thus, finding an explicit local formula 
for the Levi-Civita connection is the first step in calculating the Dirac 
operator.

\smallskip

The Levi-Civita connection $\nabla: C^\infty(TM) \to C^\infty(T^*M \otimes TM)$ 
is the unique connection on the tangent bundle that is compatible with the 
metric and torsion-free. In fact, these conditions characterize this connection 
uniquely and it can be computed by the method that will be explained shortly. 
Let us first explain that compatibility with the metric $g$ means that 
\[
g(\nabla_X Y, Z) + g(Y, \nabla_X Z) = X\cdot g(Y, Z), \qquad X, Y, Z \in C^\infty(TM).  
\] 
Moreover, torsion-freeness refers to vanishing of the torsion tensor 
\[
T(X, Y) = \nabla_X Y - \nabla_Y X - [X, Y], \qquad X, Y \in C^\infty(TM). 
\]

\smallskip

Since it is more convenient to work with coframes rather than with frames, 
for explicit calculations it is useful to take advantage of the identification of 
the tangent and cotangent bundles via the metric and write the 
Levi-Civita connection as a map from $C^\infty(T^*M)$ to $C^\infty(T^*M \otimes T^*M)$. 
Then, if $\{\theta^a\}$ is a local orthonormal coframe, i.e. an orthonormal basis for the 
local sections of $T^*M$ in a local chart $U$ where it is trivialized, one can 
write 
\[
\nabla \theta^a =
\sum_b \omega^a_b \otimes \theta^b,
\]
where $\omega^a_b$ are local differential 1-forms. In this basis $\nabla$ can be 
expressed as  
\[
\nabla = d + \omega, 
\]
where $d$ is the de Rham differential and $\omega = (\omega^a_b)$ is a matrix 
of 1-forms. In this picture, the compatibility with the metric and the torsion-freeness 
respectively read as  
\[
\omega^a_b =
- \omega^b_a, \qquad d \theta^a =
\sum_b \omega^a_b \wedge \theta^b.
\]
The latter conditions yield a unique solution for the 1-forms $\omega^a_b$ 
and thereby one achieves an explicit calculation of $\nabla$.

\smallskip

Having computed $\nabla$, one can lift it to the spin connection as follows. 
The spin group $Spin(m)$ is a double cover of the special orthogonal group 
$SO(m)$ and there is 
an explicit isomorphism $\mu:\mathfrak{so}(m)\to \mathfrak{spin}(m)$ 
identifying their Lie algebras, which is given by (see for example Lemma 4.8 in 
\cite{RoeBook}), 
\[
\mu(A)= \frac{1}{4}\sum_{a,b} A^{a b} e_a e_b,
\qquad A = (A^{ab})\in \mathfrak{so}(m). 
\]
In the latter $\{ e_a \}$ is the standard basis of $\mathbb{R}^m$ viewed inside 
the corresponding Clifford algebra, in which the 
linear span of the elements $\{ e_a e_b; a < b \}$ is the Lie algebra 
$\mathfrak{spin}(m)$ of $Spin(m)$. Combining this with uniqueness 
of the spin representation for the Clifford algebra, one can choose 
$k \times k$ matrices, $k = \textnormal{rk}(S)$, which satisfy the 
relations $(\gamma^a)^2 = - I$ and 
$\gamma^a \gamma^b + \gamma^b \gamma^a = 0$ for $a \neq b$, 
and calculate the matrix of 1-forms $\omega^S$ representing the 
spin connection
\[
\nabla^S = d + \omega^S.
\]
That is, one can write 
\[
\omega^S = \frac{1}{4}\sum_{a,b} \omega^{a}_{b} \gamma^a \gamma^b.  
\]

\smallskip

Now we write the Dirac operator explicitly: it follows from its definition 
and the above explicit formulas for the ingredients of the definition that 
\begin{eqnarray} \label{GeneralDiracOpFormula}
D &=& \sum_a \theta^a \nabla^S_{\theta_a} \nonumber \\
&=& \sum_{a, \nu} \gamma^a dx^{\nu}(\theta_a) \frac{\partial}{\partial x^{\nu}} +
\frac{1}{4} \sum_{a, b, c} \gamma^c \omega^b_{ac} \gamma^a \gamma^b,
\end{eqnarray} 
where $\{ \theta_a \}$ is a pre-dual of the coframe $\{ \theta^a \}$ and the 
$\omega^b_{ac}$ are defined by
\[
\omega^b_a = \sum_c \omega^b_{ac} \theta^c. 
\]

\smallskip

It is clear that $D$ is a differential operator of order 1. Like any differential operator, or 
more generally, like any pseudo-differential operator,  using the Fourier 
transform and the Fourier inversion formula, $D$ can be expressed 
by its pseudo-differential symbol denoted by $\sigma(D)$. The symbol $\sigma(D)$ 
is defined locally from $U \times \mathbb{R}^m$ to $M_k(\mathbb{C})$ and allows 
one to write the action of $D$ on a local section $s$ as 
\begin{eqnarray} \label{pseudodifferentialOp}
D s (x)
&=&
(2 \pi)^{-m/2} \int e^{i x \cdot \xi} \, \sigma(D)(x, \xi) \, \hat s (\xi) \, d\xi  \nonumber \\
&=& (2 \pi)^{-m} \int \int e^{i (x-y) \cdot \xi} \, \sigma(D)(x, \xi) \, s (y) \, dy\, d\xi,
\end{eqnarray}
where $\hat s$ is Fourier transform of $s$ taken component-wisely. Note that 
the endomorphisms of the bundle are locally identified with $M_k(\mathbb{C})$. 
The 
expression given by \eqref{GeneralDiracOpFormula} makes it clear that 
\begin{equation} \label{SymbolofDiracEq}
\sigma(D)(x, \xi)=
\sum_{a, \nu} \gamma^a dx^{\nu}(\theta_a) (i \xi_{\nu+1} )+
\frac{1}{4} \sum_{a, b, c} \gamma^c \omega^b_{ac} \gamma^a \gamma^b,
\end{equation}
for  $x = (x^0, x^1, \dots, x^{m-1}) \in U$ and 
$\xi = (\xi_1, \xi_2, \dots, \xi_m) \in \mathbb{R}^m$. 
For our purpose of studying the Seeley-de Witt coefficients appearing in the asymptotic 
expansion \eqref{HeatExpEq}, we need to have the pseudo-differential symbol of $D^2$. This 
can be achieved by finding an explicit formula for $D^2$ from \eqref{GeneralDiracOpFormula}, 
or more easily by using the composition rule for pseudo-differential symbols.

\smallskip

In fact, 
pseudo-differential operators are closed under composition and there is an 
explicit and handy formula, which describes the symbol of the product 
of two such operators modulo infinitely smoothing operators. That is, if 
the operators $P_1$ and $P_2$ are associated with the symbols $\sigma(P_1)$ and 
$\sigma(P_2)$, 
\begin{eqnarray*}
P_j s (x) 
&=& 
(2 \pi)^{-m/2} \int e^{i x \cdot \xi} \, \sigma(P_j)(x, \xi) \, \hat s (\xi) \, d\xi \\
&=& (2 \pi)^{-m} \int \int e^{i (x-y) \cdot \xi} \, \sigma(P_j)(x, \xi) \, s (y) \, dy\, d\xi,  
\end{eqnarray*}
then, the symbol of $P_1 P_2$ has the following asymptotic expansion: 
\begin{equation} \label{SymbolCompositionRule}
\sigma(P_1 P_2) \sim 
\sum_{\alpha \in \mathbb{Z}_{\geq 0}^m} \frac{(-i)^{|\alpha|} }{\alpha !} 
\partial_\xi^\alpha  \sigma(P_1)(x, \xi) \, \partial_x^\alpha \sigma(P_2)(x, \xi). 
\end{equation}
One can find precise technical discussions in Chapter 1 of the book \cite{GilBook1} 
about the pseudo-differential calculus that we use. 
In the case of differential operators, since the symbols are polynomials in $\xi$ 
with coefficients that are matrix-valued functions defined on a local chart $U$, the 
formula \eqref{SymbolCompositionRule} gives a precise formula for the symbol of the 
composition. Returning to the Dirac operator $D$, one can use its symbol given 
by \eqref{SymbolofDiracEq} and the composition formula \eqref{SymbolCompositionRule} 
to calculate the symbol of $D^2$:  
\begin{eqnarray*} \label{specialcomposition}
\sigma(D^2)(x, \xi) =  
\sum_{\alpha \in \mathbb{Z}_{\geq 0}^m} \frac{(-i)^{|\alpha|} }{\alpha !} 
\partial_\xi^\alpha  \sigma(D)(x, \xi) \, \partial_x^\alpha \sigma(D)(x, \xi). 
\end{eqnarray*}
The latter is in general a polynomial of order 2 in $\xi$ whose coefficients 
are matrix-valued functions defined on the local chart $U$. In Section \ref{RationalitySec}, 
explicit formulas are presented for the Dirac operators of the cosmological models that we 
are interested in.

\smallskip

\subsection{Heat expansion using pseudo-differential calculus}
\label{HeatExpSubSec}

Let $D$ be the Dirac operator on a compact spin manifold of dimension $m$, 
as we described. For any $t >0$, the operator $\exp (- t D^2) $ is an infinitely smoothing 
operator and thus can be represented by a smooth kernel. In particular 
it is a trace-class operator and as $t \to 0^+$, the trace of $\exp (- t D^2) $ 
goes to infinity. However, it is quite remarkable that there is an 
asymptotic expansion with geometric coefficients that describe the rate 
 of this divergence. That is, as $t \to 0^+$, 
\begin{equation} \label{AsymptoticExpansion}
\text{Trace}\big( \exp(-t D^2)   \big )
\sim
t^{-m/2} \sum_{n=0}^\infty a_{2n}(D^2)\, t^n, 
\end{equation}
where the coefficients $a_{2n}(D^2)$ are local invariants of the metric that 
encode geometric information obtained from the curvature tensor. In fact, 
typically they are integrals of certain 
expressions obtained from the 
Riemann curvature tensor and its covariant derivatives and contractions.

\smallskip

It is evident that the left hand side of \eqref{AsymptoticExpansion} depends 
only on the eigenvalues of $D^2$. Since, except in rare cases, in general the 
eigenvalues of the Dirac operator are not known, it is significant that there 
are methods in the literature that allow one to express the coefficients $a_{2n}(D^2)$ 
appearing on the right hand side of \eqref{AsymptoticExpansion} as integrals 
of local expressions obtained from the metric. One of these methods, which is 
quite effective, starts with the Cauchy integral formula and employs parametric 
pseudo-differential calculus to approximate the kernel $\exp(- t D^2)$ and 
thereby accomplishes a recursive procedure for finding formulas for the 
coefficients $a_{2n}(D^2)$.

\smallskip

Let us review this method briefly from Chapter 1 of the book \cite{GilBook1}. The Dirac operator $D$ 
is a self-adjoint unbounded operator, with respect to which the Hilbert space of 
$L^2$-spinors admits a spectral decomposition. Invoking the Cauchy integral formula one can write 
\begin{equation} \label{CauchyIntegral}
\exp(-t D^2) = \frac{1}{2 \pi i} \int_\gamma e^{-t \lambda} (D^2 - \lambda )^{-1} \, d\lambda,
\end{equation}
where the integration is over a contour  $\gamma$ in the complex plane that goes 
around the non-negative real numbers clockwise. Since $D^2 - \lambda$ is an 
elliptic differential operator, it admits a parametrix which is the same as $(D^2 - \lambda)^{-1}$ modulo an infinitely smoothing operator. Thus, the approximation of $(D^2 - \lambda)^{-1}$ amounts to finding or approximating the parametrix $R_\lambda$ of $D^2-\lambda$.  
This is achieved by exploiting the calculus of pseudo-differential symbols given by 
the composition rule \eqref{SymbolCompositionRule}. In order to compute the symbol of $R_\lambda$,
since $D^2 - \lambda$ is of order 2, the leading symbol of $R_\lambda$ has to be of order $-2$ and one can write 
\[
\sigma(R_\lambda) \sim \sum_{j=0}^\infty r_j (x, \xi, \lambda),
\]
where each $r_j (x, \xi, \lambda)$ is a parametric pseudo-differential
symbol of order $-2 - j$. There is an important nuance that 
$\lambda$ should be treated of order 2.  Also, for the parametric pseudo-differential 
symbols depending on the complex parameter $\lambda$, being homogeneous of order 
$-2-j$ means that for any $t>0$, 
\[
r_j(x, t \xi, t^2 \lambda) = t^{-2-j} r_j(x,  \xi,  \lambda).
\]

\smallskip

As we discussed before, the square of the Dirac operator $D^2$ is a 
differential operator of order 2 and therefore for the symbol of $D^2 - \lambda$  
we have 
\[
\sigma(D^2 - \lambda) =  \big ( p_2(x, \xi) - \lambda \big ) + p_1(x, \xi) + p_0(x, \xi), 
\]
where each $p_k(x, \xi)$ is a polynomial in $\xi$ whose coefficients are 
matrix-valued functions defined on the local chart. 
By passing to the symbols and using  the composition rule \eqref{SymbolCompositionRule}, 
the solution of the equation 
\begin{equation} \label{ParaSymbolicEqEq}
\sigma \left ( R_\lambda (D^2 - \lambda) \right ) \sim I 
\end{equation}
yields the following recursive solution of the terms $r_j$ in the expansion 
of the symbol of the parametrix $R_\lambda$. In fact, by a comparison of 
homogeneous terms on the two sides of \eqref{ParaSymbolicEqEq}, one finds that  
\begin{equation} \label{r_0formula}
r_0(x, \xi, \lambda) = (p_2(x, \xi) - \lambda)^{-1}, 
\end{equation}
and for any $n \geq 1$, 
\begin{eqnarray} \label{r_nformula}
r_n (x, \xi, \lambda)=
- \left ( \sum \frac{(-i)^{|\alpha|}}{\alpha!}  \partial_\xi^\alpha r_j(x, \xi, \lambda) \, \partial_x^\alpha  p_k(x, \xi) \right ) r_0(x, \xi, \lambda),
\end{eqnarray}
where the summations are over all 4-tuples of non-negative integers 
$\alpha$, and integers $0 \leq j < n$ and $0 \leq k  \leq 2$ 
such that $|\alpha|+j +2 -k =n$.

\smallskip

Using this recursive procedure, one can choose a large enough $N$ 
so that the operator corresponding to the symbol $r_0+\cdots +r_N$ gives 
a desired approximation of the parametrix $R_\lambda$ of $D^2- \lambda$. 
By substituting the approximation in the Cauchy integral formula \eqref{CauchyIntegral}, one 
obtains an approximation of the kernel of $\exp(-tD^2)$. By integrating 
the approximation of the kernel over the diagonal of $M \times M$ against 
the volume form, one can derive the small time asymptotic expansion \eqref{AsymptoticExpansion}.

\smallskip

It is remarkable that this method shows instructively that each coefficient 
in the expansion is  given by the integral,  
\begin{equation} \label{a_2nformula}
a_{2n}(D^2) = \int_M a_{2n}(x, D^2) \, dvol_g(x),
\end{equation}
where $a_{2n}(x, D^2)$ is an invariantly defined function on the manifold 
defined in the local chart by 
\begin{equation}  \label{DensitiesFormula}
a_{2n}(x, D^2)=
\frac{(2 \pi)^{-m}}{2 \pi i} \int_{\mathbb{R}^m} \int_{\gamma} e^{-\lambda} \, \textnormal{tr}
\big ( r_{2n}(x, \xi, \lambda) \big )  \, d \lambda \, d^m \xi.
\end{equation}

\smallskip

The analysis involved for deriving the above expansion and formula 
for the coefficient $a_{2n}(D^2)$ is quite intricate and as we mentioned 
earlier, it is presented 
in detail in Chapter 1 of the book \cite{GilBook1}. It should be stressed that the integrals 
involved in the expression 
for $a_{2n}(D^2)$ are possible to work out since one can show by 
induction from the formulas \eqref{r_0formula} and \eqref{r_nformula} that 
\[
\textnormal{tr} \big ( r_{n}(x, \xi, \lambda) \big ) = 
\sum_{ \substack{n=2j - |\alpha|-2 \\ |\alpha| \leq 3n}} r_{n, j , \alpha}(x)  \, \xi^\alpha \, \textnormal{tr}\left (r_0(x, \xi, \lambda)^j \right ). 
\] 
Therefore, using the method reviewed in this subsection, one can 
calculate the $a_{2n}(D^2)$ explicitly in concrete examples. However, it should 
be noted that this method involves heavy calculations that are cumbersome 
even with computer assistance. In Subsection \ref{WodzickiRes}, 
we review an efficient method that we devised in \cite{FanFatMar1} for computing 
the Seeley-de With coefficients $a_{2n}(D^2)$ by making use of 
Wodzicki's noncommutative residue \cite{Wod1, Wod2} which in general is the unique trace 
functional  on the algebra of classical pseudo-differential 
on a vector bundle on $M$. 
This method is significantly more convenient from a computational point 
of view and  yields elegant proofs for rationality statements of the type 
discussed in Section \ref{RationalitySec}.

\smallskip

\subsection{Calculation of heat coefficients using the noncommutative residue}
\label{WodzickiRes}

The symbol of a classical pseudo-differential operator of order 
$d$ acting on  the smooth sections of a vector bundle on an 
$m$-dimensional  
manifold $M$ admits by definition an expansion of the following form 
as $\xi \to \infty$: 
\begin{equation} \label{classicalsymbol}
\sigma (x, \xi)
 \sim
\sum_{j=0}^\infty \sigma_{d-j} (x, \xi),
\end{equation}
where each 
$\sigma_{d-j} : U \times \left ( \mathbb{R}^m \setminus \{ 0\}\right ) 
\to M_r(\mathbb{C})$ is positively homogeneous of order
$d-j$ in $\xi$. That is, identifying the endomorphisms of the vector bundle 
on a local chart $U$ on $M$ with $M_r(\mathbb{C})$, each $\sigma_{d-j} : 
U \times \left ( \mathbb{R}^m \setminus \{ 0\}\right ) \to M_r(\mathbb{C})$ 
is a smooth map such that 
\[
\sigma_{d-j}(x, t\xi) = t^{d-j} \sigma_{d-j}(x, \xi), 
\qquad  (x, \xi) \in U \times \left ( \mathbb{R}^m 
\setminus \{ 0\}\right ), \qquad 
t > 0.
\]

\smallskip

The noncommutative residue \cite{Wod1, Wod2} of the pseudo-differential operator 
$P_\sigma$ associated with a classical symbol $\sigma$ of the 
above type is defined by
\begin{equation} \label{WodResidueDefEq}
\textrm{Res}(P_\sigma) =
\int_{S^*M} \textrm{tr} \big (\sigma_{-m}(x, \xi) \big ) \,  d^{m-1}\xi \, d^mx. 
\end{equation}
Some explanations are in order for the latter. First, note that 
$m$ is the dimension of the manifold, which shows that the 
noncommutative residue vanishes on the classical operators 
of order less than $-m = - \textnormal{dim}(M)$. In particular, 
it vanishes on infinitely smoothing operators, which allows one 
to view the noncommutative residue as a linear functional defined 
on the space of classical symbols with the following  rule 
for composing to classical symbols $\sigma_1$ and $\sigma_2$, inherited 
from \eqref{SymbolCompositionRule}:
\[ \sigma_1 \circ \sigma_2
\sim 
\sum_{\alpha \in \mathbb{Z}_{\geq 0}^m} \frac{(-i)^{|\alpha|} }{\alpha !} 
\partial_\xi^\alpha  \sigma_1(x, \xi) \, \partial_x^\alpha \sigma_2(x, \xi).
\]
Second,  $S^*M = \{ (x, \xi) \in T^*M; ||\xi||_g=1 \}$ is the 
cosphere bundle of $M$ and the formula \eqref{WodResidueDefEq} is the
integral over $M$ of a $1$-density associated with the classical symbol $\sigma$, 
which is called the  {\it Wodzicki residue density.}  In each
cotangent fibre $\mathbb{R}^m \cong T_x^* M $, where $\xi$ belongs to, 
 consider the volume form on the unit sphere $|\xi|=1$,  
\[
\sigma_\xi = \sum_{j=1}^m (-1)^{j-1} \xi_j \, d\xi_1
\wedge \cdots \wedge {\widehat d \xi_j} \wedge \cdots \wedge d \xi_m.
\]
Then, the mentioned 1-density associated with the classical
symbol $\sigma$ with the expansion \eqref{classicalsymbol} is 
defined by
\[
\textnormal{wres}_x P_\sigma
=
\left ( \int_{|\xi |=1}
\textrm{tr} \left (\sigma_{-m}(x, \xi)  \right ) |\sigma_\xi |
\right ) |dx^0 \wedge dx^1 \wedge \cdots \wedge dx^{m-1}|.
\]

\smallskip

Extensive discussions and alternative formulations of the 
noncommutative residue,  which was first discovered for 1-dimensional 
symbols on the circle \cite{Adl, Man1},  
are given in \cite{Wod1, Wod2, Kassel} and
Section 7.3 of the book \cite{GraVarFig}. The spectral formulation 
of this residue plays an important role in noncommutative geometry 
since it is used in the local index formula for spectral triples, developed in 
\cite{ConMosLocal}. Also, formulas similar to the one given 
by \eqref{WodResidueDefEq} are used in \cite{FatWon, FatKha} to 
define noncommutative residues for the restriction of 
Connes' pseudo-differential calculus \cite{ConCstarDiffGeo} to 
noncommutative tori, which are handy computational tools, see 
for example \cite{Fat} for an application.

\smallskip

Now let $D$ be the Dirac operator on a 4-dimensional
manifold.  We are interested in 
computing and understanding the nature of the Seeley-de Witt 
coefficients $a_{2n}(D^2)$ appearing in an asymptotic 
expansion of the form \eqref{AsymptoticExpansion} with $m=4$, 
when $D$ is of this type, namely 
the Dirac operator of a geometry describing a cosmological model.  
Using an alternative spectral formulation of the noncommutative 
residue and the K\"unneth formula, we showed in \cite{FanFatMar1} that 
for any  integer $n\geq 1$,
\[
a_{2n}(D^2) = \frac{1}{32\, \pi^{n+3}} \textnormal{Res}(\Delta^{-1}),
\]
where $\Delta^{-1}$ denotes the parametrix of the elliptic 
differential operator 
\[
\Delta = D^2 \otimes 1 + 1 \otimes \Delta_{\mathbb{T}^{2n-2}},
\]
in which $\Delta_{\mathbb{T}^{2n-2}}$ is the flat Laplacian on
the $(2n-2)$-dimensional torus $\mathbb{T}^{2n-2} = 
\left ( \mathbb{R}/\mathbb{Z} \right )^{2n-2}$. 
Evidently $\Delta$ 
acts on the smooth sections of the tensor product of the spin bundle of $M$ and 
the 1-dimensional trivial bundle on the torus. Using the symbol of $\Delta$, which 
is intimately related to the symbol of $D^2$, we showed in Corollary 4.1 of 
\cite{FanFatMar1} that 
\begin{equation} \label{HeatCoefsResEq}
a_{2n}(D^2) 
= 
\frac{1}{32 \pi^{n+3}} \int_{S^*(M \times \mathbb{T}^{2n-2})} 
\textnormal{tr} \left ( \sigma_{-2n-2}(\Delta^{-1}) \right ) \, d^{2n+1} \xi' \, d^{4}x, 
\end{equation}
where in the local chart $U$, $ \sigma_{-2n-2}(\Delta^{-1}): (U \times \mathbb{T}^{2n-2}) 
\times \mathbb{R}^{2n+2} \to M_4(\mathbb{C})$ is the homogeneous 
component of order $-2n-2$ in the asymptotic expansion of the symbol of the 
parametrix $\Delta^{-1}$ 
of $\Delta= D^2 \otimes 1 + 1 \otimes \Delta_{\mathbb{T}^{2n-2}}. $ In the proof of the 
corollary, we explained in detail the calculation of a recursive formula for 
$\sigma_{-2n-2}(\Delta^{-1})$, and stressed that it has no dependence on coordinates 
of the torus, which is indicated in the formula \eqref{HeatCoefsResEq}.

\smallskip

As we mentioned 
in the beginning of this section, this method is used crucially  
for proving the rationality statements for Bianchi IX metrics, which are 
elaborated on in Section \ref{RationalitySec}. 

\smallskip

\section{Rationality of spectral action for Bianchi IX metrics}
\label{RationalitySec}

\smallskip

The goal of this section is to present a variant of the rationality 
result proved in \cite{FanFatMar1}.  That is, we consider a 
time dependent conformal perturbation of the triaxial Bianchi IX 
metric treated in \cite{FanFatMar1}, and show that the terms appearing 
in the expansion of its spectral action in the energy scale are 
expressed by several variable polynomials with {\it rational} 
coefficients, evaluated on the expansion factors, the conformal 
factor and their time derivatives up to a certain order. The reason 
for considering this family of metrics is that they serve as a 
general form for the Bianchi IX gravitational instantons \cite{Tod, Hit, BabKor}. 
Indeed, combined with the parametrization of the latter in terms of theta 
functions with characteristics \cite{BabKor}, the rationality result stimulates the construction 
of modular functions from the  spectral action, which is carried out in the 
sequel.

\smallskip

For convenience, in passing, we recall the formalism and explicit calculation of the 
Dirac operator of the Bianchi IX metrics from \cite{FanFatMar1}, which can then be used 
for the presentation of the Dirac operator of the conformally perturbed metric, given by  
\eqref{ConformalBianchiIXMetricEQ}.

\smallskip

\subsection{Triaxial Bianchi IX metrics}

Euclidean Bianchi IX metrics are of the form
\begin{equation} \label{BianchimetricEq}
ds^2 = w_1 w_2 w_3 \, d\mu^2 +
\frac{w_2 w_3}{w_1} \sigma_1^2 +
\frac{w_3 w_1}{w_2} \sigma_2^2+
\frac{w_1 w_2}{w_3} \sigma_3^2,
\end{equation}
where the cosmic expansion factors $w_i$ are functions
of the cosmic time $\mu$. The $\sigma_i$ are left-invariant
1-forms on $SU(2)$-orbits satisfying
\[
d \sigma_1 = \sigma_2 \wedge \sigma_3, \qquad
d \sigma_2 = \sigma_3 \wedge \sigma_1, \qquad
d \sigma_3 = \sigma_1 \wedge \sigma_2.
\]

\smallskip

In order to write this metric explicitly in a local chart, in \cite{FanFatMar1}, 
we parametrized the 3-dimensional sphere $\mathbb{S}^3$ by 
\[
( \eta, \phi, \psi) 
\mapsto 
\left ( \cos(\eta/2) e^{i (\phi+\psi)/2},   \sin(\eta/2) e^{i (\phi-\psi)/2}  \right ), 
\] 
with the parameter ranges $0 \leq \eta \leq \pi, 0 \leq \phi < 2 \pi, 0 \leq \psi < 4 \pi$. 
We then wrote the metric \eqref{BianchimetricEq} in the local coordinates 
$x = (\mu, \eta, \phi, \psi)$, the expression of 
which was found to be 
\begin{eqnarray}
ds^2 &=& w_1 w_2 w_3 \, d\mu \,d\mu+\frac{w_1 w_2 \cos (\eta )}{w_3}d\phi \,d\psi 
   +\frac{w_1 w_2 \cos (\eta )}{w_3} d\psi \,d\phi \nonumber \\ 
 &   +&\left(\frac{w_2 w_3 \sin ^2(\eta )
   \cos ^2(\psi )}{w_1}+w_1 \left(\frac{w_3 \sin ^2(\eta ) \sin ^2(\psi )}{w_2}+\frac{w_2
   \cos ^2(\eta )}{w_3}\right)\right) d\phi \,d\phi \nonumber \\ 
&  +& \frac{\left(w_1^2-w_2^2\right) w_3
   \sin (\eta ) \sin (\psi ) \cos (\psi )}{w_1 w_2} d\eta \,d\phi \nonumber \\
 &+& \frac{\left(w_1^2-w_2^2\right) w_3 \sin (\eta ) \sin (\psi ) \cos (\psi )}{w_1
   w_2}d\phi \,d\eta  \nonumber \\ 
 &  +&\left(\frac{w_2 w_3 \sin ^2(\psi )}{w_1}  +\frac{w_1 w_3 \cos
   ^2(\psi )}{w_2}\right)d\eta \,d\eta +\frac{w_1 w_2}{w_3}d\psi \,d\psi. \nonumber
\end{eqnarray}

\smallskip

Going through the definitions and the construction reviewed in Subsection \ref{DiracOpSubSec}, 
the Dirac operator of this metric was then explicitly calculated: 
\begin{eqnarray} \label{DiracBianchiEq}
D&=&-\sqrt{\frac{w_{1}}{w_{2}w_{3}}}\cot\eta\cos\psi\cdot\gamma^{1}\frac{\partial}{\partial\psi}+\sqrt{\frac{w_{1}}{w_{2}w_{3}}}\csc\eta\cos\psi\cdot\gamma^{1}\frac{\partial}{\partial\phi} \nonumber \\
&&-\sqrt{\frac{w_{1}}{w_{2}w_{3}}}\sin\psi\cdot\gamma^{1}\frac{\partial}{\partial\eta}
-\sqrt{\frac{w_{2}}{w_{1}w_{3}}}\cot\eta\sin\psi\cdot\gamma^{2}\frac{\partial}{\partial\psi}\nonumber \\
&&+\sqrt{\frac{w_{2}}{w_{1}w_{3}}}\csc\eta\sin\psi\cdot\gamma^{2}\frac{\partial}{\partial\phi}+\sqrt{\frac{w_{2}}{w_{1}w_{3}}}\cos\psi\cdot\gamma^{2}\frac{\partial}{\partial\eta}\nonumber \\
&&+\frac{1}{\sqrt{w_{1}w_{2}w_{3}}}\gamma^{0}\frac{\partial}{\partial\mu}+\sqrt{\frac{w_{3}}{w_{1}w_{2}}}\gamma^{3}\frac{\partial}{\partial\psi}+\frac{1}{4\sqrt{w_{1}w_{2}w_{3}}}\left(\frac{w_{1}^{'}}{w_{1}}+\frac{w_{2}^{'}}{w_{2}}+\frac{w_{3}^{'}}{w_{3}}\right)\gamma^{0} \nonumber \\
&&-\frac{\sqrt{w_{1}w_{2}w_{3}}}{4}\left(\frac{1}{w_{1}^{2}}+\frac{1}{w_{2}^{2}}+\frac{1}{w_{3}^{2}}\right)\gamma^{1}\gamma^{2}\gamma^{3}. 
\end{eqnarray}
The pseudo-differential symbol associated with the latter has the following expression: 
\begin{eqnarray*}
\sigma(D)(x, \xi) 
&=&-\frac{i  \gamma^1
   \sqrt{w_1} \left(\csc (\eta )
   \cos (\psi ) \left(\xi _4 \cos (\eta
   )-\xi _3\right)+\xi _2 \sin (\psi
   )\right)}{\sqrt{w_2}
   \sqrt{w_3}} \\ 
&&+\frac{i 
   \gamma^2 \sqrt{w_2} \left(\sin
   (\psi ) \left(\xi _3 \csc (\eta )-\xi _4
   \cot (\eta )\right) +\xi _2 \cos (\psi
   )\right)}{\sqrt{w_1}
   \sqrt{w_3}} \\ 
&&+\frac{i 
   \gamma^1 \xi _1}{\sqrt{w_1}
   \sqrt{w_2}
   \sqrt{w_3}}+\frac{i 
   \gamma^3 \xi _4
   \sqrt{w_3}}{\sqrt{w_1}
   \sqrt{w_2}} + 
\frac{1}{4\sqrt{w_{1}w_{2}w_{3}}}\left(\frac{w_{1}^{'}}{w_{1}}
+\frac{w_{2}^{'}}{w_{2}}+\frac{w_{3}^{'}}{w_{3}}\right)\gamma^{0} \\
&&-\frac{\sqrt{w_{1}w_{2}w_{3}}}{4}\left(\frac{1}{w_{1}^{2}}+\frac{1}{w_{2}^{2}}
+\frac{1}{w_{3}^{2}}\right)
\gamma^{1}\gamma^{2} \gamma^3. 
\end{eqnarray*}
We also note that in our calculations we use the following gamma matrices 
$\gamma^0,$ $\gamma^1$, $\gamma^2$ and $\gamma^3$, which are 
respectively written as 
{\small 
\[ 
\left(
 \begin{array}{cccc}
 0 & 0 & i & 0 \\
 0 & 0 & 0 & i \\
 i & 0 & 0 & 0 \\
 0 & i & 0 & 0
\end{array}
\right), 
\left(
\begin{array}{cccc}
 0 & 0 & 0 & 1 \\
 0 & 0 & 1 & 0 \\
 0 & -1 & 0 & 0 \\
 -1 & 0 & 0 & 0
\end{array}
\right),  
\left(
\begin{array}{cccc}
 0 & 0 & 0 & -i \\
 0 & 0 & i & 0 \\
 0 & i & 0 & 0 \\
 -i & 0 & 0 & 0
\end{array}
\right), 
\left(
\begin{array}{cccc}
 0 & 0 & 1 & 0 \\
 0 & 0 & 0 & -1 \\
 -1 & 0 & 0 & 0 \\
 0 & 1 & 0 & 0
\end{array}
\right). 
\]
}

\smallskip

Using the methods reviewed in Subsection \ref{HeatExpSubSec} 
and Subsection \ref{WodzickiRes}, the terms
$a_0,$ $a_2$ and $a_4$ in the expansion of the spectral action 
for the metric \eqref{BianchimetricEq} were computed. Moreover, 
the main result of \cite{FanFatMar1} is that a general term $a_{2n}$ 
in the expansion, modulo an integration with respect to $\mu$, is 
of the form 
\[
a_{2n}
=
(w_1w_2w_3)^{1-3n}Q_{2n}
\left(
w_1, w_2, w_3, w_1', w_2', w_3', \dots, w_1^{(2n)},  w_2^{(2n)},  w_3^{(2n)}
\right ),
\]
where $Q_{2n}$ is a polynomial of several variables with
rational coefficients.
The rationality result was proved by exploiting the
$SU(2)$ invariance of the 1-forms $\sigma_i$ appearing in \eqref{BianchimetricEq} and by
making use of the method reviewed in Subsection \ref{WodzickiRes}.

\smallskip

\subsection{Time dependent conformal perturbations 
of Bianchi IX metrics}

By making a correct choice of a conformal factor, an especially
interesting family of Bianchi IX metrics called Bianchi IX
gravitational instantons  have been explicitly
expressed in \cite{BabKor}, which
are Einstein metrics while having self-dual Weyl tensors.
It is remarkable that starting from writing a gravitational
instanton with $SU(2)$ symmetry in the general form,
\begin{equation} \label{ConformalBianchiIXMetricEq1}
d\tilde s^2= F ds^2 = F \left ( w_1 w_2 w_3 \, d\mu^2 +
\frac{w_2 w_3}{w_1} \sigma_1^2 +
\frac{w_3 w_1}{w_2} \sigma_2^2+
\frac{w_1 w_2}{w_3} \sigma_3^2  \right ),
\end{equation} 
where, like the $w_i$, $F$ is also a function of the cosmic time 
$\mu$, 
the solutions of the equations for the self-duality of the
Weyl tenor and proportionality of the Ricci tensor to the
metric are classified completely in terms of solutions
to Painlev\'e VI integrable systems \cite{Hit, Oku, Tod}.  In turn, the latter
can be solved \cite{BabKor} by using
the $\tau$-function of the Schlesinger system
formulated in terms of theta functions  \cite{KitKor}.
We will review the explicit parametrization of the Bianchi IX 
gravitational instantons in Section \ref{InstantonsSec}. In this subsection 
we present a rationality statement for the spectral action 
of the metric \eqref{ConformalBianchiIXMetricEq1}. This result indicates the existence of an 
arithmetic structure in the spectral action of these metrics, which,  
combined with the parametrization in terms of theta functions with characteristics \cite{BabKor}, 
leads to a construction of modular functions, which is one of our main 
objectives in this paper.

\smallskip

By an explicit calculation following the notions and the construction described in 
Subsection \ref{DiracOpSubSec}, we find that the Dirac operator 
$\tilde D$ of the metric \eqref{ConformalBianchiIXMetricEq1} 
is given by 
\begin{equation} \label{DiracConfBianchiIXEq}
\tilde{D}=\frac{1}{\sqrt{F}}D+\frac{3F^{'}}{4F^{\frac{3}{2}}w_{1}w_{2}w_{3}}\gamma^{0}, 
\end{equation}
where $D$ is the Dirac operator given by \eqref{DiracBianchiEq}, 
the Dirac operator of the metric \eqref{BianchimetricEq}.

\smallskip

We also find that 
\begin{eqnarray*}
\tilde{D}^{2}&=&
\frac{1}{F}D^{2}
+\frac{F^{'}}{2F^{2}w_{1}w_{2}w_{3}}\left(w_{1}\gamma^{0}\gamma^{1}\sin\psi-w_{2}\gamma^{0}\gamma^{2}\cos\psi\right)\frac{\partial}{\partial \eta} \\
&&-\frac{F^{'}\csc\eta}{2F^{2}w_{1}w_{2}w_{3}}\left(w_{1}\gamma^{0}\gamma^{1}\cos\psi-w_{2}\gamma^{0}\gamma^{2}\sin\psi\right)\frac{\partial}{\partial \phi}\\
&&+\frac{F^{'}\cot\eta}{2F^{2}w_{1}w_{2}w_{3}}\left(w_{1}\gamma^{0}\gamma^{1}\cos\psi+w_{2}\gamma^{0}\gamma^{2}\sin\psi-w_{3}\gamma^{0}\gamma^{3}\tan\eta\right)
\frac{\partial}{\partial \psi} \\
&&-\frac{F^{'}}{F^{2}w_{1}w_{2}w_{3}}\frac{\partial}{\partial \mu}+\frac{9F^{'2}}{16F^{3}w_{1}w_{2}w_{3}}+\frac{F^{'}w_{1}^{'}}{8F^{2}w_{1}^{2}w_{2}w_{3}}+\frac{F^{'}w_{2}^{'}}{8F^{2}w_{1}w_{2}^{2}w_{3}} \\
&&+\frac{F^{'}w_{3}^{'}}{8F^{2}w_{1}w_{2}w_{3}^{2}}+\frac{1}{8F^{2}w_{1}^{2}w_{2}^{2}w_{3}^{2}}(\frac{1}{w_{1}^{2}}+\frac{1}{w_{2}^{2}}+\frac{1}{w_{3}^{2}})F^{'}\gamma^{0} \gamma^{1}\gamma^{2}\gamma^{3}.  
\end{eqnarray*}
Therefore, we have the pseudo-differential symbol of $\tilde D^2$, 
since that of $D^2$ was calculated in \cite{FanFatMar1}.

\smallskip

Now, by following a quite similar approach to the one taken in 
\cite{FanFatMar1} for the rationality result, we present 
a variant of that result for the metric 
\eqref{ConformalBianchiIXMetricEq1}. That is, considering the 
asymptotic expansion 
\begin{equation} \label{ExpAsympConformalEq}
\text{Trace}\left ( \exp (-t \tilde{D}^2)   \right )
\sim
t^{-2} \sum_{n=0}^\infty \tilde{a}_{2n} t^n, \qquad t \to 0^+, 
\end{equation}
each $\tilde a_{2n}$ is of the general form written in the following 
statement. 

\smallskip

\begin{theorem} \label{ConformalRationlaityThm}
The term $\tilde{a}_{2n}$ in the above asymptotic
expansion, modulo an integration with respect to $\mu$,
is of the form
\[
\tilde{a}_{2n} =   \frac{\tilde{Q}_{2n}
\left (w_1, w_2, w_3, F, w_1', w_2', w_3', F',  \dots,
w_1^{(2n)}, w_2^{(2n)}, w_3^{(2n)}, F^{(2n)} \right )}{F^{2n} (w_1 w_2 w_3)^{3n-1}},
\]
where $\tilde{Q}_{2n}$ is a polynomial of several variables with
rational coefficients.
\begin{proof}
We provide an outline. One can exploit the $SU(2)$-invariance of the 
1-forms $\sigma_i$ appearing in the metric \eqref{ConformalBianchiIXMetricEq1} 
to show that functions of the type \eqref{DensitiesFormula}, whose integrals give 
the coefficients $\tilde a_{2n}$, have no spatial dependence when the metric 
is given by \eqref{ConformalBianchiIXMetricEq1}. Then, one employs the formula 
\eqref{HeatCoefsResEq} along with the pseudo-differential symbol of $\tilde D^2$, 
and continues with  similar arguments to those of  
Theorem 5.1 in \cite{FanFatMar1}. 
\end{proof}
\end{theorem}

\smallskip

Let us end this section by recording explicit expressions for the 
first few coefficients appearing in \eqref{ExpAsympConformalEq}, 
which were computed in two different ways to confirm their validity. 
In fact we first computed them by the method reviewed in Subsection 
\ref{HeatExpSubSec} leading to the formulas  \eqref{a_2nformula} 
and \eqref{DensitiesFormula}, and then confirmed that the expressions 
match precisely with the outcome of our calculations based on the formula 
\eqref{HeatCoefsResEq} which used the noncommutative residue.

\smallskip

The first coefficient, which is the volume term, is given by 
\begin{equation} \label{a_0Eq}
\tilde{a}_{0}=4F^{2}w_{1}w_{2}w_{3}. 
\end{equation}
The next term, which is the Einstein-Hilbert action term, has the following 
rather short expression, which indicates occurrence of remarkable 
simplifications in the 
final formula: 
\begin{eqnarray}  \label{a_2Eq}
\tilde{a}_{2} &=& -\frac{F}{3}\Big (w_{1}^{2}+w_{2}^{2}+w_{3}^{2} \Big)+\frac{F}{6}\Big (\frac{w_{1}^{2}w_{2}^{2}-w_{3}^{'2}}{w_{3}^{2}}+\frac{w_{1}^{2}w_{3}^{2}-w_{2}^{'2}}{w_{2}^{2}}+\frac{w_{2}^{2}w_{3}^{2}-w_{1}^{'2}}{w_{1}^{2}} \Big ) \nonumber \\
&& -\frac{F}{3}\Big (\frac{w_{1}^{'}w_{2}^{'}}{w_{1}w_{2}}+\frac{w_{1}^{'}w_{3}^{'}}{w_{1}w_{3}}+\frac{w_{2}^{'}w_{3}^{'}}{w_{2}w_{3}}\Big )+\frac{F}{3}\Big (\frac{w_{1}^{''}}{w_{1}}+\frac{w_{2}^{''}}{w_{2}}+\frac{w_{3}^{''}}{w_{3}}\Big )-\frac{F^{'2}}{2F}+F^{''}.
\end{eqnarray}
The term $\tilde{a}_{4}$, which is the Gauss-Bonnet term, also enjoys 
remarkable simplifications in its final formula, however, since it has a 
lengthier expression, we present it in Appendix \ref{fulla_4appendix}.

\smallskip

\section{Bianchi IX gravitational instantons}
\label{InstantonsSec}

\smallskip

There is an especially interesting family of metrics called Bianchi IX 
gravitational instantons, which have been explored in the literature by 
imposing the self-duality 
condition on triaxial Bianchi IX metrics and by employing a time-dependent conformal 
factor $F(\mu)$ to obtain an Einstein metric, see \cite{Tod, Hit, BabKor} and references therein.  Let us provide an outline of some of the main ideas and steps for 
deriving these metrics, from 
the literature. In particular, we will then present the explicit 
parametrization of the solutions from \cite{BabKor} in terms of theta functions with characteristics.

\smallskip

The  sought after gravitational instantons can be written in the 
general form  
\begin{equation} \label{ConformalBianchiIXMetricEq2}
d\tilde s^2= F ds^2 = F \left ( w_1 w_2 w_3 \, d\mu^2 +
\frac{w_2 w_3}{w_1} \sigma_1^2 +
\frac{w_3 w_1}{w_2} \sigma_2^2+
\frac{w_1 w_2}{w_3} \sigma_3^2  \right ). 
\end{equation}
Then, the differential equations derived from imposing the self-duality of the Weyl tensor 
and the condition of being an Einstein metric are solved \cite{Tod, Hit, BabKor} by turning  
them into well-studied systems of differential equations as follows.  
One can start by considering a basis of anti-self-dual 2-forms
\[
\varphi^j = w_j w_k \,d\mu \wedge \sigma_1 -  w_j \sigma_k \wedge \sigma_l, 
\]
where all cyclic permutations $(j, k, l)$ of $(1, 2, 3)$ are understood to be considered. 
Consider the connection 1-forms $\alpha^j_k$  appearing in 
$
d \varphi_j = \sum \alpha^j_k \wedge \varphi^k. 
$
This  leads to writing 
$
\alpha^j_k = \frac{A_k}{w_k} \sigma_k, 
$
where the functions $A_k$ satisfy the system of equations 
\begin{equation} \label{wjAjDiffEq}
\frac{d w_j}{d \mu} = - w_j w_k + w_j (A_k + A_l). 
\end{equation}
It can then be seen that the self-duality condition on the Weyl tensor 
yields the classical Halphen system
\begin{equation} \label{HalphenEq}
\frac{d A_j}{d \mu}= - A_k A_l +A_j(A_k + A_l). 
\end{equation}
Remarkably, the latter has well-known solutions in terms of the 
theta functions that will be defined shortly by \eqref{varthetasEq}.

\smallskip

One can define a new variable 
$
x= \frac{A_1-A_2}{A_3-A_2}, 
$
which has an evident dependence on $\mu$. Then the Halphen system \eqref{HalphenEq} reduces 
to the following differential equation which is satisfied by the reciprocal 
of the elliptic modular function, see  \cite{Tod} and references therein: 
\[
\frac{ d^3 x / d \mu^3  }{d x / d \mu} 
= 
\frac{3}{2} \frac{ d^2 x / d \mu^2  }{(d x / d \mu)^2} 
- \frac{1}{2} \frac{dx}{d \mu} \left (  \frac{1}{x^2} + \frac{1}{x(1-x)} + \frac{1}{(1-x)^2} \right ). 
\]
Therefore, one can solve the equation \eqref{HalphenEq} and substitute the solution in 
\eqref{wjAjDiffEq}. The latter can then be solved more conveniently by setting 
\begin{eqnarray*}
w_1 &=& \Omega_1 \frac{dx}{d \mu} \left ( x (1-x) \right )^{-1/2}, \\
w_2 &=& \Omega_2 \frac{dx}{d \mu} \left ( x^2 (1-x) \right )^{-1/2}, \\ 
w_3 &=& \Omega_3 \frac{dx}{d \mu} \left ( x (1-x)^2 \right )^{-1/2}, 
\end{eqnarray*}
and by viewing $x$ as the independent variable, which yields: 
\begin{eqnarray*}
\frac{d \Omega_1}{dx} = - \frac{\Omega_2 \Omega_3}{x(1-x)}, 
\qquad 
\frac{d \Omega_2}{dx} = - \frac{\Omega_3 \Omega_1}{x},  
\qquad 
\frac{d \Omega_3}{dx} = - \frac{\Omega_1 \Omega_2}{1-x}. 
\end{eqnarray*}
It is well-known that these equations  reduce to the 
Painlev\'e VI equation with particular parameters, which 
along with the condition of making the metric Einstein 
in the conformal class using a time-dependent conformal factor, 
one reduces to the following rational parameters \cite{Hit, Oku, Tod}:    
\[
(\alpha ,\beta ,\gamma , \delta)
=
(\frac{1}{8},  -\frac{1}{8}, \frac{1}{8},  \frac{3}{8}). 
\]

\smallskip

It should be noted that in general a Painlev\'e VI equation with 
parameters $(\alpha ,\beta ,\gamma , \delta)$ is  of the form
\begin{eqnarray*}
\frac{d^2X}{dt^2}&=&\frac{1}{2}\left(
\frac{1}{X}+\frac{1}{X-1}+\frac{1}{X-t}\right)
\left(\frac{dX}{dt}\right)^2  
-\left( \frac{1}{t}+\frac{1}{t-1}+\frac{1}{X-t}\right)\frac{dX}{dt} \\&&
+\frac{X(X-1)(X-t)}{t^2(t-1)^2}
\left(\alpha +
\beta\frac{t}{X^2}+\gamma\frac{t-1}{(X-1)^2}+
\delta\frac{t(t-1)}{(X-t)^2}\right). 
\end{eqnarray*}

\smallskip

Going through the process outlined above and 
solving the involved equations in terms of 
elliptic theta functions \cite{Tod, Hit, BabKor}, and using 
the formula for the $\tau$-function of the Schlesinger equation \cite{KitKor}  
with an additional elegant calculation the square root of some expressions in \cite{BabKor}, 
cf. \cite{Hit}, a  
parametrization of the Bianchi IX
gravitational instantons can be given as follows.

\smallskip

The final solutions in \cite{BabKor} are written in terms of theta functions with characteristics,
which for $p, q, z, \in \mathbb{C},  i \mu \in \mathbb{H},$
are given by
\begin{equation} \label{ThetawithCharEq}
\vartheta [p,q](z, i\mu )
=
\sum_{m\in {\mathbb Z}} \exp \left( -\pi  (m+p)^2\mu + 2\pi i (m+p)(z+q)\right).
\end{equation}
Considering Jacobi's theta function defined by
\[
\Theta( z | \tau)  = \sum_{m \in \mathbb{Z}}  e^{\pi i m^2 \tau} e^{2 \pi i m z},
\qquad z \in \mathbb{C},  \qquad \tau \in \mathbb{H},
\]
and by using the notation which sets $z=0$ in \eqref{ThetawithCharEq}, 
\begin{equation} \label{varthetapqEq}
\vartheta [p,q]( i \mu) = \vartheta [p,q] (0,i\mu ),
\end{equation}
the following functions are also necessary to be introduced: 
\begin{eqnarray} \label{varthetasEq} 
\vartheta_{2}(i\mu)&=&\vartheta[\frac{1}{2},0](i\mu)=\sum_{m\in\mathbb{Z}}\exp\{-\pi(m+\frac{1}{2})^{2}\mu\}=e^{-\frac{1}{4}\pi\mu}\Theta \big (\frac{i\mu}{2}|i\mu \big), \nonumber \\
\vartheta_{3}(i\mu)&=&\vartheta[0,0](i\mu)=\sum_{m\in\mathbb{Z}}\exp\{-\pi m^{2}\mu\}=\Theta(0|i\mu), \nonumber \\
\vartheta_{4}(i\mu)&=&\vartheta[0,\frac{1}{2}](i\mu)=\sum_{m\in\mathbb{Z}}\exp\{-\pi m^{2}\mu+\pi im\}=\Theta \big (\frac{1}{2}|i\mu \big ).
\end{eqnarray}

\smallskip

We are now ready to write the explicit formulas for the Bianchi IX gravitational instantons 
presented in  \cite{BabKor}, in the following two subsections. The first family is two-parametric, 
which consists of the case of non-vanishing cosmological constants, and the second is a 
one-parametric family whose cosmological constants vanish.  By studying the asymptotic 
behavior of these solutions as $\mu \to \infty$, it is shown in \cite{ManMar} that 
they approximate Eguchi-Hanson type gravitational instantons with $ w_1 \neq w_2 = w_3 $ 
\cite{EguHan} for large $\mu$.

\smallskip

\subsection{The two-parametric family with non-vanishing cosmological constants}

The two-parametric family of solutions with
parameters $p, q \in \mathbb{C}$ is given by 
substituting the following functions in the metric 
\eqref{ConformalBianchiIXMetricEq2}: 
\begin{eqnarray} \label{two-parametric}
w_{1}[p,q](i\mu)&=&-\frac{i}{2}\vartheta_{3}(i\mu)\vartheta_{4}(i\mu)\frac{\partial_{q}\vartheta[p,q+\frac{1}{2}](i\mu)}{e^{\pi ip}\vartheta[p,q](i\mu)}, \nonumber \\
w_{2}[p,q](i\mu)&=&\frac{i}{2}\vartheta_{2}(i\mu)\vartheta_{4}(i\mu)\frac{\partial_{q}\vartheta[p+\frac{1}{2},q+\frac{1}{2}](i\mu)}{e^{\pi ip}\vartheta[p,q](i\mu)}, \nonumber \\
w_{3}[p,q](i\mu)&=&-\frac{1}{2}\vartheta_{2}(i\mu)\vartheta_{3}(i\mu)\frac{\partial_{q}\vartheta[p+\frac{1}{2},q](i\mu)}{\vartheta[p,q](i\mu)}, \nonumber \\
F[p,q](i\mu)&=&\frac{2}{\pi\Lambda} \frac{1}{(\partial_{q}\ln\vartheta[p,q](i\mu))^{2}}=\frac{2}{\pi\Lambda} \left(\frac{\vartheta[p,q](i\mu)}{\partial_{q}\vartheta[p,q](i\mu)}\right)^{2}.
\end{eqnarray}

\smallskip

\subsection{The one-parametric family with vanishing cosmological constants}
The one-parametric family of the metrics with the parameter $q_0 \in \mathbb{R}$ is given by 
the following solutions that need to be substituted in the metric \eqref{ConformalBianchiIXMetricEq2}:
\begin{eqnarray} \label{one-parametric}
w_1[q_0](i \mu) &=& \frac{1}{\mu+q_0} +2 \frac{d}{d\mu} \log \vartheta_2 (i \mu), \nonumber \\
w_2[q_0](i \mu) &=& \frac{1}{\mu+q_0} +2 \frac{d}{d\mu} \log \vartheta_3 (i \mu), \nonumber \\
w_3[q_0](i \mu) &=&\frac{1}{\mu+q_0} +2 \frac{d}{d\mu} \log \vartheta_4 (i \mu), \nonumber \\
F[q_0](i \mu)&=& C (\mu + q_0)^2,
\end{eqnarray}
where $C$ is an arbitrary positive constant.

\smallskip

\section{Arithmetics of Bianchi IX gravitational instantons}
\label{ArithmeticsofInstantonsSec}

\smallskip

This section is devoted to the investigation of modular properties 
of the functions appearing in the formulas \eqref{two-parametric} 
and \eqref{one-parametric} and that of their derivatives. When the functions 
$w_1$, $w_2$, $w_3$ and $F$  are substituted from the latter identities in the 
metric \eqref{ConformalBianchiIXMetricEq2}, Theorem \ref{ConformalRationlaityThm} implies that the 
Seeley-de Witt coefficients 
$\tilde a_{2n}$  are rational functions of $\vartheta_2, \vartheta_3, \vartheta_4$,  $\vartheta[p,q]$, 
$\partial_q \vartheta[p,q]$, $e^{i\pi p}$ and their derivatives with rational coefficients. 
Therefore, finding out modular properties of these theta functions and consequently that of the functions 
$w_1$, $w_2$, $w_3$, $F$ and their derivatives, will help us to 
investigate modular transformation laws of the $\tilde a_{2n}$, under 
modular transformations on $i \mu$ belonging to the upper half-plane $\mathbb{H}$.

\smallskip

We begin studying the two-parametric case \eqref{two-parametric}. First, let us  note that for the derivatives 
of the function $\vartheta[p, q](i \mu)$ given by \eqref{varthetapqEq}, we have
\begin{eqnarray*}
\partial_{\mu}^{n}\vartheta[p,q](i\mu) &=& \sum_{m\in\mathbb{Z}}(-\pi)^{n}(m+p)^{2n}
e^{ -\pi(m+p)^{2}\mu+2\pi i(m+p)q}, \\
\partial_{\mu}^{n}\partial_{q}\vartheta[p,q](i\mu)&=&2i(-1)^{n}\pi^{n+1}\sum_{m\in\mathbb{Z}}(m+p)^{2n+1} e^{-\pi(m+p)^{2}\mu+2\pi i(m+p)q }. 
\end{eqnarray*}
We also need to prove the following lemma,  which will be of crucial use 
for proving the transformation properties investigated in this section. 

\smallskip

\begin{lemma} \label{PoissonsumLem}

We have

\begin{eqnarray*}
&&\sum_{m\in\mathbb{Z}}(m+p)^{n}e^{-\frac{\pi}{\mu}(m+p)^{2}+2\pi i(m+p)q} \\
&& \quad = e^{2\pi ipq}\sum_{j=0}^{[n/2]}
\Large ( \frac{(-i)^{n-2j}\mu^{n+\frac{1}{2}-j} n!}{(2\pi)^{j}(n-2j)!\cdot(2j)!!}
 \sum_{m\in\mathbb{Z}}(m-q)^{n-2j}e^{-\pi\mu(m-q)^{2}+2\pi ip(m-q)}  \Large ).
\end{eqnarray*}

\begin{proof}

It follows from the following identity and the Poisson summation
formula:

\begin{eqnarray*}
&&\int d\xi\cdot e^{2\pi i\cdot\xi x}(\xi+p)^{n}e^{-\frac{\pi}{\mu}(\xi+p)^{2}+2\pi i(\xi+p)q} \\
&&\quad =e^{2\pi ipq}e^{-\pi\mu(-x-q)^{2}+2\pi ip(-x-q)}\int d\xi\cdot(\xi+i\mu(x+q))^{n}e^{-\frac{\pi}{\mu}\xi{}^{2}} \\
&&  =e^{2\pi ipq}e^{-\pi\mu(-x-q)^{2}+2\pi ip(-x-q)}\sum_{j=0}^{[n/2]}\frac{i^{n-2j}\mu^{n+\frac{1}{2}-j}}{(2\pi)^{j}}\frac{n!}{(n-2j)!\cdot(2j)!!}(x+q)^{n-2j}.
\end{eqnarray*}

\end{proof}

\end{lemma}

\smallskip

For convenience, we need to use the constants 
\[
C(j|n)=\frac{(-i){}^{n}n!}{2^{j}(n-2j)!\cdot(2j)!!},
\]
which arise naturally in exploring the following
transformation properties. 

\smallskip

The following lemma shows that the function $\vartheta[p, q](i \mu)$ 
and its derivatives with respect to $\mu$ possess periodic and 
quasi-period properties with respect to the variables $p$ and $q$. 

\smallskip

\begin{lemma} \label{transformationsvartheta}
The functions
$\vartheta[p,q]$ are holomorphic in the half-plane $\Re (\mu)>0$
and satisfy the following properties:
\begin{eqnarray*}
\partial_{\mu}^{n}\vartheta[p,q+1](i\mu) &=& e^{2\pi ip}\partial_{\mu}^{n}\vartheta[p,q](i\mu), \\
\partial_{\mu}^{n}\partial_{q}\vartheta[p,q+1](i\mu)&=&e^{2\pi ip}\partial_{\mu}^{n}\partial_{q}\vartheta[p,q](i\mu), \\
\partial_{\mu}^{n}\vartheta[p+1,q](i\mu) &=&\partial_{\mu}^{n}\vartheta[p,q](i\mu), \\
\partial_{\mu}^{n}\partial_{q}\vartheta[p+1,q](i\mu)  &=& \partial_{\mu}^{n}\partial_{q}\vartheta[p,q](i\mu). 
\end{eqnarray*}

\begin{proof}

One can start from the definition given by \eqref{varthetapqEq} and  \eqref{ThetawithCharEq} to write 
the following for proving the first identity. 
\begin{eqnarray*}
&&\partial_{\mu}^{n}\vartheta[p,q+1](i\mu) \\
&& \quad =
e^{2\pi ip}(-1)^{n}\pi^{n}\sum_{m\in\mathbb{Z}}(m+p)^{2n}\exp\{-\pi(m+p)^{2}\mu+2\pi i(m+p)q\} \\
&& \quad =e^{2\pi ip}\partial_{\mu}^{n}\vartheta[p,q](i\mu),
\end{eqnarray*}
The second identity can be seen to hold by writing 
\begin{eqnarray*}
&&\partial_{\mu}^{n}\partial_{q}\vartheta[p,q+1](i\mu)\\
&&\quad =e^{2\pi ip}2i(-1)^{n}\pi^{n+1}\sum_{m\in\mathbb{Z}}(m+p)^{2n+1}\cdot\exp\{-\pi(m+p)^{2}\mu+2\pi i(m+p)q\} \\
&& \quad =e^{2\pi ip}\partial_{\mu}^{n}\partial_{q}\vartheta[p,q](i\mu). 
\end{eqnarray*}
This proves the quasi-periodicity property of the theta function and its derivative with the quasi-period 
1 in the $q$-variable. 

\smallskip

The periodicity properties with period 1 in the $p$-variable can be similarly investigated by writing 
\begin{eqnarray*}
&&\partial_{\mu}^{n}\vartheta[p+1,q](i\mu) \\
&& \quad =(-1)^{n}\pi^{n}\sum_{m\in\mathbb{Z}}(m+1+p)^{2n}\exp\{-\pi(m+1+p)^{2}\mu+2\pi i(m+1+p)q\} \\
&& \quad =\partial_{\mu}^{n}\vartheta[p,q](i\mu),
\end{eqnarray*}
and 
\begin{eqnarray*}
&&\partial_{\mu}^{n}\partial_{q}\vartheta[p+1,q](i\mu) \\
&&  =2i(-1)^{n}\pi^{n+1}\sum_{m\in\mathbb{Z}}(m+1+p)^{2n+1}\cdot\exp\{-\pi(m+1+p)^{2}\mu+2\pi i(m+1+p)q\} \\
&&\quad =\partial_{\mu}^{n}\partial_{q}\vartheta[p,q](i\mu). 
\end{eqnarray*}

\end{proof}

\end{lemma}

\smallskip

We now focus on modular transformations. For convenience, we use the following 
notation for the linear fractional transformations corresponding to the generators 
of the modular group acting on the upper half-plane $\mathbb{H}$ in the 
complex plane: 
\begin{equation} \label{T_1andSEq}
T_1(\tau) = \tau + 1, \qquad S(\tau) = \frac{-1}{\tau}, \qquad \tau \in \mathbb{H}. 
\end{equation}

\smallskip

In the following lemma, we present the transformation properties of the 
$\vartheta[p, q](i \mu)$ and its derivatives under the modular transformations $T_1$, 
$T_2$, and $S$ on $i \mu \in \mathcal{H}$. 

\smallskip
 
\begin{lemma} \label{transformationsvarthetapq}

Let $\mu$ be a complex number in the right half-plane $\Re (\mu) >0$.  We have
\begin{eqnarray*}
\partial_{\mu}^{n}\vartheta[p,q](i\mu+1)&=& e^{-\pi ip(p+1)}\partial_{\mu}^{n}\vartheta[p,q+p+\frac{1}{2}](i\mu),\\
\partial_{\mu}^{n}\partial_{q}\vartheta[p,q](i\mu+1)  &=& e^{-\pi ip(p+1)}\partial_{\mu}^{n}\partial_{q}\vartheta[p,q+p+\frac{1}{2}](i\mu),\\
\partial_{\mu}^{n}\vartheta[p,q](i\mu+2)&=& e^{-2\pi ip^{2}}\partial_{\mu}^{n}\vartheta[p,q+2p](i\mu),\\
\partial_{\mu}^{n}\partial_{q}\vartheta[p,q](i\mu+2)  &=& e^{-2\pi ip^{2}}\partial_{\mu}^{n}\partial_{q}\vartheta[p,q+2p](i\mu),\\
\partial_{\mu}^{n}\vartheta[p,q](\frac{i}{\mu}) &=& e^{2\pi ipq}\sum_{j=0}^{n}C(j|2n)\mu^{2n+\frac{1}{2}-j}\partial_{\mu}^{n-j}\vartheta[-q,p](i\mu),\\
\partial_{\mu}^{n}\partial_{q}\vartheta[p,q](\frac{i}{\mu}) &=& e^{2\pi ipq}\sum_{j=0}^{n}C(j|2n+1)\mu^{2n+\frac{3}{2}-j}\partial_{\mu}^{n-j}\partial_{p}\vartheta[-q,p](i\mu).
\end{eqnarray*}

\begin{proof}

Considering the formulas \eqref{varthetapqEq} and  \eqref{ThetawithCharEq} we can write
\begin{eqnarray*}
&&\partial_{\mu}^{n}\vartheta[p,q](i\mu+1) \\
&&\quad =(-1)^{n}\pi^{n}\sum_{m\in\mathbb{Z}}(m+p)^{2n}e^{-\pi(m+p)^{2}\mu+2\pi i(m+p)(q+p+\frac{1}{2})-\pi ip^2 - \pi i p} \\
&& \quad =e^{-\pi ip(p+1)}\partial_{\mu}^{n}\vartheta[p,q+p+\frac{1}{2}](i\mu),
\end{eqnarray*}
and
\begin{eqnarray*}
&&\partial_{\mu}^{n}\partial_{q}\vartheta[p,q](i\mu+1) \\
&& \quad =2i(-1)^{n}\pi^{n+1}\sum_{m\in\mathbb{Z}}(m+p)^{2n+1}\cdot e^{-\pi(m+p)^{2}\mu+2\pi i(m+p)(q+p+\frac{1}{2})-\pi ip^2 - \pi i p} \\
&& \quad =e^{-2\pi ip^{2}}\partial_{\mu}^{n}\partial_{q}\vartheta[p,q+p+\frac{1}{2}](i\mu). 
\end{eqnarray*}
This establishes the first two identities for the transformation properties of $\vartheta[p, q]$ and 
its derivatives with respect to the modular action $T_1$ given by \eqref{T_1andSEq}; the third and the fourth identity follows immediately from the first and the second by applying them twice.

\smallskip

In order to investigate the transformation properties with respect to the action of $S$ given in \eqref{T_1andSEq}, 
we need to use Lemma \ref{PoissonsumLem}, which allows us to write
\begin{eqnarray*}
\partial_{\mu}^{n}\vartheta[p,q](\frac{i}{\mu})&=&(-1)^{n}\pi^{n}\sum_{m\in\mathbb{Z}}(m+p)^{2n}\exp\{-\frac{\pi}{\mu}(m+p)^{2}+2\pi i(m+p)q\} \\
&=& e^{2\pi ipq}\sum_{j=0}^{n}\mu^{2n+\frac{1}{2}-j}\frac{(-1)^{n}(2n)!}{2^{j}(2n-2j)!\cdot(2j)!!}(-1)^{n-j}\pi^{n-j} \times \\
&& \qquad  \qquad \qquad\qquad   \sum_{m\in\mathbb{Z}}(m-q)^{2n-2j}e^{-\pi\mu(m-q)^{2}+2\pi ip(m-q)} \\
&=&e^{2\pi ipq}\sum_{j=0}^{n}\mu^{2n+\frac{1}{2}-j}\frac{(-i){}^{2n}(2n)!}{2^{j}(2n-2j)!\cdot(2j)!!}\partial_{\mu}^{n-j}\vartheta[-q,p](i\mu) \\
&=&e^{2\pi ipq}\sum_{j=0}^{n}C(j|2n)\mu^{2n+\frac{1}{2}-j}\partial_{\mu}^{n-j}\vartheta[-q,p](i\mu). 
\end{eqnarray*}
Also we have: 
\begin{eqnarray*}
\partial_{\mu}^{n}\partial_{q}\vartheta[p,q](\frac{i}{\mu})&=&2i(-1)^{n}\pi^{n+1}\sum_{m\in\mathbb{Z}}(m+p)^{2n+1}\cdot e^{ -\frac{\pi}{\mu}(m+p)^{2}+2\pi i(m+p)q } \\
&=&e^{2\pi ipq}\sum_{j=0}^{n}\mu^{2n+\frac{3}{2}-j}\frac{-i(-1)^{n}(2n+1)!}{2^{j}(2n+1-2j)!\cdot(2j)!!}2i(-1)^{n-j}\pi^{n-j+1} \times \\
&&  \qquad \qquad \qquad \qquad \quad  \sum_{m\in\mathbb{Z}}(m-q)^{2(n-j)+1}e^{-\pi\mu(m-q)^{2}+2\pi ip(m-q)} \\
&=& e^{2\pi ipq}\sum_{j=0}^{n}\mu^{2n+\frac{3}{2}-j}\frac{(-i)^{2n+1}(2n+1)!}{2^{j}(2n+1-2j)!\cdot(2j)!!}\partial_{\mu}^{n-j}\partial_{p}\vartheta[-q,p](i\mu) \\
&=&e^{2\pi ipq}\sum_{j=0}^{n}C(j|2n+1)\mu^{2n+\frac{3}{2}-j}\partial_{\mu}^{n-j}\partial_{p}\vartheta[-q,p](i\mu).
\end{eqnarray*}

\end{proof}

\end{lemma}

\smallskip

Now we investigate the transformation properties of 
the functions $\vartheta_2, \vartheta_3, \vartheta_4,$
given by \eqref{varthetasEq} and their derivatives, under 
the same modular actions as the ones considered above, namely 
 $T_1$ and $S$ given by \eqref{T_1andSEq} 
transforming $i \mu$ in the upper half-plane.   

\smallskip

\begin{lemma}
\label{transformationsvartheta234}

Let $\Re(\mu) >0$. 
The functions $\vartheta_2, \vartheta_3, \vartheta_4$ and their derivative 
satisfy the following modular transformation properties: 
\begin{align*}
\partial_{\mu}^{n}\vartheta_{2}(i\mu+1)&= e^{\frac{\pi i}{4}}\partial_{\mu}^{n}\vartheta_{2}(i\mu), &
\partial_{\mu}^{n}\vartheta_{2}(\frac{i}{\mu}) &= \sum_{j=0}^{n}C(j|2n)\mu^{2n+\frac{1}{2}-j}\partial_{\mu}^{n-j}\vartheta_{4}(i\mu), \\
\partial_{\mu}^{n}\vartheta_{3}(i\mu+1) & = \partial_{\mu}^{n}\vartheta_{4}(i\mu), &
\partial_{\mu}^{n}\vartheta_{3}(\frac{i}{\mu})  &= \sum_{j=0}^{n}C(j|2n)\mu^{2n+\frac{1}{2}-j}\partial_{\mu}^{n-j}\vartheta_{3}(i\mu),& \\
\partial_{\mu}^{n}\vartheta_{4}(i\mu+1) &= \partial_{\mu}^{n}\vartheta_{3}(i\mu), &
\partial_{\mu}^{n}\vartheta_{4}(\frac{i}{\mu}) &= \sum_{j=0}^{n}C(j|2n)\mu^{2n+\frac{1}{2}-j}\partial_{\mu}^{n-j}\vartheta_{2}(i\mu).
\end{align*}

\begin{proof}

The first identity follows from considering the formula \eqref{varthetasEq} and writing
\begin{eqnarray*}
\partial_{\mu}^{n}\vartheta_{2}(i\mu+1)&=&\partial_{\mu}^{n}\vartheta[\frac{1}{2},0](i\mu+1)
=e^{-\frac{3\pi i}{4}}\partial_{\mu}^{n}\vartheta[\frac{1}{2},1](i\mu)\\
&=& e^{\frac{\pi i}{4}}\partial_{\mu}^{n}\vartheta[\frac{1}{2},0](i\mu)
=e^{\frac{\pi i}{4}}\partial_{\mu}^{n}\vartheta_{2}(i\mu). 
\end{eqnarray*}

\smallskip

For the next identity, which involves the action $S(i \mu) = i/\mu$, one needs to use 
Lemma \ref{PoissonsumLem}: 
\begin{eqnarray*}
\partial_{\mu}^{n}\vartheta_{2}(\frac{i}{\mu}) &=& \partial_{\mu}^{n}\vartheta[\frac{1}{2},0](\frac{i}{\mu}) 
=\sum_{j=0}^{n}C(j|2n)\mu^{2n+\frac{1}{2}-j}\partial_{\mu}^{n-j}\vartheta[0,\frac{1}{2}](i\mu) \\
&=&\sum_{j=0}^{n}C(j|2n)\mu^{2n+\frac{1}{2}-j}\partial_{\mu}^{n-j}\vartheta_{4}(i\mu). 
\end{eqnarray*}

\smallskip

The analogous identities for the functions $\vartheta_3$, $\vartheta_4$, and their derivatives 
can be proved in a similar manner. In the case of $\vartheta_3$ we have 
\begin{eqnarray*}
\partial_{\mu}^{n}\vartheta_{3}(i\mu+1)&=&\partial_{\mu}^{n}\vartheta[0,0](i\mu+1)
=\partial_{\mu}^{n}\vartheta[0,\frac{1}{2}](i\mu) 
=\partial_{\mu}^{n}\vartheta_{4}(i\mu),
\end{eqnarray*}
\begin{eqnarray*}
\partial_{\mu}^{n}\vartheta_{3}(\frac{i}{\mu}) &=&
\partial_{\mu}^{n}\vartheta[0,0](\frac{i}{\mu}) 
=\sum_{j=0}^{n}C(j|2n)\mu^{2n+\frac{1}{2}-j}\partial_{\mu}^{n-j}\vartheta[0,0](i\mu) \\
&=&\sum_{j=0}^{n}C(j|2n)\mu^{2n+\frac{1}{2}-j}\partial_{\mu}^{n-j}\vartheta_{3}(i\mu),
\end{eqnarray*}
and for $\vartheta_4$ we have
\begin{eqnarray*}
\partial_{\mu}^{n}\vartheta_{4}(i\mu+1)=\partial_{\mu}^{n}\vartheta[0,1](i\mu)=\partial_{\mu}^{n}\vartheta_{3}(i\mu), 
\end{eqnarray*}
\begin{eqnarray*}
\partial_{\mu}^{n}\vartheta_{4}(\frac{i}{\mu}) &=& \partial_{\mu}^{n}\vartheta[0,\frac{1}{2}](i\mu) 
=\sum_{j=0}^{n}C(j|2n)\mu^{2n+\frac{1}{2}-j}\partial_{\mu}^{n-j}\vartheta[\frac{1}{2},0](i\mu) \\
&=&\sum_{j=0}^{n}C(j|2n)\mu^{2n+\frac{1}{2}-j}\partial_{\mu}^{n-j}\vartheta_{2}(i\mu).
\end{eqnarray*}

\end{proof}

\end{lemma}

\smallskip

\subsection{The case of the two-parametric family of metrics}
\label{ModularPropsTwoParametricSubSec}

Using the above lemmas, we proceed to work out the modular
transformation rules for the functions $w_j[p, q]$ given by 
\eqref{two-parametric}  and their
derivatives. First, let us deal with the transformation $T_1$ given 
by \eqref{T_1andSEq}. 

\smallskip

\begin{lemma}
\label{transformationsw_j1}
The functions $w_j[p, q]$ and their derivatives of an arbitrary order $n \geq 1$ with 
respect to $\mu$ satisfy the
following identities: 
\begin{align*}
w_{1}[p,q](i\mu+1) &= w_{1}[p,q+p+\frac{1}{2}](i\mu), &
w_{1}^{(n)}[p,q](i\mu+1)&= w_{1}^{(n)}[p,q+p+\frac{1}{2}](i\mu), \\
 w_{2}[p,q](i\mu+1) &= w_{3}[p,q+p+\frac{1}{2}](i\mu), &  w_{2}^{(n)}[p,q](i\mu+1) &=w_{3}^{(n)}[p,q+p+\frac{1}{2}](i\mu), \\
w_{3}[p,q](i\mu+1) &= w_{2}[p,q+p+\frac{1}{2}](i\mu), &
w_{3}^{(n)}[p,q](i\mu+1)&=w_{2}^{(n)}[p,q+p+\frac{1}{2}](i\mu).
\end{align*}


\begin{proof}
Using Lemma \ref{transformationsvarthetapq} and Lemma \ref{transformationsvartheta234}, we have 
\begin{eqnarray*}
w_{1}[p,q](i\mu+1)&=&-\frac{i}{2}\vartheta_{3}(i\mu+1)\vartheta_{4}(i\mu+1)\frac{\partial_{q}\vartheta[p,q+\frac{1}{2}](i\mu+1)}{e^{\pi ip}\vartheta[p,q](i\mu+1)} \\
&=&-\frac{i}{2}\vartheta_{3}(i\mu)\vartheta_{4}(i\mu)\frac{e^{-\pi i p(p+1)}\partial_{q}\vartheta[p,q+\frac{1}{2}+p+\frac{1}{2}](i\mu)}{e^{-\pi i p(p+1)}e^{\pi ip}\vartheta[p,q+p+\frac{1}{2}](i\mu)}\\
&=&w_{1}[p,q+p+\frac{1}{2}](i\mu). 
\end{eqnarray*}



\smallskip

Similarly for $w_2$ and $w_3$ we can write: 
\begin{eqnarray*}
&&w_{2}[p,q](i\mu+1) \\
&& \quad =\frac{i}{2}\vartheta_{2}(i\mu+1)\vartheta_{4}(i\mu+1)\frac{e^{-\pi i(p+\frac{1}{2})(p+\frac{3}{2})}\partial_{q}\vartheta[p+\frac{1}{2},q+\frac{1}{2}+p+1](i\mu)}{e^{-\pi ip(p+1)}e^{\pi ip}\vartheta[p,q+p+\frac{1}{2}](i\mu)} \\
&&\quad =-\frac{1}{2}\vartheta_{2}(i\mu)\vartheta_{3}(i\mu)\frac{\partial_{q}\vartheta[p+\frac{1}{2},q+p+\frac{1}{2}](i\mu)}{\vartheta[p,q+p+\frac{1}{2}](i\mu)} \\
&& \quad =w_{3}[p,q+p+\frac{1}{2}](i\mu),
\end{eqnarray*}
\begin{eqnarray*}
&& w_{3}[p,q](i\mu+1) \\
 && \qquad =-\frac{1}{2}\vartheta_{2}(i\mu+1)\vartheta_{3}(i\mu+1)\frac{e^{-\pi i(p+\frac{1}{2})(p+\frac{3}{2})}\partial_{q}\vartheta[p+\frac{1}{2},q+p+\frac{1}{2}+\frac{1}{2}](i\mu)}{e^{-\pi ip(p+1)}\vartheta[p,q+p+\frac{1}{2}](i\mu)} \\
&&\qquad = \frac{i}{2}\vartheta_{2}(i\mu)\vartheta_{4}(i\mu)\frac{\partial_{q}\vartheta[p+\frac{1}{2},q+p+\frac{1}{2}+\frac{1}{2}](i\mu)}{\vartheta[p,q+p+\frac{1}{2}](i\mu)}
\\
&&\qquad =w_{2}[p,q+p+\frac{1}{2}](i\mu). 
\end{eqnarray*}

\smallskip

The identities for the arbitrary derivatives follow easily from differentiating the above 
equalities with respect to $\mu$. 



\end{proof}

\end{lemma}

\smallskip

The following lemmas show that the functions $w_j$ and their derivatives satisfy 
some transformation laws with respect to the modular action $S$ given by 
\eqref{T_1andSEq} as well. Let us start by the properties of $w_1$.

\smallskip

\begin{lemma}
\label{transformationsw_12}
Assume that the variable $\mu$ belongs to the right half-plane $\Re(\mu) >0$. The function $w_1[p, q]$ and its derivatives up to order 4 with respect to $\mu$ 
satisfy the following identities: 
\begin{eqnarray*}
w_{1}[p,q](\frac{i}{\mu}) &=&  \mu^{2}w_{3}[-q,p](i\mu),   \\
w_{1}^{'}[p,q](\frac{i}{\mu})&=&  -\mu^{4}w_{3}^{'}[-q,p](i\mu)-2\mu^{3}w_{3}[-q,p](i\mu),  \\
w_{1}^{''}[p,q](\frac{i}{\mu})&=&\mu^{6}w_{3}^{''}[-q,p](i\mu)+6\mu^{5}w_{3}^{'}[-q,p](i\mu)+6\mu^{4}w_{3}[-q,p](i\mu),  \\
w_{1}^{(3)}[p,q](\frac{i}{\mu}) &=& -\mu^{8}w_{3}^{(3)}[-q,p](i\mu)-12\mu^{7}w_{3}^{''}[-q,p](i\mu)-36\mu^{6}w_{3}^{'}[-q,p](i\mu) \\
&&-24\mu^{5}w_{3}[-q,p](i\mu), \\
w_{1}^{(4)}[p,q](\frac{i}{\mu}) &=& \mu^{10}w_{3}^{(4)}[-q,p](i\mu)+20\mu^{9}w_{3}^{(3)}[-q,p](i\mu)+120\mu^{8}w_{3}^{''}[-q,p](i\mu) \\
&&+240\mu^{7}w_{3}^{'}[-q,p](i\mu)+120\mu^{6}w_{3}[-q,p](i\mu). 
\end{eqnarray*}

\begin{proof}

Using lemmas \ref{transformationsvarthetapq} and \ref{transformationsvartheta234} we 
have
\begin{eqnarray*}
w_{1}[p,q](\frac{i}{\mu}) &=&-\frac{i}{2}\vartheta_{3}(\frac{i}{\mu})\vartheta_{4}(\frac{i}{\mu})\frac{\partial_{q}\vartheta[p,q+\frac{1}{2}](\frac{i}{\mu})}{e^{\pi ip}\vartheta[p,q](\frac{i}{\mu})} \\
&=&\mu^{2} (-\frac{1}{2}\vartheta_{3}(i\mu)\vartheta_{2}(i\mu)\frac{\partial_{q}\vartheta[-q+\frac{1}{2},p](i\mu)}{\vartheta[-q,p](i\mu)}) \\
&=&\mu^{2}w_{3}[-q,p](i\mu). 
\end{eqnarray*}
By taking a derivative with respect to $\mu$, it follows from the latter that 
\begin{eqnarray*}
w_{1}^{'}[p,q](\frac{i}{\mu})&=&\frac{d\mu}{d\frac{1}{\mu}}\partial_{\mu}w_{1}[p,q](\frac{i}{\mu})=-\mu^{2}\partial_{\mu}(\mu^{2}w_{3}[-q,p](i\mu)) \\
&=& -\mu^{2}(2\mu w_{3}[-q,p](i\mu)+\mu^{2}w_{3}^{'}[-q,p](i\mu)) \\
&=&-\mu^{4}w_{3}^{'}[-q,p](i\mu)-2\mu^{3}w_{3}[-q,p](i\mu). 
\end{eqnarray*}

\smallskip

One can then continue by taking another derivative to obtain
\begin{eqnarray*}
w_{1}^{''}[p,q](\frac{i}{\mu})&=&\frac{d\mu}{d\frac{1}{\mu}}\partial_{\mu}w_{1}^{'}[p,q](\frac{i}{\mu}) \\
&=&-\mu^{2}\partial_{\mu}(-\mu^{4}w_{3}^{'}[-q,p](i\mu)-2\mu^{3}w_{3}[-q,p](i\mu)) \\
&=&\mu^{6}w_{3}^{''}[-q,p](i\mu)+6\mu^{5}w_{3}^{'}[-q,p](i\mu)+6\mu^{4}w_{3}[-q,p](i\mu). 
\end{eqnarray*}
For the next higher derivatives we find that: 
\begin{eqnarray*}
w_{1}^{(3)}[p,q](\frac{i}{\mu}) &=& \frac{d\mu}{d\frac{1}{\mu}}\partial_{\mu}w_{1}^{''}[p,q](\frac{i}{\mu}) \\
&=&-\mu^{2}\partial_{\mu}(\mu^{6}w_{3}^{''}[-q,p](i\mu)+6\mu^{5}w_{3}^{'}[-q,p](i\mu)+6\mu^{4}w_{3}[-q,p](i\mu))\\
&=&-\mu^{8}w_{3}^{(3)}[-q,p](i\mu)-12\mu^{7}w_{3}^{''}[-q,p](i\mu)-36\mu^{6}w_{3}^{'}[-q,p](i\mu) \\
&&-24\mu^{5}w_{3}[-q,p](i\mu),
\end{eqnarray*}
\begin{eqnarray*}
&& w_{1}^{(4)}[p,q](\frac{i}{\mu}) \\
&& \qquad= -\mu^{2}\partial_{\mu}(-\mu^{8}w_{3}^{(3)}[-q,p](i\mu)-12\mu^{7}w_{3}^{''}[-q,p](i\mu)-36\mu^{6}w_{3}^{'}[-q,p](i\mu) \\
&&\qquad \quad -24\mu^{5}w_{3}[-q,p](i\mu)) \\
&&\qquad =\mu^{10}w_{3}^{(4)}[-q,p](i\mu)+20\mu^{9}w_{3}^{(3)}[-q,p](i\mu)+120\mu^{8}w_{3}^{''}[-q,p](i\mu) \\
&&\qquad \quad +240\mu^{7}w_{3}^{'}[-q,p](i\mu)+120\mu^{6}w_{3}[-q,p](i\mu). 
\end{eqnarray*}

\end{proof}

\end{lemma}

\smallskip

We record in the next two  lemmas the transformation rules for the functions 
$w_2[p, q]$, $w_3[p, q]$ and their derivatives up to order 4 with respect 
to the modular action $S(i \mu) = \frac{i}{\mu}$. Their proofs are presented 
in  Appendix \ref{transformationsw_j2appendix}, which are similar to 
the proof of Lemma \ref{transformationsw_12}. First, we treat $w_2[p, q]$ and its 
derivatives with respect to $\mu$ noting 
that the action of $S$ on $i \mu$ yields expressions in terms of 
$w_2[-q, p]$ and its derivatives.  

\smallskip

\begin{lemma} 
\label{transformationsw_22}
The function $w_2[p, q]$ and its derivatives up to order 4 with respect to $\mu$, $\Re(\mu) >0$, 
satisfy the following identities: 
\begin{eqnarray*}
w_{2}[p,q](\frac{i}{\mu}) &=&  \mu^{2}w_{2}[-q,p](i\mu), \\
w_{2}^{'}[p,q](\frac{i}{\mu}) &=& -\mu^{4}w_{2}^{'}[-q,p](i\mu)-2\mu^{3}w_{2}[-q,p](i\mu),\\
w_{2}^{''}[p,q](\frac{i}{\mu}) &=& \mu^{6}w_{2}^{''}[-q,p](i\mu)+6\mu^{5}w_{2}^{'}[-q,p](i\mu)+6\mu^{4}w_{2}[-q,p](i\mu), \\
w_{2}^{(3)}[p,q](\frac{i}{\mu}) &=& -\mu^{8}w_{2}^{(3)}[-q,p](i\mu)-12\mu^{7}w_{2}^{''}[-q,p](i\mu)-36\mu^{6}w_{2}^{'}[-q,p](i\mu) \\
&&-24\mu^{5}w_{2}[-q,p](i\mu), \\
w_{2}^{(4)}[p,q](\frac{i}{\mu}) &=&  \mu^{10}w_{2}^{(4)}[-q,p](i\mu)+20\mu^{9}w_{2}^{(3)}[-q,p](i\mu)+120\mu^{8}w_{2}^{''}[-q,p](i\mu) \\
&&+240\mu^{7}w_{2}^{'}[-q,p](i\mu) +120\mu^{6}w_{2}[-q,p](i\mu). 
\end{eqnarray*}
\begin{proof}
It is given in Appendix \ref{transformationsw_j2appendix}.  
\end{proof}

\end{lemma}

\smallskip

Now we present the transformation laws for $w_{3}[p,q]$ and its 
derivatives under the modular transformation $S(i \mu) = \frac{i}{\mu}$. 
We note that $w_{3}[p,q]$ behaves similarly to $w_1[p, q]$ under the action of $S$. 
In fact, the statement of the following lemma can be obtained from that of 
Lemma \ref{transformationsw_12}  by swapping the indices 1 and 3.   
 
\smallskip 

\begin{lemma}
\label{transformationsw_32} Assuming $\Re(\mu) >0$, 
the function $w_3[p, q]$ and its derivatives up to order 4 with respect to $\mu$
satisfy the following identities: 
\begin{eqnarray*}
w_{3}[p,q](\frac{i}{\mu}) &=& -\mu^{2}w_{1}[-q,p](i\mu), \\
w_{3}^{'}[p,q](\frac{i}{\mu}) &=&  \mu^{4}w_{1}^{'}[-q,p](i\mu)+2\mu^{3}w_{1}[-q,p](i\mu), \\
w_{3}^{''}[p,q](\frac{i}{\mu}) &=& -\mu^{6}w_{1}^{''}[-q,p](i\mu)-6\mu^{5}w_{1}^{'}[-q,p](i\mu)-6\mu^{4}w_{1}[-q,p](i\mu), \\
w_{3}^{(3)}[p,q](\frac{i}{\mu}) &=& \mu^{8}w_{1}^{(3)}[-q,p](i\mu)+12\mu^{7}w_{1}^{''}[-q,p](i\mu)+36\mu^{6}w_{1}^{'}[-q,p](i\mu) \\
&&+24\mu^{5}w_{1}[-q,p](i\mu), \\
w_{3}^{(4)}[p,q](\frac{i}{\mu})&=& -\mu^{10}w_{1}^{(4)}[-q,p](i\mu)-20\mu^{9}w_{1}^{(3)}[-q,p](i\mu)-120\mu^{8}w_{1}^{''}[-q,p](i\mu) \\
&&-240\mu^{7}w_{1}^{'}[-q,p](i\mu)-120\mu^{6}w_{1}[-q,p](i\mu).
\end{eqnarray*}
\begin{proof}
See Appendix \ref{transformationsw_j2appendix}.  
\end{proof}

\end{lemma}

\smallskip

We also need to know how the function $F[p, q]$ given in \eqref{two-parametric} 
and its derivatives, transform under the modular transformations on $i \mu$ 
in the upper half-plane. First, their properties with respect to the 
action of $T_1$ given by \eqref{T_1andSEq} are presented. 

\smallskip 

\begin{lemma}
\label{transformationsF}

For $\mu$ in the right half-plane $\Re(\mu) >0 $, the function 
$F[p, q]$ and its derivative of an arbitrary order $n \geq 1$ 
satisfy the following properties: 
\begin{eqnarray*}
F[p,q](i\mu+1) &=&  F[p,q+p+\frac{1}{2}](i\mu), \\
F^{(n)}[p,q](i\mu+1) &=&F^{(n)}[p,q+p+\frac{1}{2}](i\mu). 
\end{eqnarray*}

\begin{proof}

Using Lemma \ref{transformationsvarthetapq}, we have 
\begin{eqnarray*}
F[p,q](i\mu+1) &=& \frac{2}{\pi\Lambda}\cdot\left(\frac{\vartheta[p,q](i\mu+1)}{\partial_{q}\vartheta[p,q](i\mu+1)}\right)^{2}
= \frac{2}{\pi\Lambda}\cdot\left(\frac{\vartheta[p,q+p+\frac{1}{2}](i\mu)}{\partial_{q}\vartheta[p,q+p+\frac{1}{2}](i\mu)}\right)^{2} \\
&=& F[p,q+p+\frac{1}{2}](i\mu). 
\end{eqnarray*}
The identity for the derivatives of $F[p, q]$ follows easily from differentiating the latter with respect to $\mu$. 

\end{proof}

\end{lemma}

\smallskip

The following lemma gives the transformation properties of $F[p, q]$ and its derivatives with 
respect to the action $S$ given by \eqref{T_1andSEq} on the upper half-plane. 

\smallskip

\begin{lemma} \label{transformationsFSlemma}
If $\Re(\mu) > 0$, then 
\begin{eqnarray*}
F[p,q](\frac{i}{\mu}) &=& -\mu^{-2} F[-q,p](i\mu), \\
F^{'}[p,q](\frac{i}{\mu}) &=& F^{'}[-q,p](i\mu)-2\mu^{-1}F[-q,p](i\mu), \\
F^{''}[p,q](\frac{i}{\mu}) &=& -\mu^{2}F^{''}[-q,p](i\mu)+2\mu F^{'}[-q,p](i\mu)-2F[-q,p](i\mu), \\
F^{(3)}[p,q](\frac{i}{\mu}) &=& \mu^{4}F^{(3)}[-q,p](i\mu), \\
F^{(4)}[p,q](\frac{i}{\mu}) &=& -\mu^{6}F^{(4)}[-q,p](i\mu)-4\mu^{5}F^{(3)}[-q,p](i\mu).
\end{eqnarray*}

\begin{proof}
We can use Lemma \ref{transformationsvarthetapq} to write
\begin{eqnarray*}
F[p,q](\frac{i}{\mu}) &=& \frac{2}{\pi\Lambda}\cdot\left(\frac{\vartheta[p,q](\frac{i}{\mu})}{\partial_{q}\vartheta[p,q](\frac{i}{\mu})}\right)^{2}=-\mu^{-2}\frac{2}{\pi\Lambda}\cdot\left(\frac{\vartheta[-q,p](i\mu)}{\partial_{q}\vartheta[-q,p](i\mu)}\right)^{2} \\
&=& -\mu^{-2}\cdot F[-q,p](i\mu). 
\end{eqnarray*}
Taking the derivate of the latter, we have:  
\begin{eqnarray*}
F^{'}[p,q](\frac{i}{\mu}) &=& \frac{d\mu}{d\frac{1}{\mu}}\partial_{\mu}F[p,q](\frac{i}{\mu})=-\mu^{2}\partial_{\mu}(-\mu^{-2}\cdot F[-q,p](i\mu)) \\
&=&-\mu^{2}(-\mu^{-2}\cdot F^{'}[-q,p](i\mu)+2\mu^{-3}\cdot F[-q,p](i\mu)) \\
&=&F^{'}[-q,p](i\mu)-2\mu^{-1}F[-q,p](i\mu). 
\end{eqnarray*}

\smallskip

We then continue by taking another derivative to obtain
\begin{eqnarray*}
F^{''}[p,q](\frac{i}{\mu}) &=& \frac{d\mu}{d\frac{1}{\mu}}\partial_{\mu}F^{'}[p,q](\frac{i}{\mu})\\
&=&-\mu^{2}(F^{''}[-q,p](i\mu)-2\mu^{-1}F^{'}[-q,p](i\mu)+2\mu^{-2}F[-q,p](i\mu))\\
&=&-\mu^{2}F^{''}[-q,p](i\mu)+2\mu F^{'}[-q,p](i\mu)-2F[-q,p](i\mu). 
\end{eqnarray*}
Continuing this process, we find the properties of the higher derivative of $F[p, q]$: 
\begin{eqnarray*}
F^{(3)}[p,q](\frac{i}{\mu}) = \frac{d\mu}{d\frac{1}{\mu}}\partial_{\mu}F^{''}[p,q](\frac{i}{\mu})=\mu^{4}F^{(3)}[-q,p](i\mu),
\end{eqnarray*}
\begin{eqnarray*}
F^{(4)}[p,q](\frac{i}{\mu}) &=& \frac{d\mu}{d\frac{1}{\mu}}\partial_{\mu}F^{(3)}[p,q](\frac{i}{\mu})\\
 &=&-\mu^{6}F^{(4)}[-q,p](i\mu)-4\mu^{5}F^{(3)}[-q,p](i\mu).
\end{eqnarray*}

\end{proof}

\end{lemma}

\smallskip

\subsection{The case of the one-parametric family of metrics}

We now turn  to the one-parameteric case \eqref{one-parametric},
for which, similarly to the case of the two-parametric case \eqref{two-parametric} 
analysed in Subsection \ref{ModularPropsTwoParametricSubSec}, the necessary modular 
properties can be investigated by using the lemmas proved in the beginning of this section. 
Since the proofs are based on direct calculations combined with making use of the lemmas 
\ref{transformationsvarthetapq} and \ref{transformationsvartheta234}, 
we just present the statements.

\smallskip

First, let us deal with the functions $w_j[q_0]$ given in \eqref{one-parametric}. With respect to the 
modular transformation $T_1$ given by \eqref{T_1andSEq} acting 
on $i \mu$ in the upper half-plane, these functions and their derivatives 
satisfy the following properties. 

\smallskip

\begin{lemma}\label{transfsw_jT_2}
The functions $w_j[q_0]$ and their derivative of  an arbitrary order $n \geq 0$ with respect to $\mu$ 
satisfy these properties when $\Re(\mu)>0$: 
\begin{eqnarray*}
w_{1}^{(n)}[q_0](i\mu+1)&=&w_{1}^{(n)}[q_0-i](i\mu),  \\
w_{2}^{(n)}[q_0](i\mu+1)&=&w_{3}^{(n)}[q_0-i](i\mu),  \\
w_{3}^{(n)}[q_0](i\mu+1)&=&w_{2}^{(n)}[q_0-i](i\mu).
\end{eqnarray*}

\end{lemma}

\smallskip

The next three lemma presents the transformation properties of the functions $w_j[q_0]$ 
and their derivatives up to order 4 with 
respect to the modular action of $S$ given by \eqref{T_1andSEq} on $i \mu$. 

\smallskip

\begin{lemma}
\label{w_1q_0SLem}
The function $w_1[q_0]$ and its derivatives up to order 4 with respect to $\mu$ satisfy 
the following identities provided that $\Re(\mu) >0$: 
\begin{eqnarray*}
w_{1}[q_{0}](\frac{i}{\mu}) &=& -\mu^{2}w_{3}[\frac{1}{q_{0}}](i\mu), \\
w_{1}^{'}[q_{0}](\frac{i}{\mu}) &=& \mu^{4}w_{3}^{'}[\frac{1}{q_{0}}](i\mu)+2\mu^{3}w_{3}[\frac{1}{q_{0}}](i\mu), \\
w_{1}^{''}[q_{0}](\frac{i}{\mu}) &=&
-\mu^{6}w_{3}^{''}[\frac{1}{q_{0}}](i\mu)-6\mu^{5}w_{3}^{'}[\frac{1}{q_{0}}](i\mu)-6\mu^{4}w_{3}[\frac{1}{q_{0}}](i\mu), \\
w_{1}^{(3)}[q_{0}](\frac{i}{\mu})
&=&\mu^{8}w_{3}^{(3)}[\frac{1}{q_{0}}](i\mu)+12\mu^{7}w_{3}^{''}[\frac{1}{q_{0}}](i\mu)+36\mu^{6}w_{3}^{'}[\frac{1}{q_{0}}](i\mu) \\
&&  +24\mu^{5}w_{3}[\frac{1}{q_{0}}](i\mu), \\
w_{1}^{(4)}[q_{0}](\frac{i}{\mu})&=&-\mu^{10}w_{3}^{(4)}[\frac{1}{q_{0}}](i\mu)-20\mu^{9}w_{3}^{(3)}[\frac{1}{q_{0}}](i\mu)-120\mu^{8}w_{3}^{''}[\frac{1}{q_{0}}](i\mu) \\
&&-240\mu^{7}w_{3}^{'}[\frac{1}{q_{0}}](i\mu)-120\mu^{6}w_{3}[\frac{1}{q_{0}}](i\mu). 
\end{eqnarray*}

\end{lemma}

\smallskip

While in the latter lemma for $w_1[q_0]$, the expressions on the right hand sides of the identities 
involve $w_3[\frac{1}{q_0}]$, in the following lemma we observe that for $w_2[q_0]$ similar 
properties hold, however, the index 2 does not change in the sense that the properties are expressed 
in terms of $w_2[\frac{1}{q_0}]$. 

\smallskip

\begin{lemma} \label{w_2q_0SLem}
When $\Re(\mu) >0$, the function $w_2[q_0]$ and its derivatives up to order 4 with respect to $\mu$ satisfy 
the following properties:
\begin{eqnarray*}
w_{2}[q_{0}](\frac{i}{\mu})&=&-\mu^{2}w_{2}[\frac{1}{q_{0}}](i\mu),\\
w_{2}^{'}[q_{0}](\frac{i}{\mu}) &=&
\mu^{4}w_{2}^{'}[\frac{1}{q_{0}}](i\mu)+2\mu^{3}w_{2}[\frac{1}{q_{0}}](i\mu),\\
w_{2}^{''}[q_{0}](\frac{i}{\mu})&=&
-\mu^{6}w_{2}^{''}[\frac{1}{q_{0}}](i\mu)-6\mu^{5}w_{2}^{'}[\frac{1}{q_{0}}](i\mu)-6\mu^{4}w_{2}[\frac{1}{q_{0}}](i\mu),\\
w_{2}^{(3)}[q_{0}](\frac{i}{\mu})&=&
\mu^{8}w_{2}^{(3)}[\frac{1}{q_{0}}](i\mu)+12\mu^{7}w_{2}^{''}[\frac{1}{q_{0}}](i\mu)+36\mu^{6}w_{2}^{'}[\frac{1}{q_{0}}](i\mu)\\
&&+24\mu^{5}w_{2}[\frac{1}{q_{0}}](i\mu),\\
w_{2}^{(4)}[q_{0}](\frac{i}{\mu})&=&
-\mu^{10}w_{2}^{(4)}[\frac{1}{q_{0}}](i\mu)-20\mu^{9}w_{2}^{(3)}[\frac{1}{q_{0}}](i\mu)-120\mu^{8}w_{2}^{''}[\frac{1}{q_{0}}](i\mu) \\
&&-240\mu^{7}w_{2}^{'}[\frac{1}{q_{0}}](i\mu)-120\mu^{6}w_{2}[\frac{1}{q_{0}}](i\mu). 
\end{eqnarray*}

\end{lemma}

\smallskip

The function $w_3[q_0]$ and its derivatives behave similarly to the case of $w_1[q_0]$: 
the statement given in the following lemma can be obtained from swapping the indices 1 and 3 
in Lemma \ref{w_1q_0SLem}. 

\smallskip

\begin{lemma} \label{w_3q_0SLem}
The function $w_3[q_0]$ and its derivatives up to order 4 with respect to $\mu$ satisfy 
the following identities provided that $\Re(\mu) >0$:
\begin{eqnarray*}
w_{3}[q_{0}](\frac{i}{\mu})&=&-\mu^{2}w_{1}[\frac{1}{q_{0}}](i\mu), \\
w_{3}^{'}[q_{0}](\frac{i}{\mu})&=&
\mu^{4}w_{1}^{'}[\frac{1}{q_{0}}](i\mu)+2\mu^{3}w_{1}[\frac{1}{q_{0}}](i\mu), \\
w_{3}^{''}[q_{0}](\frac{i}{\mu})&=&
-\mu^{6}w_{1}^{''}[\frac{1}{q_{0}}](i\mu)-6\mu^{5}w_{1}^{'}[\frac{1}{q_{0}}](i\mu)-6\mu^{4}w_{1}[\frac{1}{q_{0}}](i\mu), \\
w_{3}^{(3)}[q_{0}](\frac{i}{\mu})&=&
\mu^{8}w_{1}^{(3)}[\frac{1}{q_{0}}](i\mu)+12\mu^{7}w_{1}^{''}[\frac{1}{q_{0}}](i\mu)+36\mu^{6}w_{1}^{'}[\frac{1}{q_{0}}](i\mu) \\
&&+24\mu^{5}w_{1}[\frac{1}{q_{0}}](i\mu), \\
w_{3}^{(4)}[q_{0}](\frac{i}{\mu})&=&
-\mu^{10}w_{1}^{(4)}[\frac{1}{q_{0}}](i\mu)-20\mu^{9}w_{1}^{(3)}[\frac{1}{q_{0}}](i\mu)-120\mu^{8}w_{1}^{''}[\frac{1}{q_{0}}](i\mu) \\
&&-240\mu^{7}w_{1}^{'}[\frac{1}{q_{0}}](i\mu)-120\mu^{6}w_{1}[\frac{1}{q_{0}}](i\mu). 
\end{eqnarray*}

\end{lemma}

\smallskip

Finally we present the modular transformation properties of the function $F[q_0]$ given in 
\eqref{one-parametric} and its derivatives. 

\smallskip

\begin{lemma}
\label{transformationsFlemma}
 Provided that the complex number $\mu$ belongs to the 
right half-plane $\Re(\mu) > 0$, the function $F[q_0]$ and its derivatives 
with respect to $\mu$ satisfy the following identities: 
\begin{align*}
&F[q_{0}](i\mu+1)= F[q_{0}-i](i\mu), &
 F^{'}[q_{0}](i\mu+1)= F^{'}[q_{0}-i](i\mu), \\
& F^{''}[q_{0}](i\mu+1)=F^{''}[q_{0}-i](i\mu), &
F^{(n)}[q_{0}](i\mu)=0, \qquad n \geq 3, \\
& F[q_{0}](\frac{i}{\mu})=q_{0}^{2}\mu^{-2}F[\frac{1}{q_{0}}](i\mu), &
F^{'}[q_{0}](\frac{i}{\mu}) =q_{0}\mu^{-1}F^{'}[\frac{1}{q_{0}}](i\mu).
\end{align*}

\end{lemma}

\smallskip

\section{Modular properties of $\tilde a_0,$ $\tilde a_2$ and $\tilde a_4$ by direct calculations}
\label{a_0a_2a_4Sec}

\smallskip

Our aim in this section is to investigate modular properties 
of the Seeley-de Witt coefficient $\tilde a_0$, $\tilde a_2$,  and $\tilde a_4$,  
which are respectively given by \eqref{a_0Eq}, \eqref{a_2Eq}, and in 
Appendix \ref{fulla_4appendix},  
in the case of the gravitational instantons parametrized by 
\eqref{two-parametric} and \eqref{one-parametric}. The 
terms $\tilde a_{2n}$ appear in general in the asymptotic expansion 
\eqref{ExpAsympConformalEq} of the heat kernel of $\tilde D^2$, where 
$\tilde D$ with the expression \eqref{DiracConfBianchiIXEq} is the Dirac operator of a 
time dependent conformal perturbation of a triaxial Bianchi IX metric, given 
by \eqref{ConformalBianchiIXMetricEq1}. The latter serves as a general 
form of Bianchi IX gravitational instantons, as we explained and 
provided references about it in Section \ref{InstantonsSec}. 
Since the expressions in \eqref{two-parametric} are parametrized 
by a pair of parameters $[p, q]$, we denote the corresponding Seeley-de 
Witt coefficients by $\tilde a_{2n}[p, q]$. Similarly, the terms $\tilde a_{2n}$ 
associated with the one-parametric case \eqref{one-parametric} with the parameter 
$q_0$ will be denoted by $\tilde a_{2n}[q_0]$.

\smallskip

We begin with substituting the solutions \eqref{two-parametric} and 
\eqref{one-parametric} for $w_1$, $w_2$, $w_3$ and $F$, in the explicit expressions for 
$\tilde a_0$, $\tilde a_2$ and $\tilde a_4$.  It is evident from the definitions 
provided by \eqref{ThetawithCharEq}, \eqref{varthetapqEq} and  \eqref{varthetasEq} 
that we are allowed to take $\mu$ in the right half-plane 
in complex numbers so that $i \mu$ belongs to $\mathbb{H}$, the upper half-plane. 
Therefore, in this section, using the modular identities explored in Section  \ref{ArithmeticsofInstantonsSec}, 
we investigate transformations of the terms $\tilde a_0[p, q](i \mu)$, $\tilde a_2[p, q](i \mu)$, 
$\tilde a_4[p, q](i \mu)$, $\tilde a_0[q_0](i \mu)$, $\tilde a_2[q_0](i \mu)$ and $\tilde a_4[q_0](i \mu)$ 
under the modular actions $S$ and $T_1$ given by \eqref{T_1andSEq} on $i \mu \in \mathbb{H}$. Let us start studying the transformations associated with $T_1$.

\smallskip

\begin{theorem} \label{modualra024T_1thm}
Let $\mu$ be a complex number in the right half-plane $\Re(\mu)>0$. We have
\[
\tilde{a}_{0}[p,q](i\mu+1)= \tilde{a}_{0}[p,q+p+\frac{1}{2}](i\mu),
\]
\[
\tilde a_{2}[p,q](i\mu+1)= a_{2}[p,q+p+\frac{1}{2}](i\mu),
\]
\[
\tilde a_{4}[p,q](i\mu+1)=\tilde a_{4}[p,q+p+\frac{1}{2}](i\mu). 
\]

\begin{proof}

According to Lemma 5.5 and Lemma 5.9, so far as the functions $w_1(i\mu)$, $w_2(i\mu)$, $w_3(i\mu)$ and $F(i\mu)$ are concerned, the transformation $i\mu\mapsto i\mu+1$ is equivalent to the transformation $(p,q)\mapsto(p,q+p+\frac{1}{2})$ followed by the exchange of $w_2$ and $w_3$. 

\smallskip

Also, from the explicit expressions of $\tilde{a}_0[p,q](i\mu)$ and $\tilde{a}_2[p,q](i\mu)$ presented in (23) and (24), as well as the expression of $\tilde{a}_4[p,q](i\mu)$ shown in Appendix B, we know that  $\tilde{a}_0[p,q](i\mu)$,  $\tilde{a}_2[p,q](i\mu)$, and  $\tilde{a}_4[p,q](i\mu)$ are invariant under exchanging $w_2$ and $w_3$, so our claim is confirmed. 

\smallskip


\end{proof}

\end{theorem}

\smallskip

Moreover, the terms $\tilde a_0[p, q]$, $\tilde a_2[p, q]$ and $\tilde a_4[p, q]$ 
satisfy modular transformation properties with respect to the action of $S$ given in 
\eqref{T_1andSEq} on $i \mu \in \mathbb{H}$. 

\smallskip

\begin{theorem} \label{modualra024Sthm}
Assuming that $\Re(\mu) >0$, we have
\[
\tilde{a}_{0}[p,q](\frac{i}{\mu})=-\mu^{2} \tilde{a}_{0}[-q,p](i\mu),
\]
\[
\tilde a_{2}[p,q](\frac{i}{\mu})=-\mu^{2} \tilde a_{2}[-q,p](i\mu),
\]
\[
\tilde a_{4}[p,q](\frac{i}{\mu})=-\mu^{2} \tilde a_{4}[-q,p](i\mu).
\]

\begin{proof}
The identities can be seen to hold by applying directly lemmas \ref{transformationsw_12}, 
\ref{transformationsw_22}, \ref{transformationsw_32} and \ref{transformationsFSlemma} 
to the explicit expressions for $\tilde a_0$, $\tilde a_2$ and $\tilde a_4$ given 
respectively by \eqref{a_0Eq}, \eqref{a_2Eq} and in Appendix \ref{fulla_4appendix}. In fact, for the first term we can write 
\begin{eqnarray*}
\tilde a_{0}[p,q](\frac{i}{\mu}) &=& -4\mu^{2}F^{2}[-q,p](i\mu)w_{1}[-q,p](i\mu)w_{2}[-q,p](i\mu)w_{3}[-q,p](i\mu) \\
&=&-\mu^{2} \tilde a_{0}[-q,p](i\mu),
\end{eqnarray*}

\smallskip

For the next term, first we just consider its expression, replace $i \mu$ by $S(i \mu)=i/\mu$ and 
use the mentioned lemmas to write:  
\begin{eqnarray*}
&&\tilde a_{2}[p,q](\frac{i}{\mu}) \\
&& \quad =\mu^{2}\frac{F[-q,p](i\mu)}{3}\Big (w_{1}^{2}[-q,p](i\mu)+w_{2}^{2}[-q,p](i\mu)+w_{3}^{2}[-q,p](i\mu) \Big ) \\
&& \qquad -\frac{\mu^{2}F[-q,p](i\mu)}{6}\Big (\frac{w_{1}^{2}[-q,p](i\mu)w_{2}^{2}[-q,p](i\mu)}{w_{3}^{2}[-q,p](i\mu)}+\frac{w_{1}^{2}[-q,p](i\mu)w_{3}^{2}[-q,p](i\mu)}{w_{2}^{2}[-q,p](i\mu)} \\
&& \qquad +\frac{w_{2}^{2}[-q,p](i\mu)w_{3}^{2}[-q,p](i\mu)}{w_{1}^{2}[-q,p](i\mu)} \Big)
+\frac{\mu^{2}F[-q,p](i\mu)}{6}\Big (\frac{w_{3}^{'2}[-q,p](i\mu)}{w_{3}^{2}[-q,p](i\mu)}
\end{eqnarray*}
\begin{eqnarray*}
&& \qquad+\frac{w_{2}^{'2}[-q,p](i\mu)}{w_{2}^{2}[-q,p](i\mu)} +\frac{w_{1}^{'2}[-q,p](i\mu)}{w_{1}^{2}[-q,p](i\mu)}\Big )+\frac{\mu F[-q,p](i\mu)}{3}\Big (\frac{2w_{3}^{'}[-q,p](i\mu)}{w_{3}[-q,p](i\mu)} \\
&& \qquad +\frac{2w_{2}^{'}[-q,p](i\mu)}{w_{2}[-q,p](i\mu)}+\frac{2w_{1}^{'}[-q,p](i\mu)}{w_{1}[-q,p](i\mu)} \Big )+2F[-q,p](i\mu)
\end{eqnarray*}
\begin{eqnarray*}
&& \qquad +\frac{\mu^{2}F[-q,p](i\mu)}{3}\Big (\frac{w_{3}^{'}[-q,p](i\mu)w_{2}^{'}[-q,p](i\mu)}{w_{3}[-q,p](i\mu)w_{2}[-q,p](i\mu)}+\frac{w_{3}^{'}[-q,p](i\mu)w_{1}^{'}[-q,p](i\mu)}{w_{3}[-q,p](i\mu)w_{1}[-q,p](i\mu)}
\end{eqnarray*}
\begin{eqnarray*}
&& \qquad +\frac{w_{1}^{'}[-q,p](i\mu)w_{2}^{'}[-q,p](i\mu)}{w_{1}[-q,p](i\mu)w_{2}[-q,p](i\mu)}\Big ) +\frac{\mu F[-q,p](i\mu)}{3}\Big (\frac{4w_{1}^{'}[-q,p](i\mu)}{w_{1}[-q,p](i\mu)} \\
&& \qquad +\frac{4w_{2}^{'}[-q,p](i\mu)}{w_{2}[-q,p](i\mu)} +\frac{4w_{3}^{'}[-q,p](i\mu)}{w_{3}[-q,p](i\mu)}\Big )+4F[-q,p](i\mu) \\
&& \qquad -\frac{\mu^{2}F[-q,p](i\mu)}{3}\Big (\frac{w_{3}^{''}[-q,p](i\mu)}{w_{3}[-q,p](i\mu)}+\frac{w_{2}^{''}[-q,p](i\mu)}{w_{2}[-q,p](i\mu)}+\frac{w_{1}^{''}[-q,p](i\mu)}{w_{1}[-q,p](i\mu)}\Big )
\end{eqnarray*}
\begin{eqnarray*}
&& \qquad -\mu F[-q,p](i\mu) \Big (\frac{2w_{3}^{'}[-q,p](i\mu)}{w_{3}[-q,p](i\mu)}+\frac{2w_{2}^{'}[-q,p](i\mu)}{w_{2}[-q,p](i\mu)}+\frac{2w_{1}^{'}[-q,p](i\mu)}{w_{1}[-q,p](i\mu)}\Big ) \\
&& \qquad -6F[-q,p](i\mu)+\mu^{2}\frac{F^{'2}[-q,p](i\mu)}{2F[-q,p](i\mu)}-\mu^{2}F^{''}[-q,p](i\mu) \\
&& \quad =-\mu^{2}\tilde a_{2}[-q,p](i\mu). 
\end{eqnarray*}

\end{proof}

\end{theorem}

\smallskip

Now we turn our focus to the modular properties of the terms $\tilde a_0[q_0]$, $\tilde a_2[q_0]$ and $\tilde a_0[q_0]$ associated with the one parametric case \eqref{one-parametric}. 

\begin{theorem} \label{modualra024Soneparathm}
If $\Re(\mu) >0 $, we have 
\begin{align*}
\tilde a_{0}[q_{0}](i\mu+1)&= \tilde a_{0}[q_{0}-i](i\mu), & \tilde a_{0}[q_{0}](\frac{i}{\mu})&=-q_{0}^{4}\mu^{2} \tilde a_{0}[\frac{1}{q_{0}}](i\mu), \\
\tilde a_{2}[q_{0}](i\mu+1)&= \tilde a_{2}[q_{0}-i](i\mu), & \tilde a_{2}[q_{0}](\frac{i}{\mu})&=q_{0}^{2}\mu^{2} \tilde a_{2}[\frac{1}{q_{0}}](i\mu), \\
\tilde a_{4}[q_{0}](i\mu+1)&= \tilde a_{4}[q_{0}-i](i\mu), &
\tilde a_{4}[q_{0}](\frac{i}{\mu})&=-\mu^{2}\tilde a_{4}[\frac{1}{q_{0}}](i\mu).
\end{align*}

\begin{proof}
The identities follow directly from applying  lemmas \ref{transfsw_jT_2}, 
\ref{w_1q_0SLem}, \ref{w_2q_0SLem}, \ref{w_3q_0SLem} and \ref{transformationsFlemma} to the 
explicit expressions for $\tilde a_0[q_0]$, $\tilde a_2[q_0]$ and $\tilde a_4[q_0]$. 
\end{proof}

\end{theorem}

\smallskip

The direct calculations carried out in this section show that the first three Seeley-de Witt 
coefficients behave similarly to vector-valued modular forms \cite{EicZag} since the modular 
transformations act accordingly on the parameters of the metrics. This strong evidence indicates 
that there should be an interesting connection between the Dirac operators of the metrics 
that are related to each other by the modular transformations. Indeed, by taking a close 
look at the spectral properties of these Dirac operators, we prove in Section \ref{a_2nSec} 
that the modular transformation properties proved for the first three terms hold for 
all of the terms $\tilde a_{2n}$.

\smallskip

\section{Modular properties of $\tilde a_{2n}$ using identities for Dirac operators}
\label{a_2nSec}

\smallskip

Motivated by the modular transformation properties presented in theorems \ref{modualra024T_1thm}, 
\ref{modualra024Sthm} and \ref{modualra024Soneparathm} 
for the terms $\tilde a_0$, $\tilde a_2$ and 
$\tilde a_4$, the main goal in this section is to prove that all of the Seeley-de Witt $\tilde a_{2n}$ satisfy the same 
properties. Clearly, this cannot be achieved by resorting to direct calculations. 
However, the direct calculations performed in Section \ref{a_0a_2a_4Sec} for the first three terms 
reveal that the effect of the modular transformations on the Seeley-de Witt coefficients is 
reflected in a certain way on the parameters of the metric, and by studying the corresponding 
Dirac operators carefully in this section we can prove that the transformation properties hold for the general terms. 

\smallskip

Note that the the gravitational instantons parametrized by \eqref{two-parametric} and \eqref{one-parametric} 
are given in terms of theta functions and their derivatives. Thus in the two-parametric case 
\eqref{two-parametric} with the parameters $(p, q)$ 
as well as in the one-parametric case \eqref{one-parametric} with the parameter $q_0$, 
we are allowed to consider an extension of the definitions to $p,q, q_0\in\mathbb{C}$ with $i\mu\in\mathbb{H}$, excluding possible zeros of the theta functions and their derivatives that appear in the denominators. Thus with $(p,q)$ and $q_{0}$ fixed, we may consider the operators $\tilde{D}^{2}[p,q]$
and $\tilde{D}^{2}[\frac{1}{q_{0}}]$ acting on a spin bundle over
the base manifold $M_{0}=I_{(a,b)}\times\mathbb{S}^{3}$, where $I_{(a,b)}$
is an arbitrary horizontal path in the upper-half complex plane $\mathbb{H}$. We may also consider the operators
$\tilde{D}^{2}[p,q+p+\frac{1}{2}]$ and $\tilde{D}^{2}[q_{0}-i]$ acting on
$M_{1}=(a+i,b+i)\times\mathbb{S}^{3}$, and the operators $\tilde{D}^{2}[-q,p]$
and $\tilde{D}^{2}[\frac{1}{q_{0}}]$ acting on $M_{2}=(\frac{1}{b},\frac{1}{a})\times\mathbb{S}^{3}$.

\smallskip

In the following theorem, we prove the isospectrality of the Dirac operators corresponding to the 
metrics whose parameters are related to each other via the effect of the modular 
transformations on the Seeley-de Witt coefficients $\tilde a_0$, $\tilde a_2$ and $\tilde a_4$, 
presented in theorems \ref{modualra024T_1thm} and \ref{modualra024Sthm}. 

\smallskip

\begin{theorem} \label{IsoSpecDiracs2paraThm}

For the two-parameter family of metrics parametrized by $(p,q)$,
the operators $\tilde{D}^2[p,q]$, $\tilde{D}^2[p,q+p+\frac{1}{2}]$ and
$\tilde{D}^2[-q,p]$ are isospectral on spin bundles over $M_0$, $M_1$ and $M_2$ 
respectively. More precisely, under the isospectral transformations $(p,q)\mapsto(p,q+p+\frac{1}{2})$ and $(p,q)\mapsto(-q,p)$, an eigenspinor ${u_n(\mu)}$ corresponding to a particular eigenvalue $\lambda_n^2[p,q]$ of $\tilde{D}^2$ transforms as $u_n(\mu)\mapsto u_n(\mu-i)$ and $u_n(\mu)\mapsto -\gamma^1u_n(\frac{1}{\mu})$ respectively, where the spatial dependence of $u_n$ is invariant and suppressed. The transformations of the eigenspinors define a bijection between the spaces of eigenspinors of the corresponding operators.

\end{theorem}

\begin{proof}

For the two-parametric family of metrics, suppose $u_{n}(\mu,\eta,\phi,\psi)$
is a section of the spin bundle such that $\tilde{D}[p,q]u_{n}(\mu,\eta,\phi,\psi)=\lambda_{n}u_{n}(\mu,\eta,\phi,\psi)$. For simplicity, let us suppress writing the dependence of $u_{n}$ on $\eta$, $\phi$, and $\psi$
and assume that it is understood.  Using the explicit expression \eqref{DiracConfBianchiIXEq} for the Dirac operator we 
have:  
{\small 
\begin{eqnarray*}
&& \left(-\tilde{D}[-q,p]\left(-\gamma^{0}u_{n}(\frac{1}{\mu})\right)\right)\mid_{\mu=\mu_{0}} = 
\end{eqnarray*}
\[
   -(F[-q,p](i\mu_{0})w_{1}[-q,p](i\mu_{0})w_{2}[-q,p](i\mu_{0})  w_{3}[-q,p](i\mu_{0}))^{-\frac{1}{2}}
 \gamma^{0}\partial_{\text{\ensuremath{\mu}}}\left(-\gamma^{0}u_{n}(\frac{1}{\mu})\right)\mid_{\mu=\mu_{0}} 
\]
\[
+ \sin\psi\cdot\left(\frac{F[-q,p](i\mu_{0})w_{2}[-q,p](i\mu_{0})w_{3}[-q,p](i\mu_{0})}{w_{1}[-q,p](i\mu_{0})}\right)^{-\frac{1}{2}} 
 \gamma^{1}\partial_{\eta}\left(-\gamma^{0}u_{n}(\frac{1}{\mu})\right)\mid_{\mu=\mu_{0}} 
\]
\[
-\cos\psi\cdot\left(\frac{F[-q,p](i\mu_{0})w_{1}[-q,p](i\mu_{0})w_{3}[-q,p](i\mu_{0})}{w_{2}[-q,p](i\mu_{0})}\right)^{-\frac{1}{2}} 
\gamma^{2}\partial_{\eta}\left(-\gamma^{0}u_{n}(\frac{1}{\mu})\right)\mid_{\mu=\mu_{0}} 
\]
\[
  -\cos\psi\csc\eta\cdot\left(\frac{F[-q,p](i\mu_{0})w_{2}[-q,p](i\mu_{0})w_{3}[-q,p](i\mu_{0})}{w_{1}[-q,p](i\mu_{0})}\right)^{-\frac{1}{2}} 
\gamma^{1}\partial_{\phi}\left(-\gamma^{0}u_{n}(\frac{1}{\mu})\right)\mid_{\mu=\mu_{0}} 
\]
\[
-\sin\psi\csc\eta\cdot\left(\frac{F[-q,p](i\mu_{0})w_{1}[-q,p](i\mu_{0})w_{3}[-q,p](i\mu_{0})}{w_{2}[-q,p](i\mu_{0})}\right)^{-\frac{1}{2}}  
 \gamma^{2}\partial_{\phi}\left(-\gamma^{0}u_{n}(\frac{1}{\mu})\right)\mid_{\mu=\mu_{0}} 
\]
\[
+\cos\psi\cot\eta\cdot\left(\frac{F[-q,p](i\mu_{0})w_{2}[-q,p](i\mu_{0})w_{3}[-q,p](i\mu_{0})}{w_{1}[-q,p](i\mu_{0})}\right)^{-\frac{1}{2}}\gamma^{1}\partial_{\psi}
\left(-\gamma^{0}u_{n}(\frac{1}{\mu})\right)\mid_{\mu=\mu_{0}} 
\]
\[
+\sin\psi\cot\eta\cdot\left(\frac{F[-q,p](i\mu_{0})w_{1}[-q,p](i\mu_{0})w_{3}[-q,p](i\mu_{0})}{w_{2}[-q,p](i\mu_{0})}\right)^{-\frac{1}{2}}
 \gamma^{2}\partial_{\psi}\left(-\gamma^{0}u_{n}(\frac{1}{\mu})\right)\mid_{\mu=\mu_{0}}
\]
\[
-\left(\frac{F[-q,p](i\mu_{0})w_{1}[-q,p](i\mu_{0})w_{2}[-q,p](i\mu_{0})}{w_{3}[-q,p](i\mu_{0})}\right)^{-\frac{1}{2}}\gamma^{3}\partial_{\psi}\left(-\gamma^{0}u_{n}(\frac{1}{\mu})\right)\mid_{\mu=\mu_{0}} 
\]
\[
  -\frac{1}{4}\frac{\frac{w_{1}^{'}[-q,p](i\mu_{0})}{w_{1}[-q,p](i\mu_{0})}+\frac{w_{2}^{'}[-q,p](i\mu_{0})}{w_{2}[-q,p](i\mu_{0})}+\frac{w_{3}^{'}[-q,p](i\mu_{0})}{w_{3}[-q,p](i\mu_{0})}+3\frac{F^{'}[-q,p](i\mu_{0})}{F[-q,p](i\mu_{0})}}{\left(F[-q,p](i\mu_{0})w_{1}[-q,p](i\mu_{0})w_{2}[-q,p](i\mu_{0})w_{3}[-q,p](i\mu_{0})\right)^{\frac{1}{2}}} 
 \gamma^{0}\left(-\gamma^{0}u_{n}(\frac{1}{\mu_{0}})\right) 
\]
\[ +\frac{1}{4}\left(\frac{w_{1}[-q,p](i\mu_{0})w_{2}[-q,p](i\mu_{0})w_{3}[-q,p](i\mu_{0})}{F[-q,p](i\mu_{0})}\right)^{\frac{1}{2}} \times 
\]
\[
 \quad  \left(\frac{1}{w_{1}^{2}[-q,p](i\mu_{0})}+\frac{1}{w_{2}^{2}[-q,p](i\mu_{0})}+\frac{1}{w_{3}^{2}[-q,p](i\mu_{0})}\right) \gamma^{1}\gamma^{2}\gamma^{3}\left(-\gamma^{0}u_{n}(\frac{1}{\mu_{0}})\right). 
\]
} 
\smallskip

Now we can use in the latter lemmas \ref{transformationsw_12}, \ref{transformationsw_22}, \ref{transformationsw_32} and  
\ref{transformationsFSlemma}, which show how the modular transformation $S(i \mu) = i/ \mu$ affects 
the functions $w_j[p, q]$ and $F[p, q]$. This yields 
{\small 
\begin{eqnarray*}
\left(-\tilde{D}[-q,p]\left(-\gamma^{0}u_{n}(\frac{1}{\mu})\right)\right)\mid_{\mu=\mu_{0}} =
\end{eqnarray*}
\[
-(F[p,q](\frac{i}{\mu_{0}})w_{1}[p,q](\frac{i}{\mu_{0}})w_{2}[p,q](\frac{i}{\mu_{0}})w_{3}[p,q](\frac{i}{\mu_{0}}))^{-\frac{1}{2}}  \left(-\partial_{\text{\ensuremath{\frac{1}{\mu}}}}\gamma^{0}\left(-\gamma^{0}u_{n}(\frac{1}{\mu})\right)\mid_{\mu=\mu_{0}}\right) 
\]
\[
+\sin\psi\cdot\left(\frac{F[p,q](\frac{i}{\mu_{0}})w_{2}[p,q](\frac{i}{\mu_{0}})w_{3}[p,q](\frac{i}{\mu_{0}})}{w_{1}[p,q](\frac{i}{\mu_{0}})}\right)^{-\frac{1}{2}}\gamma^{0}\gamma^{1}\gamma^{0}\partial_{\eta}\left(-\gamma^{0}u_{n}(\frac{1}{\mu})\right)\mid_{\mu=\mu_{0}} 
\]
\[
 -\cos\psi\cdot\left(\frac{F[p,q](\frac{i}{\mu_{0}})w_{1}[p,q](\frac{i}{\mu_{0}})w_{3}[p,q](\frac{i}{\mu_{0}})}{w_{2}[p,q](\frac{i}{\mu_{0}})}\right)^{-\frac{1}{2}}\gamma^{0}\gamma^{2}\gamma^{0}\partial_{\eta}\left(-\gamma^{0}u_{n}(\frac{1}{\mu})\right)\mid_{\mu=\mu_{0}} 
\]
\[ 
-\cos\psi\csc\eta\cdot\left(\frac{F[p,q](\frac{i}{\mu_{0}})w_{2}[p,q](\frac{i}{\mu_{0}})w_{3}[p,q](\frac{i}{\mu_{0}})}{w_{1}[p,q](\frac{i}{\mu_{0}})}\right)^{-\frac{1}{2}}  \gamma^{0}\gamma^{1}\gamma^{0}\partial_{\phi}\left(-\gamma^{0}u_{n}(\frac{1}{\mu})\right)\mid_{\mu=\mu_{0}}
\]
\[
-\sin\psi\csc\eta\cdot\left(\frac{F[p,q](\frac{i}{\mu_{0}})w_{1}[p,q](\frac{i}{\mu_{0}})w_{3}[p,q](\frac{i}{\mu_{0}})}{w_{2}[p,q](\frac{i}{\mu_{0}})}\right)^{-\frac{1}{2}}\gamma^{0}\gamma^{2}\gamma^{0}\partial_{\phi}\left(-\gamma^{0}u_{n}(\frac{1}{\mu})\right)\mid_{\mu=\mu_{0}}
\]
\[
+\cos\psi\cot\eta\cdot\left(\frac{F[p,q](\frac{i}{\mu_{0}})w_{2}[p,q](\frac{i}{\mu_{0}})w_{3}[p,q](\frac{i}{\mu_{0}})}{w_{1}[p,q](\frac{i}{\mu_{0}})}\right)^{-\frac{1}{2}}\gamma^{0}\gamma^{1}\gamma^{0}\partial_{\psi}\left(-\gamma^{0}u_{n}(\frac{1}{\mu})\right)\mid_{\mu=\mu_{0}}
\]
\[
+\sin\psi\cot\eta\cdot\left(\frac{F[p,q](\frac{i}{\mu_{0}})w_{1}[p,q](\frac{i}{\mu_{0}})w_{3}[p,q](\frac{i}{\mu_{0}})}{w_{2}[p,q](\frac{i}{\mu_{0}})}\right)^{-\frac{1}{2}}\gamma^{0}\gamma^{2}\gamma^{0}\partial_{\psi}\left(-\gamma^{0}u_{n}(\frac{1}{\mu})\right)\mid_{\mu=\mu_{0}}
\]
\[
-\left(\frac{F[p,q](\frac{i}{\mu_{0}})w_{1}[p,q](\frac{i}{\mu_{0}})w_{2}[p,q](\frac{i}{\mu_{0}})}{w_{3}[p,q](\frac{i}{\mu_{0}})}\right)^{-\frac{1}{2}}\gamma^{0}\gamma^{3}\gamma^{0}\partial_{\psi}\left(-\gamma^{0}u_{n}(\frac{1}{\mu})\right)\mid_{\mu=\mu_{0}}
\]
\[
-\frac{1}{4}\frac{-\frac{w_{3}^{'}[p,q](\frac{i}{\mu_{0}})}{w_{3}[p,q](\frac{i}{\mu_{0}})}-\frac{2}{\mu_{0}}-\frac{w_{2}^{'}[p,q](\frac{i}{\mu_{0}})}{w_{2}[p,q](\frac{i}{\mu_{0}})}-\frac{2}{\mu_{0}}-\frac{w_{1}^{'}[p,q](\frac{i}{\mu_{0}})}{w_{1}[p,q](\frac{i}{\mu_{0}})}-\frac{2}{\mu_{0}}-3\frac{F^{'}[p,q](\frac{i}{\mu_{0}})}{F[p,q](\frac{i}{\mu_{0}})}+\frac{6}{\mu_{0}}}{\left(F[p,q](\frac{i}{\mu_{0}})w_{1}[p,q](\frac{i}{\mu_{0}})w_{2}[p,q](\frac{i}{\mu_{0}})w_{3}[p,q](\frac{i}{\mu_{0}})\right)^{\frac{1}{2}}}\times 
\]
\[
 \qquad \qquad \qquad \qquad \qquad \qquad \qquad \qquad\qquad \qquad\qquad\qquad\qquad \qquad \gamma^{0}\left(-\gamma^{0}u_{n}(\frac{1}{\mu_{0}})\right)+
\]
\[
\left(\frac{w_{1}[p,q](\frac{i}{\mu_{0}})w_{2}[p,q](\frac{i}{\mu_{0}})w_{3}[p,q](\frac{i}{\mu_{0}})}{F[p,q](\frac{i}{\mu_{0}})}\right)^{\frac{1}{2}}\cdot\left(\frac{1}{w_{1}^{2}[p,q](\frac{i}{\mu_{0}})}+\frac{1}{w_{2}^{2}[p,q](\frac{i}{\mu_{0}})}+\frac{1}{w_{3}^{2}[p,q](\frac{i}{\mu_{0}})}\right) 
\]
\[  
\qquad \qquad  \qquad \qquad \qquad \qquad \qquad \qquad \qquad \qquad \qquad \qquad \qquad  \times \frac{1}{4} \gamma^{1}\gamma^{2}\gamma^{3}\left(-\gamma^{0}u_{n}(\frac{1}{\mu_{0}})\right)
\]
\begin{eqnarray*}
&=&-\gamma^{0}\left(\tilde{D}[p,q]u_{n}(\mu)\right)\mid_{\mu=\frac{1}{\mu_{0}}} \\ 
&=& \lambda_{n}\left(-\gamma^{0}u_{n}(\frac{1}{\mu}) \right )\mid_{\mu=\mu_{0}}.
\end{eqnarray*}
}

\smallskip

Thus we showed that, for any eigenspinor $u_{n}(\mu)$ of $\tilde{D}[p,q]$ with
eigenvalue $\lambda_{n}$, the spinor $-\gamma^{0}u_{n}(\frac{1}{\mu})$
is an eigenspinor of $-\tilde{D}[-q,p]$ with the same eigenvalue
$\lambda_{n}$. From the above equations, one can also verify immediately
that for any eigenspinor $u_{n}(\mu)$ of $-\tilde{D}[-q,p]$ with
eigenvalue $\lambda_{n}$, the spinor $\gamma^{0}u_{n}(\frac{1}{\mu})$
is an eigenspinor of $\tilde{D}[p,q]$ with the same eigenvalue $\lambda_{n}$.
Consequently, we see that there is a bijection $\lambda_{n}\mapsto-\lambda_{n}$
between the eigenvalues of $\tilde{D}[p,q]$ and those of $-\tilde{D}[-q,p]$.
Since the eigenspinors of $\tilde{D}^{2}$ are exactly the eigenspinors
of $\tilde{D}$ with the corresponding eigenvalues squared, it follows
that $\tilde{D}^{2}[p,q]$ and $\tilde{D}^{2}[-q,p]$ are isospectral.

\smallskip

In a similar manner, we also have that
{\small
\[
\left(\tilde{D}[p,q+p+\frac{1}{2}]u_{n}(\mu-i)\right)\mid_{\mu=\mu_{0}}=
\]
\[
(F[p,q+p+\frac{1}{2}](i\mu_{0})w_{1}[p,q+p+\frac{1}{2}](i\mu_{0})w_{2}[p,q+p+\frac{1}{2}](i\mu_{0}) 
  w_{3}[p,q+p+\frac{1}{2}](i\mu_{0}))^{-\frac{1}{2}}\times 
\]
\[\qquad \qquad  \qquad \qquad \qquad \qquad \qquad \qquad \qquad \qquad \qquad \qquad \qquad
\gamma^{0}\partial_{\text{\ensuremath{\mu}}}u_{n}(\mu-i)\mid_{\mu=\mu_{0}}
\]
\[
-\sin\psi\cdot\left(\frac{F[p,q+p+\frac{1}{2}](i\mu_{0})w_{2}[p,q+p+\frac{1}{2}](i\mu_{0})w_{3}[p,q+p+\frac{1}{2}](i\mu_{0})}{w_{1}[p,q+p+\frac{1}{2}](i\mu_{0})}\right)^{-\frac{1}{2}}\]
\[ \qquad \qquad  \qquad \qquad \qquad \qquad \qquad \qquad \qquad \qquad \qquad \qquad \qquad
\gamma^{1}\partial_{\eta}u_{n}(\mu-i)\mid_{\mu=\mu_{0}}
\]
\[
+\cos\psi\cdot\left(\frac{F[p,q+p+\frac{1}{2}](i\mu_{0})w_{1}[p,q+p+\frac{1}{2}](i\mu_{0})w_{3}[p,q+p+\frac{1}{2}](i\mu_{0})}{w_{2}[p,q+p+\frac{1}{2}](i\mu_{0})}\right)^{-\frac{1}{2}} \times \]
\[
\qquad \qquad  \qquad \qquad \qquad \qquad \qquad \qquad \qquad \qquad \qquad \qquad \qquad
\gamma^{2}\partial_{\eta}u_{n}(\mu-i)\mid_{\mu=\mu_{0}}
\]
\[
+\cos\psi\csc\eta\cdot\left(\frac{F[p,q+p+\frac{1}{2}](i\mu_{0})w_{2}[p,q+p+\frac{1}{2}](i\mu_{0})w_{3}[p,q+p+\frac{1}{2}](i\mu_{0})}{w_{1}[p,q+p+\frac{1}{2}](i\mu_{0})}\right)^{-\frac{1}{2}} \times 
\]
\[
\qquad \qquad  \qquad \qquad \qquad \qquad \qquad \qquad \qquad \qquad \qquad \qquad \qquad \gamma^{1}\partial_{\phi}u_{n}(\mu-i)\mid_{\mu=\mu_{0}}
\]
\[
+\sin\psi\csc\eta\cdot\left(\frac{F[p,q+p+\frac{1}{2}](i\mu_{0})w_{1}[p,q+p+\frac{1}{2}](i\mu_{0})w_{3}[p,q+p+\frac{1}{2}](i\mu_{0})}{w_{2}[p,q+p+\frac{1}{2}](i\mu_{0})}\right)^{-\frac{1}{2}} \times 
\]
\[
\qquad \qquad  \qquad \qquad \qquad \qquad \qquad \qquad \qquad \qquad \qquad \qquad \qquad \gamma^{2}\partial_{\phi}u_{n}(\mu-i)\mid_{\mu=\mu_{0}}
\]
\[
-\cos\psi\cot\eta\cdot\left(\frac{F[p,q+p+\frac{1}{2}](i\mu_{0})w_{2}[p,q+p+\frac{1}{2}](i\mu_{0})w_{3}[p,q+p+\frac{1}{2}](i\mu_{0})}{w_{1}[p,q+p+\frac{1}{2}](i\mu_{0})}\right)^{-\frac{1}{2}}\times 
\]
\[
\qquad \qquad  \qquad \qquad \qquad \qquad \qquad \qquad \qquad \qquad \qquad \qquad \qquad
\gamma^{1}\partial_{\psi}u_{n}(\mu-i)\mid_{\mu=\mu_{0}}
\]
\[
-\sin\psi\cot\eta\cdot\left(\frac{F[p,q+p+\frac{1}{2}](i\mu_{0})w_{1}[p,q+p+\frac{1}{2}](i\mu_{0})w_{3}[p,q+p+\frac{1}{2}](i\mu_{0})}{w_{2}[p,q+p+\frac{1}{2}](i\mu_{0})}\right)^{-\frac{1}{2}} \times
\]
\[\qquad \qquad  \qquad \qquad \qquad \qquad \qquad \qquad \qquad \qquad \qquad \qquad \qquad
\gamma^{2}\partial_{\psi}u_{n}(\mu-i)\mid_{\mu=\mu_{0}}
\]
\[
\quad +\left(\frac{F[p,q+p+\frac{1}{2}](i\mu_{0})w_{1}[p,q+p+\frac{1}{2}](i\mu_{0})w_{2}[p,q+p+\frac{1}{2}](i\mu_{0})}{w_{3}[p,q+p+\frac{1}{2}](i\mu_{0})}\right)^{-\frac{1}{2}} \times 
\]
\[\qquad \qquad  \qquad \qquad \qquad \qquad \qquad \qquad \qquad \qquad \qquad \qquad \qquad
\gamma^{3}\partial_{\psi}u_{n}(\mu-i)\mid_{\mu=\mu_{0}}
\]
\[ 
+\frac{1}{4}\frac{\frac{w_{1}^{'}[p,q+p+\frac{1}{2}](i\mu_{0})}{w_{1}[p,q+p+\frac{1}{2}](i\mu_{0})}+\frac{w_{2}^{'}[p,q+p+\frac{1}{2}](i\mu_{0})}{w_{2}[p,q+p+\frac{1}{2}](i\mu_{0})}+\frac{w_{3}^{'}[p,q+p+\frac{1}{2}](i\mu_{0})}{w_{3}[p,q+p+\frac{1}{2}](i\mu_{0})}+3\frac{F^{'}[p,q+p+\frac{1}{2}](i\mu_{0})}{F[p,q+p+\frac{1}{2}](i\mu_{0})}}{\left(F[p,q+p+\frac{1}{2}](i\mu_{0})w_{1}[p,q+p+\frac{1}{2}](i\mu_{0})w_{2}(i\mu_{0})w_{3}[p,q+p+\frac{1}{2}](i\mu_{0})\right)^{\frac{1}{2}}} \times 
\]
\[\qquad \qquad  \qquad \qquad \qquad \qquad \qquad \qquad \qquad \qquad \qquad \qquad \qquad \gamma^{0}u_{n}(\mu_{0}-i)
\]
\[
-\frac{1}{4}\left(\frac{w_{1}[p,q+p+\frac{1}{2}](i\mu_{0})w_{2}[p,q+p+\frac{1}{2}](i\mu_{0})w_{3}[p,q+p+\frac{1}{2}](i\mu_{0})}{F[p,q+p+\frac{1}{2}](i\mu_{0})}\right)^{\frac{1}{2}}\times  
\]
\[
 \left(\frac{1}{w_{1}^{2}[p,q+p+\frac{1}{2}](i\mu_{0})}+\frac{1}{w_{2}^{2}[p,q+p+\frac{1}{2}](i\mu_{0})}+\frac{1}{w_{3}^{2}[p,q+p+\frac{1}{2}](i\mu_{0})}\right)\gamma^{1}\gamma^{2}\gamma^{3}u_{n}(\mu_{0}-i).
\]
}

\smallskip

Therefore using lemmas \ref{transformationsw_j1} and \ref{transformationsF} in the latter, which explain how the 
parameters of the functions $w_j$ and $F$ are affected by the modular transformation $T_1(i \mu)=i \mu+1$, 
we can  write: 
{\small 
\[
\left(\tilde{D}[p,q+p+\frac{1}{2}]u_{n}(\mu-i)\right)\mid_{\mu=\mu_{0}}=
\]
\[
(F[p,q](i\mu_{0}+1)w_{1}[p,q](i\mu_{0}+1)w_{2}[p,q](i\mu_{0}+1)w_{3}[p,q](i\mu_{0}+1))^{-\frac{1}{2}}\cdot\gamma^{0}\partial_{\text{\ensuremath{\mu}}}u_{n}(\mu)\mid_{\mu=\mu_{0}-i}
\]
\[
 -\sin\psi\cdot\left(\frac{F[p,q](i\mu_{0}+1)w_{2}[p,q](i\mu_{0}+1)w_{3}[p,q](i\mu_{0}+1)}{w_{1}[p,q](i\mu_{0}+1)}\right)^{-\frac{1}{2}}\gamma^{1}\partial_{\eta}u_{n}(\mu)\mid_{\mu=\mu_{0}-i}
\]
\[+\cos\psi\cdot\left(\frac{F[p,q](i\mu_{0}+1)w_{1}[p,q](i\mu_{0}+1)w_{3}[p,q](i\mu_{0}+1)}{w_{2}[p,q](i\mu_{0}+1)}\right)^{-\frac{1}{2}}\gamma^{2}\partial_{\eta}u_{n}(\mu)\mid_{\mu=\mu_{0}-i}
\]
\[
 +\cos\psi\csc\eta\cdot\left(\frac{F[p,q](i\mu_{0}+1)w_{2}[p,q](i\mu_{0}+1)w_{3}[p,q](i\mu_{0}+1)}{w_{1}[p,q](i\mu_{0}+1)}\right)^{-\frac{1}{2}}\gamma^{1}\partial_{\phi}u_{n}(\mu)\mid_{\mu=\mu_{0}-i}
\]
\[
+\sin\psi\csc\eta\cdot\left(\frac{F[p,q](i\mu_{0}+1)w_{1}[p,q](i\mu_{0}+1)w_{3}[p,q](i\mu_{0}+1)}{w_{2}[p,q](i\mu_{0}+1)}\right)^{-\frac{1}{2}}\gamma^{2}\partial_{\phi}u_{n}(\mu)\mid_{\mu=\mu_{0}-i}
\]
\[ 
-\cos\psi\cot\eta\cdot\left(\frac{F[p,q](i\mu_{0}+1)w_{2}[p,q](i\mu_{0}+1)w_{3}[p,q](i\mu_{0}+1)}{w_{1}[p,q](i\mu_{0}+1)}\right)^{-\frac{1}{2}}\gamma^{1}\partial_{\psi}u_{n}(\mu)\mid_{\mu=\mu_{0}-i}
\]
\[
-\sin\psi\cot\eta\cdot\left(\frac{F[p,q](i\mu_{0}+1)w_{1}[p,q](i\mu_{0}+1)w_{3}[p,q](i\mu_{0}+1)}{w_{2}[p,q](i\mu_{0}+1)}\right)^{-\frac{1}{2}}\gamma^{2}\partial_{\psi}u_{n}(\mu)\mid_{\mu=\mu_{0}-i}
\]
\[
+\left(\frac{F[p,q](i\mu_{0}+1)w_{1}[p,q](i\mu_{0}+1)w_{2}[p,q](i\mu_{0}+1)}{w_{3}[p,q](i\mu_{0}+1)}\right)^{-\frac{1}{2}}\gamma^{3}\partial_{\psi}u_{n}(\mu)\mid_{\mu=\mu_{0}-i}
\]
\[
+\frac{1}{4}\frac{\frac{w_{1}^{'}[p,q](i\mu_{0}+1)}{w_{1}[p,q](i\mu_{0}+1)}+\frac{w_{2}^{'}[p,q](i\mu_{0}+1)}{w_{2}[p,q](i\mu_{0}+1)}+\frac{w_{3}^{'}[p,q](i\mu_{0}+1)}{w_{3}[p,q](i\mu_{0}+1)}+3\frac{F^{'}[p,q](i\mu_{0}+1)}{F[p,q](i\mu_{0}+1)}}{\left(F[p,q](i\mu_{0}+1)w_{1}[p,q](i\mu_{0}+1)w_{2}(i\mu_{0})w_{3}[p,q](i\mu_{0}+1)\right)^{\frac{1}{2}}}\cdot\gamma^{0}u_{n}(\mu_{0}-i)
\]
\[
 -\frac{1}{4}\left(\frac{w_{1}[p,q](i\mu_{0}+1)w_{2}[p,q](i\mu_{0}+1)w_{3}[p,q](i\mu_{0}+1)}{F[p,q](i\mu_{0}+1)}\right)^{\frac{1}{2}} \times \]
\[
 \qquad \qquad \left(\frac{1}{w_{1}^{2}[p,q](i\mu_{0}+1)}+\frac{1}{w_{2}^{2}[p,q](i\mu_{0}+1)}+\frac{1}{w_{3}^{2}[p,q](i\mu_{0}+1)}\right)
\gamma^{1}\gamma^{2}\gamma^{3}u_{n}(\mu_{0}-i)
\]
\begin{eqnarray*}
&=&\left(\tilde{D}[p,q]u_{n}(\mu)\right)\mid_{\mu=\mu_{0}-i} \\
&=&\lambda_{n}u_{n}(\mu-i)\mid_{\mu=\mu_{0}}.
\end{eqnarray*}
}
\smallskip

Thus the following identity is proved
\[
\tilde{D}[p,q+p+\frac{1}{2}]u_{n}(\mu-i)=\lambda_{n}u_{n}(\mu-i), 
\]
which implies that the operators $\tilde{D}^{2}[p,q+p+\frac{1}{2}]$
and $\tilde{D}^{2}[p,q]$ are isospectral. 

\smallskip 
Finally, it is easy to see that the transformations $u_n(\mu)\mapsto u_n(\mu+i)$ and $u_n(\mu)\mapsto \gamma^1u_n(\frac{1}{\mu})$ are two-sided inverses of the two eigenspinor transformations defined in the theorem, so the eigenspinor transformations are indeed bijections between the sets of eigenspinors.

\end{proof}

\smallskip

Since the kernel $K_t[p, q]$ of the operator $\exp \left ( -t \tilde D^2[p, q] \right )$ can be written 
as a sum that involves its eigenvalues and their corresponding eigenspinors, 
it follows immediately from Theorem \ref{IsoSpecDiracs2paraThm} that 
indeed the heat kernels corresponding to the parameters $(p, q)$ satisfy 
modular transformation properties. 

\smallskip

\begin{corollary} \label{modularheatkernelCor}

With the spatial dependence suppressed, we have 
the following modular transformation properties for  the heat kernel 
$ K_t[p,q](i\mu_1,i\mu_2) $ of the operator $D^2[p, q]$: 
 
\begin{eqnarray*}
K_t[p,q](i\mu_1+1,i\mu_2+1) &=& K_t[p,q+p+\frac{1}{2}](i\mu_1,i\mu_2), \\
K_t[p,q](-\frac{1}{i\mu_1},-\frac{1}{i\mu_2}) &=& (i\mu_2)^2K_t[-q,p](i\mu_1,i\mu_2). 
\end{eqnarray*}

\end{corollary}

\begin{proof}
Recall that the defining equation for the heat kernel is
\begin{eqnarray*}
\int_{M} K_t[p,q](i\mu_1,i\mu_2)u_n(\mu_2)dvol(\mu_2) = e^{-t\lambda_n^2}u_n(\mu_1), 
\end{eqnarray*}
for any eigenspinor $u_n(\mu)$ of $\tilde{D}^2[p,q]$ acting on $M$. As a result, for $(\mu_1,\mu_2)\in$ $M_1 \times M_1$, we have $(\mu_1-i,\mu_2-i)\in$ $M_0 \times M_0$, so
\begin{eqnarray*}
e^{-t\lambda_n^2}u_n(\mu_1-i) &=& \int_{M_0} K_t[p,q](i\mu_1+1,i\mu_2+1)u_n(\mu_2-i) \, dvol(\mu_2-i)\\
&=& \int_{M_1} K_t[p,q](i\mu_1+1,i\mu_2+1)u_n(\mu_2-i) \,dvol(\mu_2), 
\end{eqnarray*}
where $u_n(\mu-i)$ is the general form of an eigenspinor of $\tilde{D}^2[p,q+p+\frac{1}{2}]$ acting on $M_1$ according to Theorem \ref{IsoSpecDiracs2paraThm}.

\smallskip

Similarly, for $(\mu_1,\mu_2)\in M_2 \times M_2$, we have $(\frac{1}{\mu_1},\frac{1}{\mu_2})\in M_0 \times M_0$, so
\begin{eqnarray*}
e^{-t\lambda_n^2}\left(-\gamma^{0}u_{n}(\frac{1}{\mu_1})\right) &=& 
-\gamma^{0}e^{-t\lambda_n^2}u_{n}(\frac{1}{\mu_2}) \\
&=& \int_{M_0} K_t[p,q](-\frac{1}{i\mu_1},-\frac{1}{i\mu_2})\left(-\gamma^{0}u_{n}(\frac{1}{\mu_2})\right)dvol(\frac{1}{\mu_2})
\end{eqnarray*}
\begin{eqnarray*}
\qquad \qquad &=& \int_{M_2} \big(-\frac{1}{\mu_2^2 }K_t[p,q](-\frac{1}{i\mu_1},-\frac{1}{i\mu_2})\big)\left(-\gamma^{0}u_{n}(\frac{1}{\mu_2})\right)dvol(\mu_2), 
\end{eqnarray*}
where $\left(-\gamma^{0}u_{n}(\frac{1}{\mu})\right)$ is the general form of an eigenspinor of $\tilde{D}^2[-q,p]$ acting on $M_2$ according to Theorem \ref{IsoSpecDiracs2paraThm}. 

\smallskip

Thus we see that $K_t[p,q](i\mu_1+1,i\mu_2+1)$ and $-\frac{1}{\mu_2^2 }K_t[p,q](-\frac{1}{i\mu_1},-\frac{1}{i\mu_2})$ satisfy the defining equations of $K_t[p,q+p+\frac{1}{2}](i\mu_1,i\mu_2)$ and $K_t[-q,p](i\mu_1,i\mu_2)$ respectively. Now, the uniqueness of the heat kernel implies that
\begin{eqnarray*}
K_t[p,q](i\mu_1+1,i\mu_2+1) &=& K_t[p,q+p+\frac{1}{2}](i\mu_1,i\mu_2),\\
K_t[p,q](-\frac{1}{i\mu_1},-\frac{1}{i\mu_2})& = &(i\mu_2)^2K_t[-q,p](i\mu_1,i\mu_2). 
\end{eqnarray*}
\end{proof}

\smallskip 

Having established the modular transformation properties for the heat kernel $K_t[p, q]$, 
we can now show that all of the Seeley-de Witt coefficients $\tilde a_{2n}[p, q]$ inherit 
the same properties from the kernel. 

\begin{corollary} \label{alltermstwoparamodularCor}

For any non-negative integer $n$ we have
\begin{eqnarray*}
\tilde{a}_{2n}[p,q](i\mu+1)&=&\tilde{a}_{2n}[p,q+p+\frac{1}{2}](i\mu), \\
\tilde{a}_{2n}[p,q](\frac{i}{\mu})&=&(i\mu)^{2}\tilde{a}_{2n}[-q,p](i\mu).
\end{eqnarray*}
In addition, at any point $i\mu=P \in\mathbb{H}$, let $v_P(f)$ be the order of zero in $Q$ of the function $f(Q)$, where $Q=e^{-\pi\mu}$. Then we have 
\begin{eqnarray*}
v_P\big(\tilde{a}_{2n}[p,q](Q)\big)=v_P\big(K_t[p,q](Q)\big), 
\end{eqnarray*}
where 
\[
K_t[p,q](Q)=\int_{\mathbb{S}^3}\mathrm{Trace}\big\{K_t[p,q](i\mu,i\mu)\big\}dvol^3.
\] 
So in particular, all of the above Seeley-de Witt coefficients have the same zeros with the same orders in $\mathbb{H}$.

\end{corollary}

\begin{proof}

The Seeley-de Witt coefficients $\tilde{a}_{2n}[p,q](i\mu)$ can be uniquely defined by the asymptotic expansion
\begin{eqnarray*}
K_t[p,q](Q) \sim t^{-2}\sum_{n=0}^\infty \tilde{a}_{2n}[p,q](Q)t^n,
\end{eqnarray*}
or equivalently,
\begin{eqnarray*}
\int_{\mathbb{S}^3}\mathrm{Trace}\big\{K_t[p,q](i\mu,i\mu)\big\}dvol^3 \sim t^{-2}\sum_{n =0}^\infty \tilde{a}_{2n}[p,q](i\mu)t^n. 
\end{eqnarray*}

\smallskip 

So, using  Corollary \ref{modularheatkernelCor}, we have: 
\[
\int_{\mathbb{S}^3}\mathrm{Trace}\big\{K_t[p,q](i\mu+1,i\mu+1)\big\}dvol^3 
=
\int_{\mathbb{S}^3}\mathrm{Trace}\big\{K_t[p,q+p+\frac{1}{2}](i\mu,i\mu)\big\}dvol^3. 
\]
Since the left and the right hand side of the latter have the following small time asymptotic expansions
respectively,
\[
 t^{-2}\sum_{n = 0}^\infty \tilde{a}_{2n}[p,q](i\mu+1)t^n,  
\qquad 
t^{-2}\sum_{n =0}^\infty \tilde{a}_{2n}[p,q+p+\frac{1}{2}](i\mu)t^n, 
\]
it follows from the uniqueness of the asymptotic expansion that
\[
\tilde{a}_{2n}[p,q](i\mu+1)
=
\tilde{a}_{2n}[p,q+p+\frac{1}{2}](i\mu), \qquad n \in \mathbb{Z}_{\geq 0}. 
\]
Also, in a similar manner, we can write
\begin{eqnarray*}
&& \int_{\mathbb{S}^3}\mathrm{Trace}\big\{K_t[p,q](-\frac{1}{i\mu},-\frac{1}{i\mu})\big\}dvol^3 
 \sim t^{-2}\sum_{n=0}^\infty \tilde{a}_{2n}[p,q](-\frac{1}{i\mu})t^n \\
&& \qquad \sim (i\mu)^{2}\int_{\mathbb{S}^3}\mathrm{Trace}\big\{K_t[-q,p](i\mu,i\mu)\big\}dvol^3 \\
&& \qquad \sim t^{-2}\sum_{n=0}^\infty(i\mu)^{2}\tilde{a}_{2n}[-q,p](i\mu)t^n,
\end{eqnarray*} 
which implies that 
\[
\tilde{a}_{2n}[p,q](\frac{i}{\mu})
=
(i\mu)^{2}\tilde{a}_{2n}[-q,p](i\mu), \qquad n \in \mathbb{Z}_{\geq 0}.
\]

\smallskip 
If $v$ is the order of zero in $Q$ of the function $K_t[p,q](Q)$ at $Q_0$, then
\begin{eqnarray*}
\lim_{Q\rightarrow Q_0}(Q-Q_0)^{-v}K_t[p,q](Q) = C_t[p,q], 
\end{eqnarray*}
for $t\in(0,1)$ and some finite   $C_t[p,q]\neq 0$. 
On the other hand, the asymptotic expansion means that for any integer $k$ there is some $N(k)$ such that
\begin{eqnarray*}
\Big|K_t[p,q](Q)-t^{-2}\sum_{n=0}^{N(k)}\tilde{a}_{2n}[p,q](Q)t^n\Big|_{\infty,k}<C_kt^k, \qquad  0<t<1. 
\end{eqnarray*}
See sections 1.1 and 1.7 of the book \cite{GilBook1} for the latter inequality and the definition 
of the norm $|\cdot|_{\infty, k}$.  
So it follows that
\begin{eqnarray*}
\tilde{a}_{2n}[p,q](Q) = \frac{1}{n!}\lim_{t\rightarrow 0^+}\frac{d^n}{dt^n}\big(t^2K_t[p,q](Q)\big), 
\end{eqnarray*}
where the convergence is uniform. Consequently, suppose that
\begin{eqnarray*}
\lim_{Q\rightarrow Q_0}(Q-Q_0)^{-v_n}\tilde{a}_{2n}[p,q](Q) = C_{n,t}[p,q],
\end{eqnarray*}
for some finite $C_{n,t}[p,q]\neq 0$, then we can switch the order of the two limits below and obtain
\begin{eqnarray*}
C_{n,t}[p,q]&=&\lim_{Q\rightarrow Q_0}(Q-Q_0)^{-v_n}\tilde{a}_{2n}[p,q](Q)\big) \\
&=&\frac{1}{n!}\lim_{Q\rightarrow Q_0}\lim_{t\rightarrow 0^+}\frac{d^n}{dt^n}\left (t^2(Q-Q_0)^{-v_n}K_t[p,q](Q)\right )\\
&=&\frac{1}{n!}\lim_{t\rightarrow 0^+}\frac{d^n}{dt^n}\left (t^2C_t[p,q]\lim_{Q\rightarrow Q_0}(Q-Q_0)^{v-v_n}\right). 
\end{eqnarray*}
As a result, for $C_{n,t}[p,q]$ to be finite and nonzero, we need to have $v_n=v$, which proves that the Seeley-de Witt coefficients have the same  zeros of the same orders.

\end{proof}

\smallskip

As an important remark, it should be mentioned that the proof of the latter corollary 
covers the case of poles of meromorphic functions 
considered as zeros of negative orders.  

\smallskip

We now turn our focus to the one-parametric case. We proved in Theorem \ref{modualra024Soneparathm} that for the one-parametric Bianchi 
IX gravitational instantons \eqref{one-parametric} as well the terms $\tilde a_0[q_0]$, 
$\tilde a_2[q_0]$ and $\tilde a_4[q_0]$  satisfy modular transformation properties. In 
order to show that all of the terms $\tilde a_{2n}[q_0]$ satisfy the properties mentioned 
in this theorem, similarly to the two-parametric case treated so far in this section, 
we can use the isospectrality of the involved Dirac operators.

\smallskip

\begin{theorem} \label{IsoSpecOneParaThm}
For the one-parameter family
of metrics characterized by $q_{0}$, the operators $\tilde{D}^2[q_{0}]$,
$\tilde{D}^2[q_{0}-i]$,  $-q_{0}^{-2}\tilde{D}^2[\frac{1}{q_{0}}]$
are isospectral.
\end{theorem}

\begin{proof}
It is given in Appendix \ref{IsoSpecOneParaThmPfappendix}.
\end{proof}

An immediate corollary of the latter theorem is the following statement for all of the 
Seeley-de Witt coefficients $\tilde a_{2n}[q_0]$. 

\smallskip

\begin{theorem}
For any non-negative integer $n$, we have: 
\begin{eqnarray*}
\tilde{a}_{2n}[q_{0}](i\mu+1)&=&\tilde{a}_{2n}[q_{0}-i](i\mu), \\
\tilde{a}_{2n}[q_{0}](\frac{i}{\mu})&=&(-1)^{n+1}q_{0}^{4-2n}\mu^{2}\tilde{a}_{2n}[\frac{1}{q_{0}}](i\mu).
\end{eqnarray*}

\begin{proof}

Since the operators  
$\tilde{D}^{2}[q_{0}],$ $\tilde{D}^{2}[q_{0}-i]$ and  $-q_{0}^{-2}\tilde{D}^{2}[\frac{1}{q_{0}}]$ 
are  isospectral, they have 
the same small time heat kernel expansions, namely that, for all non-negative integers $n$ we have 
\[
\tilde{a}_{2n}[q_{0}]=(-1)^{n}q_{0}^{4-2n}\tilde{a}_{2n}[\frac{1}{q_{0}}]=\tilde{a}_{2n}[q_{0}-i].
\]
Therefore, for arbitrary real numbers $a$ and $b$ we have 
\begin{eqnarray*}
\int_{a}^{b}d\mu\cdot\tilde{a}_{2n}[q_{0}](i\mu) 
&=& \int_{\frac{1}{b}}^{\frac{1}{a}}d\frac{1}{\mu}\cdot(-1)^{n}q_{0}^{4-2n}\tilde{a}_{2n}[\frac{1}{q_{0}}](\frac{i}{\mu}) \\
&=&\int_{a}^{b}d\mu\cdot\left((-1)^{n+1}\mu^{-2}q_{0}^{4-2n}\tilde{a}_{2n}[\frac{1}{q_{0}}](\frac{i}{\mu})\right) \\
&=&\int_{a+i}^{b+i}d(\mu+i)\cdot\tilde{a}_{2n}[q_{0}-i](i(\mu+i)) \\
&=&\int_{a}^{b}d\mu\cdot\tilde{a}_{2n}[q_{0}-i](i\mu-1).  
\end{eqnarray*}
This show that 
\[
\tilde{a}_{2n}[q_{0}](i\mu)=(-1)^{n+1}\mu^{-2}q_{0}^{4-2n}\tilde{a}_{2n}[\frac{1}{q_{0}}]=\tilde{a}_{2n}[q_{0}-i](i\mu-1), 
\]
which is equivalent to the statement of this theorem.

\end{proof}

\end{theorem}

\smallskip

\section{Modular  forms arising from $\tilde a_{2n}[p, q]$}
\label{ModularFormsSec}

\smallskip 

In this section, we use the modular transformation properties of the 
Seeley-de Witt coefficients studied in the previous sections to show that when 
the parameters of the metric are rational in the two-parametric case, they give rise to vector-valued 
modular forms. Then we show that by a summation over a finite orbit of 
the parameters, they give rise to ordinary modular functions. At the end 
we investigate their connection with well-known modular forms. Indeed, 
we show that, in examples of two general cases, one with poles at infinity and the other 
with no poles at infinity, the modular functions corresponding to the  
Seeley-de Witt coefficients   land in a direct way in the modular forms 
of weight 14 or in the cusp forms of weight 18.

\smallskip

\subsection{Vector-valued and ordinary modular functions from $\tilde a_{2n}[p, q]$}

The following lemma shows the periodicity of all Seeley-de Witt coefficients 
$\tilde a_{2n}[p, q]$ in both of the parameters of the metric with period 1. 
This is a crucial step for showing that each $\tilde a_{2n}$ is defining 
a vector-valued modular form with respect to a finite dimensional representation 
of the modular group $PSL_2(\mathbb{Z})$. More importantly, it also allows one 
to construct ordinary modular functions from each $\tilde a_{2n}[p, q]$, which 
can then be related to well-known modular forms. 

\smallskip

\begin{lemma} \label{periodicityinbothparametersLemma}
For any non-negative integer $n$ and any parameters $(p, q)$ 
of the Bianchi IX gravitational instantons we have: 
\[
\tilde a_{2n}[p+1, q] = \tilde a_{2n}[p, q+1] = \tilde a_{2n}[p, q],
\]
\end{lemma}
\begin{proof}
Recalling the ingredients of the explicit formulas for the metric from Subsection 4.1 and using Lemma 5.2 we know that all the involved functions are invariant under $p\mapsto p+1$, so the periodicity in $p$  follows trivially. In addition, we  have
\begin{eqnarray*}
w_{1}[p,q+1](i\mu)&=&-\frac{i}{2}\vartheta_{3}(i\mu)\vartheta_{4}(i\mu)\frac{e^{2\pi i p}\partial_{q}\vartheta[p,q+\frac{1}{2}](i\mu)}{e^{2\pi i p}e^{\pi ip}\vartheta[p,q](i\mu)}= w_{1}[p,q](i\mu), \\
w_{2}[p,q+1](i\mu)&=&\frac{i}{2}\vartheta_{2}(i\mu)\vartheta_{4}(i\mu)\frac{-e^{2\pi i p}\partial_{q}\vartheta[p+\frac{1}{2},q+\frac{1}{2}](i\mu)}{e^{2\pi i p}e^{\pi ip}\vartheta[p,q](i\mu)}=-w_{2}[p,q](i\mu),\\
w_{3}[p,q+1](i\mu)&=&-\frac{1}{2}\vartheta_{2}(i\mu)\vartheta_{3}(i\mu)\frac{-e^{2\pi i p}\partial_{q}\vartheta[p+\frac{1}{2},q](i\mu)}{e^{2\pi i p}\vartheta[p,q](i\mu)}=-w_{3}[p,q](i\mu),\\
F[p,q](i\mu)&=&\frac{2}{\pi\Lambda} \left(\frac{e^{2\pi i p}\vartheta[p,q](i\mu)}{e^{2\pi i p}\partial_{q}\vartheta[p,q](i\mu)}\right)^{2}=F[p,q](i\mu).
\end{eqnarray*}
Consequently, we see from the equations above that the metric
\begin{equation*}
d\tilde s^2= F \left ( w_1 w_2 w_3 \, d\mu^2 +
\frac{w_2 w_3}{w_1} \sigma_1^2 +
\frac{w_3 w_1}{w_2} \sigma_2^2+
\frac{w_1 w_2}{w_3} \sigma_3^2  \right ). 
\end{equation*}
is invariant under $q\mapsto q+1$, so the periodicity in $q$ also follows.
\end{proof}

\smallskip

Now, relying on  Lemma \ref{periodicityinbothparametersLemma}, we can 
correspond the following maps to the generators of $PSL_2(\mathbb{Z})$ acting 
on the ordered pair $(p,q)\in S=[0,1)^2 = (\mathbb{R}/\mathbb{Z})^2$:   
\begin{eqnarray*}
\tilde{S}(p,q)&=&(-q,p), \\
\tilde{T}_1(p,q)&=&(p,q+p+\frac{1}{2}),
\end{eqnarray*}
in both of which the parameters are considered modulo 1. 
For rational $p,q$ with $N$ being a common multiple of $2$ and the denominators of $p$ and $q$, it follows immediately from the definitions above that the orbit $\mathcal{O}_{(p,q)}$ of $(p,q)$ under the action of $PSL_2(\mathbb{Z})$ consists of $(p,q)$ pairs with $p,q\in\mathcal{N}=\{0,2/N ,\ldots, (N-1)/N \}$. Namely, we have
\[
\mathcal{O}_{(p,q)} \subset \mathcal{N}^2 \subset [0,1)^2,
\]
and thus $\mathcal{O}_{(p,q)}$ is finite, with any element of $PSL_2(\mathbb{Z})$ acting as a permutation on $\mathcal{O}_{(p,q)}$.

\begin{theorem} \label{vectorvaluedmodularformThm}
Starting from a pair of rational numbers  $(p, q)$, for any non-negative integer $n$, the term  $\tilde a_{2n}[p', q'](i\mu)$, where $(p', q') \in \mathcal{O}_{(p, q)}$, is a vector-valued modular function of weight 2 for the modular group $PSL_2(\mathbb{Z})$.
\begin{proof}
Since the orbit $\mathcal{O}_{(p,q)}$ is finite in this case, we can arrange the functions $\tilde{a}_{2n}[p',q'](i\mu)$ with $(p',q')\in\mathcal{O}_{(p,q)}$ into a finite column vector $\tilde{A}_{2n}\Big(i\mu;\mathcal{O}_{(p,q)}\Big)$ of some dimension $d$. Since any $M\in PSL_2 (\mathbb{Z})$ acts as a permutation on  $\mathcal{O}_{(p,q)}$, we may denote by $\rho: S_d\mapsto GL(p,\mathbb{C})$ the natural permutation representation of $S_d$ which acts on $\tilde{A}_{2n}\Big(i\mu;\mathcal{O}_{(p,q)}\Big)$ by permuting its components in the corresponding way. 

\smallskip

From Corollary \ref{alltermstwoparamodularCor} we know that
\begin{align*}
\tilde{A}_{2n}\Big(M(i\mu);\mathcal{O}_{(p,q)}\Big) &= \big(c\cdot i\mu+d\big)^{2}\tilde{a}_{2n} \big[M(p,q)\big] (i\mu)
\\
&\big(c\cdot i\mu+d\big)^{2}\rho(M)\tilde{A}_{2n} [(p,q)] (i\mu), 
\end{align*}
for any $M\in PSL_2(\mathbb{Z})$. So by definition, $\tilde a_{2n}[p, q](i\mu)$ is a vector-valued modular 
function of weight 2.
\end{proof}
\end{theorem}

\smallskip

Since the orbit is finite for a rational choice of parameters, by a summation over the 
orbit we obtain ordinary modular functions as follows.

\smallskip

\begin{corollary} \label{modularfunctionbysumCor}
For any pair of rational number $(p, q)$, and any non-negative integer $n$, the sum 
\[
\tilde{a}_{2n}\big(i\mu;\mathcal{O}_{(p,q)}\big) = \sum_{(p', q')\in \mathcal{O}_{(p,q)}} \tilde{a}_{2n}[p', q'](i\mu)
\] 
defines a modular function  of weight 2 for the modular group $PSL_2(\mathbb{Z})$.
\begin{proof}
Summing up all the components of the column vector $\tilde{A}_{2n}\left (i\mu;\mathcal{O}_{(p,q)}\right )$ in Theorem \ref{vectorvaluedmodularformThm}, we find that the sum $\tilde{a}_{2n}\big(i\mu;\mathcal{O}_{(p,q)}\big)$ satisfies
\begin{eqnarray*}
\tilde{a}_{2n}\big(M(i\mu);\mathcal{O}_{(p,q)}\big) &=& \sum_{(p',q')\in \mathcal{O}_{(p,q)}} \big(c\cdot i\mu+d\big)^{2}\tilde{a}_{2n} \big[M(p',q')\big] (i\mu)
\\
&=&\sum_{(p',q')\in \mathcal{O}_{(p,q)}} \big(c\cdot i\mu+d\big)^{2}\tilde{a}_{2n} [(p',q')] (i\mu) \\
&=&  \big(c\cdot i\mu+d\big)^{2}\tilde{a}_{2n}\big(i\mu;\mathcal{O}_{(p,q)}\big), 
\end{eqnarray*}
for any $M\in PSL_2(\mathbb{Z})$ which acts on the variable $i\mu$ as a M\"obius transformation, and on $(p,q)$ as defined previously. Namely, $\tilde{a}_{2n}\big(i\mu;\mathcal{O}_{(p,q)}\big)$ is a modular function  of weight 2 with respect to  $PSL_2(\mathbb{Z})$. 
\end{proof}
\end{corollary}

\smallskip

\subsection{Connection between $\tilde a_{2n}[p, q]$ and well-known modular forms} Here 
we investigate the connection between modular functions of type constructed in  
Corollary \ref{modularfunctionbysumCor} and well-known modular forms. It was seen in this corollary that when the parameters 
$(p, q)$ of the metric are rational, each $\tilde{a}_{2n}\big(i\mu;\mathcal{O}_{(p,q)}\big)$ has the standard modular transformation properties of a modular function of weight 2. However, since there 
are no non-trivial holomorphic modular forms of weight 2, it is necessary for $\tilde{a}_{2n}\big(i\mu;\mathcal{O}_{(p,q)}\big)$ to have poles in the variable $i\mu$. For a detailed discussion of 
holomorphic modular forms, Eisenstein series and some related fundamental results used in this 
subsection, one can for example refer to \cite{SerBook}.

\smallskip

By using Corollary \ref{alltermstwoparamodularCor}, in order to find the locations and multiplicities of the zeros of $\tilde{a}_{2n}[p,q](i\mu)$, it is enough for us to investigate the poles of $\tilde{a}_0[p,q](i\mu)$, and the result would apply to all $\tilde{a}_{2n}[p,q](i\mu)$. Recall that
\begin{eqnarray*}
\tilde{a}_0[p,q]&=&4F^{2}w_{1}w_{2}w_{3} \\
&=&-\frac{2}{\pi^2\Lambda^2}\cdot\frac{\vartheta_2^2\vartheta_3^2\vartheta_4^2\vartheta[p,q]\partial_q\vartheta[p,q+\frac{1}{2}]\partial_q\vartheta[p+\frac{1}{2},q+\frac{1}{2}]\partial_q\vartheta[p+\frac{1}{2},q]}{(\partial_q\vartheta[p,q])^4}. 
\end{eqnarray*}

\smallskip

Since all the theta functions and theta derivatives are holomorphic for $i\mu\in\mathbb{H}$, the singularities of 
\[
\tilde{a}_{2n}\big(i\mu;\mathcal{O}_{(p,q)}\big) = \sum_{(p', q')\in \mathcal{O}_{(p,q)}} \tilde{a}_{2n}[p', q'](i\mu)
\]
may appear only at the zeros of the function $\partial_q\vartheta[p,q](i\mu)$. In addition, because of the modular properties, it would be enough for us to look for poles in the fundamental domain $\mathbb{H}/PSL_2(\mathbb{Z})$ and at infinity. 

\smallskip

Consequently, in principle, we only need to know the locations and multiplicities of the zeros of $\partial_q\vartheta[p,q](i\mu)$ in order to figure out the space of modular forms that can be constructed from $\tilde{a}_{2n}\big(i\mu;\mathcal{O}_{(p,q)}\big)$ by removing the poles. So we proceed by proving the following lemmas concerning the zeros of $\vartheta[p,q](i\mu)$ and $\partial_q\vartheta[p,q](i\mu)$ when the parameters $p$ and $q$ are both real. In fact, according to Lemma \ref{periodicityinbothparametersLemma}, $p$ and $q$ are defined modulo 1, so in what is to be presented we will restrict to $(p,q)\in S=[0,1)^2$.

\smallskip

\begin{lemma} \label{orderofzeroatinfinityLemma1}
\label{ZerosOfThetaAtInf}
Let $v_\infty(F) $ be the order of zero of any function $F(i\mu)$ at infinity, and denote $v_\infty(\partial_q\vartheta[p,q])$ by $v_\infty[p,q]$ for simplicity.  We have
\begin{eqnarray*}
v_\infty(\vartheta_2) = \frac{1}{8}, \qquad  v_\infty(\vartheta_3) = v_\infty(\vartheta_4) = 0, \qquad v_\infty(\vartheta[p,q]) = \frac{\langle p\rangle^2}{2},
\end{eqnarray*}

\[
v_\infty(\partial_q\vartheta[p,q])=
\begin{cases}
+\infty   & \text{if } p=0, q \in \{ 0, \frac{1}{2} \},  \\
+\infty & \text{if } p=\frac{1}{2}, q=0, \\ 
\frac{1}{2} & \text{if } p=0, q \not\in \{ 0, \frac{1}{2} \},  \\
\frac{\langle p\rangle^2}{2} & \text{otherwise.}
\end{cases}
\]
where $\langle p\rangle$ is defined as the number $\langle p\rangle\equiv p\;\mathrm{mod}\; 1$ such that $\langle p\rangle\in[-\frac{1}{2},\frac{1}{2})$. 

\begin{proof}
These results can be obtained directly by keeping only the leading order terms in the defining formula (29) and its $q$ derivative.
\end{proof}
\end{lemma}

\smallskip 

The latter can be used to obtain information about the order of zero of $\tilde a_0[p, q]$ at infinity as follows.  

\smallskip

\begin{corollary} \label{orderofzeroatinfinityLemma2}
For $(p,q)\not\in \big \{ (0,0),(0,\frac{1}{2}),(\frac{1}{2},0),(\frac{1}{2},\frac{1}{2}) \big  \}$ we have
\begin{eqnarray*}
v_\infty(\tilde{a}_0[p,q])=\langle p\rangle^2+\langle p+\frac{1}{2}\rangle^2+\frac{1}{4}\langle p\rangle^2=|p-\frac{1}{2}|\ge 0, 
\end{eqnarray*}
if $p\neq0$, and
\begin{eqnarray*}
v_\infty(\tilde{a}_0[p,q])=\frac{1}{4}+2\times\frac{1}{8}+\frac{1}{2}-4\times\frac{1}{2}=-1, 
\end{eqnarray*}
if $p=0$.
\begin{proof}
Both of the above statements can be proved by the substitution of the identities given in  Lemma \ref{orderofzeroatinfinityLemma1} into the formula
\begin{eqnarray*}
v_\infty(\tilde{a}_0[p,q])&=&2v_\infty(\vartheta_2)+2v_\infty(\vartheta_3)+2v_\infty(\vartheta_4)+v_\infty(\vartheta[p,q])+v_\infty[p+\frac{1}{2},q]\\
&&+v_\infty[p+\frac{1}{2},q+\frac{1}{2}]+v_\infty[p,q+\frac{1}{2}]-4v_\infty[p,q]. 
\end{eqnarray*}
\end{proof}
\end{corollary}

\smallskip 

Now we can prove the following statements which will play crucial roles 
in studying the poles of the modular functions that are studied in detail 
in the sequel and have been related to well-known modular forms. 

\smallskip

\begin{lemma}
\label{ZerosOfTheta}

Let $n$ be the number of points in the orbit $\mathcal{O}_{(p,q)}$ where $\mathcal{O}_{(p,q)}\neq\{(\frac{1}{2},\frac{1}{2})\}$ and $\mathcal{O}_{(p,q)}\neq\{(0,0),(\frac{1}{2},0),(0,\frac{1}{2})\}$. Then $n$ is necessarily even, and with $v_\tau[p,q]$ denoting the order of zero of the function $\partial_q\vartheta[p,q](i\mu)$ at $i\mu=\tau$, we have the equation
\begin{eqnarray*}
\sum_{(p',q')\in\mathcal{O}_{(p,q)}}\Big(\frac{1}{2}v_i[p',q'] + \frac{1}{3}v_\rho[p',q'] + \sideset{}{^*}\sum_{P\in\mathbb{H}/PSL_2(\mathbb{Z})} v_P[p',q']\Big) = \frac{n}{12}-\frac{n_0}{2}, 
\end{eqnarray*}
where $\rho=e^{\frac{2\pi i}{3}}$, $n_0$ is the number of points $(p', q') \in \mathcal{O}_{(p, q)}$ such that   $p'=0$, and the starred  summation excludes the points $i$ and $\rho$.

\begin{proof}
First notice that if $(p,q)$ and $(-p,-q)$ correspond to the same point in $S=[0,1)^2$ then we have $p\equiv-p\;\mathrm{mod}\;1$ and $q\equiv -q\;\mathrm{mod}\;1$, i.e., $p\equiv 0\;\mathrm{mod}\;\frac{1}{2}$ and $q\equiv 0\;\mathrm{mod}\;\frac{1}{2}$. So in $S$, the only possibility is that 
\begin{eqnarray*}
(p,q) \in \Big \{ (0,0),(0,\frac{1}{2}),(\frac{1}{2},0), (\frac{1}{2},\frac{1}{2}) \Big \}. 
\end{eqnarray*}
Thus, if $\mathcal{O}_{(p,q)}\neq\{(\frac{1}{2},\frac{1}{2})\}$ and $\mathcal{O}_{(p,q)}\neq\{(0,0),(\frac{1}{2},0),(0,\frac{1}{2})\}$, then $(p,q)$ and $(-p,-q)$ correspond to different points in $S$.

\smallskip

Also if $(p',q')\in\mathcal{O}_{p,q}$, then $(-q',p')\in\mathcal{O}_{p,q}$ and thus $(-p',-q')\in\mathcal{O}_{p,q}$. 
So, when $\mathcal{O}_{(p,q)}\neq\{(\frac{1}{2},\frac{1}{2})\}$ and $\mathcal{O}_{(p,q)}\neq\{(0,0),(\frac{1}{2},0),(0,\frac{1}{2})\}$, for any $(p',q')\in\mathcal{O}_{(p,q)}$, it has a partner $(-p',-q')\neq(p',q')$ in $S$ that is also in the orbit. It is evident that only identical points in $S$ can have the same partner, so it follows that in this case $n$ is necessarily even.

\smallskip

Recall from Lemma \ref{transformationsvarthetapq} that
\begin{eqnarray*}
\partial_q^n\vartheta[p,q](i\mu+1) &=& e^{-\pi i p(p+1)}\partial_q^n\vartheta[p,q+p+\frac{1}{2}](i\mu), \\
\vartheta[p,q](-\frac{1}{i\mu}) &=& e^{2\pi ipq}\mu^{\frac{1}{2}}\vartheta[-q,p](i\mu), \\
\partial_q\vartheta[p,q](-\frac{1}{i\mu}) &=& -ie^{2\pi ipq}\mu^{\frac{3}{2}}\partial_q\vartheta[-q,p](i\mu). 
\end{eqnarray*}
So we have
\begin{eqnarray*}
\frac{\partial_q\vartheta[p,q](i\mu+1)}{\vartheta[p,q](i\mu+1)} &=& \frac{\partial_q\vartheta[p,q](i\mu)}{\vartheta[p,q](i\mu)}, \\
\frac{\partial_q\vartheta[p,q](-\frac{1}{i\mu})}{\vartheta[p,q](-\frac{1}{i\mu})} &=& -i\mu\frac{\partial_q\vartheta[p,q](i\mu)}{\vartheta[p,q](i\mu)}.
\end{eqnarray*}
Consequently, with $n$ an even number, if we define the following function 
\begin{eqnarray*}
f(i\mu)=\prod_{(p,q)\in\mathcal{O}_{(p,q)}}\frac{\partial_q\vartheta[p,q](i\mu)}{\vartheta[p,q](i\mu)},
\end{eqnarray*}
then we  have
\begin{eqnarray*}
f(i\mu+1)&=&f(i\mu),\\
f(-\frac{1}{i\mu})&=&(-1)^n(i\mu)^nf(i\mu)=(i\mu)^nf(i\mu),
\end{eqnarray*}
which shows that $f(i\mu)$ is a modular function of weight $n$. Therefore using the valence formula 
we can write 
\begin{eqnarray*}
 v_\infty(f) +  \frac{1}{2}v_i(f) + \frac{1}{3}v_\rho(f) + \sideset{}{^*}\sum_{P\in\mathbb{H}/PSL_2(\mathbb{Z})}v_P(f) 
= 
\frac{n}{12}. 
\end{eqnarray*}
Recall that the zeros of $\vartheta[p,q](i\mu)$ satisfy the equation
\begin{eqnarray*}
\left (p-\frac{1}{2}+m \right )i\mu + q-\frac{1}{2}+k=0
\end{eqnarray*}
for some $m, k\in\mathbb{Z}$. So, for any real $(p,q)\neq(\frac{1}{2},\frac{1}{2})$, $\vartheta[p,q](i\mu)$ does not have any zeros in $\mathbb{H}$. Also, notice that the theta functions and the theta derivatives are all holomorphic in $i\mu\in\mathbb{H}$. So it follows that the function $f(i\mu)$ has no pole away from infinity, and for $P\in\mathbb{H}$ we have exactly $v_P(f)=\sum_{(p',q')\in\mathcal{O}_{(p,q)}}v_P[p',q']$, with $v_P[p',q']\ge 0$ at any $P$ in the fundamental domain. At infinity, we know from Lemma \ref{orderofzeroatinfinityLemma1}  that
\begin{eqnarray*}
v_\infty(f)&=&\sum_{(p',q')\in\mathcal{O}_{(p,q)}}\Big(v_P[p',q']-v_\infty(\vartheta[p',q'])\Big)\\
&=&\sum_{\substack{(p',q')\in\mathcal{O}_{(p,q)}  \\ p'\neq 0 }}\Big(v_P[p',q']-v_\infty(\vartheta[p',q'])\Big) \\ 
&&+\sum_{(0,q')\in\mathcal{O}_{(p,q)}}\Big(v_P[0,q']-v_\infty(\vartheta[0,q'])\Big)\\
&=&\sum_{\substack{(p',q')\in\mathcal{O}_{(p,q)}\\ p'\neq 0}}\Big(\frac{\langle p'\rangle^2}{2}-\frac{\langle p'\rangle^2}{2}\Big)+\sum_{(0,q')\in\mathcal{O}_{(p, q)}}\Big(\frac{1}{2}-\frac{\langle 0\rangle^2}{2}\Big)\\
&=&\frac{n_0}{2},
\end{eqnarray*}
where $n_0$ is the number of points in the orbit whose first coordinates are zero, $p'=0$. Finally it follows from the valence formula for $f$  that
\begin{eqnarray*}
\sum_{(p',q')\in\mathcal{O}_{(p,q)}}\Big(\frac{1}{2}v_i[p',q'] + \frac{1}{3}v_\rho[p',q'] + \sideset{}{^*}\sum_{P\in\mathbb{H}/PSL_2(\mathbb{Z})}v_P[p',q']\Big) = \frac{n}{12}-\frac{n_0}{2}. 
\end{eqnarray*}
\end{proof}
\end{lemma}

\smallskip

\begin{corollary} \label{nandn0Cor}
Let $n$ and $n_0$ be defined as in Lemma \ref{ZerosOfTheta}. Then, we have $n\ge 6n_0$, and the equality holds if and only if $\partial_q[p,q](i\mu)$ has no zeros in the upper-half complex plane for any $(p,q)$ in the orbit.
\begin{proof}
This follows directly from the non-negativity of $v_P[p,q]$.
\end{proof}
\end{corollary}

\smallskip

When $n$ is small, for instance when $n=n_0$, one can solve for all the $v_P[p',q']$ by simple arithmetic arguments. For example, when $n=8$, one has
\begin{eqnarray*}
\sum_{(p',q')\in\mathcal{O}_{(p,q)}}\Big(\frac{1}{2}v_i[p',q'] + \frac{1}{3}v_\rho[p',q'] + \sideset{}{^*}\sum_{P\in\mathbb{H}/PSL_2(\mathbb{Z})}v_P[p',q']\Big) = \frac{2}{3}. 
\end{eqnarray*}
So, since $\tilde{a}_0[p',q']=\tilde{a}_0[-p',-q']$, we have $v_P[p',q']=v_P[-p',-q']$ where $(p',q')$ and $(-p',-q')$ correspond to different points in $S$. Therefore, the following can be the 
only non-negative solution:   
$v_\rho[p_0,q_0]$$=v_\rho[-p_0,-q_0]=1$ for some $(p_0,q_0)\in \mathcal{O}_{(p,q)}$, and 
$v_P[p',q']=0$ for any other $P$ in the fundamental domain and $(p',q')\in\mathcal{O}_{(p,q)}$.

\smallskip

We now work out different examples explicitly. For instance,  we look at the following orbit generated by $(p,q)=(0,\frac{1}{3})$: 
\begin{eqnarray*}
\mathcal{O}_{(0,\frac{1}{3})}&=&\Big\{(\frac{1}{2},\frac{2}{3}),(\frac{1}{2},\frac{1}{3}),(\frac{5}{6},\frac{2}{3}),(\frac{5}{6},\frac{1}{3}),(\frac{5}{6},0),(\frac{1}{6},\frac{2}{3}),(\frac{1}{6},\frac{1}{3}),(\frac{1}{6},0),(\frac{2}{3},\frac{5}{6}), \\ 
&& (\frac{2}{3},\frac{2}{3}),
(\frac{2}{3},\frac{1}{2}),(\frac{2}{3},\frac{1}{3}),(\frac{2}{3},\frac{1}{6}),(\frac{2}{3},0),(\frac{1}{3},\frac{5}{6}),(\frac{1}{3},\frac{2}{3}),(\frac{1}{3},\frac{1}{2}), 
(\frac{1}{3},\frac{1}{3}),\\
&&(\frac{1}{3},\frac{1}{6}),(\frac{1}{3},0),
(0,\frac{1}{6}),(0,\frac{2}{3}),(0,\frac{5}{6}),(0,\frac{1}{3})\Big\}. 
\end{eqnarray*}
For this case we have the following statement. 

\smallskip

\begin{theorem} \label{poleatinfinityThmExplicit}

For any non-negative integer $n$, $\tilde{a}_{2n}\big(i\mu;\mathcal{O}_{(0,\frac{1}{3})}\big)$ is in the one-dimensional space spanned by 
\[
\frac{G_{14}(i\mu)}{\Delta(i\mu)}, 
\] where $\Delta$ is the modular discriminant (a cusp form of weight 12), and $G_{14}$ is the Eisenstein series of weight $14$.

\begin{proof}

From the orbit above one observes that $n=24=6n_0$ in this case, so Corollary \ref{nandn0Cor} shows that $\partial_q[p,q](i\mu)$ has no zeros in the upper-half complex plane for any $(p,q)$ in the orbit. Therefore, $\tilde{a}_{0}\big(i\mu;\mathcal{O}_{(0,\frac{1}{3})}\big)$ and thus $\tilde{a}_{2n}\big(i\mu;\mathcal{O}_{(0,\frac{1}{3})}\big)$ are holomorphic on $\mathbb{H}$ as we have argued previously.

\smallskip

In addition, by either Corollary \ref{orderofzeroatinfinityLemma2} or direct expansion in $Q=e^{-2\pi\mu}$ we know that $\tilde{a}_{0}\big(i\mu;\mathcal{O}_{(0,\frac{1}{3})}\big)$ has a simple pole at infinity, $i \mu = \infty$, $Q=0$. Since the modular discriminant $\Delta$ has a zero of order $1$ at $i\mu=\infty$, it follows that $\Delta(i\mu)\cdot\tilde{a}_{0}\big(i\mu;\mathcal{O}_{(0,\frac{1}{3})}\big)$ is a modular function holomorphic on $\mathbb{H}$ as well as at $i\mu=\infty$. Namely, $\Delta(i\mu)\cdot\tilde{a}_{0}\big(i\mu;\mathcal{O}_{(0,\frac{1}{3})}\big)$ is a modular form of weight $12+2=14$. Since the space of modular forms of weight 14 is a one-dimensional space generated by the Eisenstein series $G_{14}$, the desired result follows.

\smallskip

\end{proof}

\end{theorem}

Indeed, by writing their $Q$-expansions explicitly, we confirm that for  $\tilde{a}_{0}\big(i\mu;\mathcal{O}_{(0,\frac{1}{3})}\big)$, $\tilde{a}_{2}\big(i\mu;\mathcal{O}_{(0,\frac{1}{3})}\big)$ and $\tilde{a}_{4}\big(i\mu;\mathcal{O}_{(0,\frac{1}{3})}\big)$ we have: 
\begin{eqnarray*}
\tilde{a}_{0}\big(i\mu;\mathcal{O}_{(0,\frac{1}{3})}\big) &=& \frac{-\frac{4}{3}Q^{-1}+262512Q+\frac{171950080}{3}Q^2+3457199880Q^3+\cdots}{\pi^3\Lambda^2}\\
&=& -\frac{6081075}{\pi^{17}\Lambda^2}\cdot\frac{G_{14}(i\mu)}{\Delta(i\mu)}, 
\end{eqnarray*}
\begin{eqnarray*}
\tilde{a}_{2}\big(i\mu;\mathcal{O}_{(0,\frac{1}{3})}\big) &=& \frac{\frac{4}{3}Q^{-1}-262512Q-\frac{171950080}{3}Q^2-3457199880Q^3+\cdots}{\pi\Lambda}\\
&=& \frac{6081075}{\pi^{15}\Lambda}\cdot\frac{G_{14}(i\mu)}{\Delta(i\mu)},
\end{eqnarray*}
\begin{eqnarray*}
\tilde{a}_{4}\big(i\mu;\mathcal{O}_{(0,\frac{1}{3})}\big) &=& \frac{-\frac{4}{15}Q^{-1}+\frac{87504}{5}Q+\frac{34390016}{9}Q^2+230479992Q^3+\cdots}{\pi^{-1}\Lambda^0}\\
&=& -\frac{405405}{\pi^{13}}\cdot\frac{G_{14}(i\mu)}{\Delta(i\mu)}. 
\end{eqnarray*}

\smallskip

Using this approach and taking advantage of the lemmas proved here, one can prove similar results 
for many other orbits $\mathcal{O}_{(p,q)}$. 
As another example, we can also look at the following orbit generated by $(p,q)=(\frac{1}{6},\frac{5}{6})$ containing 8 points:
\begin{eqnarray*}
\mathcal{O}_{(\frac{1}{6},\frac{5}{6})}=\Big\{(\frac{1}{2},\frac{1}{6}),(\frac{5}{6},\frac{1}{2}),(\frac{5}{6},\frac{5}{6}),(\frac{5}{6},\frac{1}{6}),(\frac{1}{2},\frac{5}{6}),(\frac{1}{6},\frac{5}{6}),(\frac{1}{6},\frac{1}{2}),(\frac{1}{6},\frac{1}{6})\Big\}. 
\end{eqnarray*}
In this case, the statement is as follows, in which $G_6$ denotes the Eisenstein series of weight 6. 

\smallskip

\begin{theorem}

For any non-negative integer $n$, $\tilde{a}_{2n}\big(i\mu;\mathcal{O}_{(\frac{1}{6},\frac{5}{6})}\big)$ is in the one-dimensional space spanned by 
\[
\frac{\Delta(i\mu)G_{6}(i\mu)}{G_4(i\mu)^4}.
\]

\begin{proof}

Since $\mathcal{O}_{(\frac{1}{6},\frac{5}{6})}$ has 8 points, we know from our analysis following Corollary 
\ref{nandn0Cor} that
$v_\rho[p_0,q_0]=v_\rho[-p_0,-q_0]=1$ for some $(p_0,q_0)\in \mathcal{O}_{(\frac{1}{6},\frac{5}{6})}$, 
and $v_P[p,q]=0$ for any other $P$ in the fundamental domain and $(p,q)\in \mathcal{O}_{(\frac{1}{6},\frac{5}{6})}$.

\smallskip

One observes from the explicit orbit that $(\pm p,\pm q)$ cannot be identified with any of $(p+\frac{1}{2},q),(p+\frac{1}{2},q+\frac{1}{2})$, or $(p,q+\frac{1}{2})$ in $S$, so it follows that the simple zero of $\partial_q\vartheta[\pm p_0,\pm q_0](i\mu)$ at $i\mu=\rho$ is not canceled by possible zeros of $\partial_q\vartheta[\pm p+\frac{1}{2},\pm q_0],\partial_q\vartheta[\pm p+\frac{1}{2},q+\frac{1}{2}]$, or $\partial_q\vartheta[\pm p,\pm q+\frac{1}{2}]$ in the numerator of $\tilde{a}[\pm p_0,\pm q_0]_{0}(i\mu)$.\\ 
As a result, a factor of $(\partial_q\vartheta[\pm p_0,\pm q_0])^4$ in the denominator of $\tilde{a}_{0}[\pm p_0,\pm q_0](i\mu)$ implies that $\tilde{a}_{0}[\pm p_0,\pm q_0](i\mu)$ have simple poles of order $4$ at $\rho$, and this is their only singularity, including $i\mu=\infty$. Moreover, similar to our argument in 
Theorem \ref{poleatinfinityThmExplicit},  this implies that $\tilde{a}_{2n}\big(i\mu;\mathcal{O}_{(\frac{1}{6},\frac{5}{6})}\big)$ is meromorphic on $\mathbb{H}\cup \{ \infty \}$ with only a pole of order $4$ at $i\mu=\rho$.

\smallskip

Recall that the Eisenstein series $G_4(i\mu)$ is a modular form of weight $4$ with a simple zero at $i\mu=\rho$, so the function
\begin{eqnarray*}
\tilde{a}_{2n}\big(i\mu;\mathcal{O}_{(\frac{1}{6},\frac{5}{6})}\big)\cdot G_4(i\mu)^4
\end{eqnarray*}
is modular of weight $2+4\times4=18$, and holomorphic on $\mathbb{H}\cup \{\infty \}$, so it is in the space of modular forms of weight $18$. 

\smallskip

In addition, we know from explicit $Q$-expansion of $\tilde{a}_{0}\big(i\mu;\mathcal{O}_{(\frac{1}{6},\frac{5}{6})}\big)$ that the function $\tilde{a}_{2n}\big(i\mu;\mathcal{O}_{(\frac{1}{6},\frac{5}{6})}\big)$ has a simple zero at $Q=0$, so the function $\tilde{a}_{2n}\big(i\mu;\mathcal{O}_{(\frac{1}{6},\frac{5}{6})}\big)\cdot G_4(i\mu)$ has a zero of order $1+4\times 0=1$ at $Q=0$. Namely, $\tilde{a}_{2n}\big(i\mu;\mathcal{O}_{(\frac{1}{6},\frac{5}{6})}\big)\cdot G_4(i\mu)$ is a cusp form of weight $18$. Since the space of cusp forms of weight $18$ is generated by $\Delta\cdot G_6$, we see that $\tilde{a}_{2n}\big(i\mu;\mathcal{O}_{(p,q)}\big)$ is contained in the one-dimensional space generated by  
\[
\frac{\Delta(i\mu)G_{6}(i\mu)}{G_4(i\mu)^4}.
\]

\end{proof}

\end{theorem}

\smallskip

Indeed, explicit $Q$-expansions in the latter case also confirm that: 
\begin{eqnarray*}
\tilde{a}_{0}\big(i\mu;\mathcal{O}_{(\frac{1}{6},\frac{5}{6})}\big) &=& \frac{-294912Q + 438829056Q^2-315542863872Q^3+\cdots}{\pi^3\Lambda^2}\\
&=&-\frac{114688\pi^7}{3375\Lambda^2}\cdot\frac{\Delta(i\mu)G_{6}(i\mu)}{G_4(i\mu)^4},
\end{eqnarray*}
\begin{eqnarray*}
\tilde{a}_{2}\big(i\mu;\mathcal{O}_{(\frac{1}{6},\frac{5}{6})}\big) &=& \frac{294912Q - 438829056Q^2+315542863872Q^3+\cdots}{\pi\Lambda}\\
&=&\frac{114688\pi^9}{3375\Lambda}\cdot\frac{\Delta(i\mu)G_{6}(i\mu)}{G_4(i\mu)^4},
\end{eqnarray*}
\begin{eqnarray*}
\tilde{a}_{4}\big(i\mu;\mathcal{O}_{(\frac{1}{6},\frac{5}{6})}\big) &=& \frac{-270336Q + 402259968Q^2-289247625216 Q^3+\cdots}{5\pi^{-1}\Lambda^0}\\
&=&-\frac{315392\pi^{11}}{50625}\cdot\frac{\Delta(i\mu)G_{6}(i\mu)}{G_4(i\mu)^4}.
\end{eqnarray*}

\smallskip

\section{Conclusions}
\label{ConclusionsSec}

\smallskip

The results obtained in this paper present  a novel
occurrence in quantum cosmology of modular functions and the
vector-valued modular forms considered in the Eichler-Zagier theory
of Jacobi forms \cite{EicZag}.  This was indeed suggested to us by
a combination of two different sources: the existence of an explicit
parametrization of Bianchi IX gravitational instantons in terms of theta
functions with characteristics \cite{BabKor} (see also \cite{Tod, Hit})
and our rationality result about the Seeley-de Witt coefficients
in the asymptotic expansion of the spectral action
for triaxial Bianchi IX metrics \cite{FanFatMar1}. These two
results combined reveal that each Seeley-de Witt coefficient in the
expansion is a rational function, with rational coefficients, in the theta
functions $\vartheta_2, \vartheta_3, \vartheta_4$,  $\vartheta[p,q]$,
$\partial_q \vartheta[p,q]$, $e^{i\pi p}$ and their derivatives (the latter
theta functions are written explicitly in Section \ref{InstantonsSec}).

\smallskip

Bianchi IX gravitational instantons are especially interesting since
they admit an explicit parametrization in terms of elliptic modular functions and theta
functions with characteristics \cite{BabKor}, see also \cite{Tod, Oku, Hit}
and references therein. These are obtained by imposing the
self-duality condition on the Weyl tensor of Bianchi IX metrics,
and reducing the corresponding partial differential equations
to well known ordinary differential equations, namely the Halphen system and the
Painlev\'e VI equation. This result is followed by still another crucial step aimed at
making the result an Einstein metric.  That is, a correct choice of a time-dependent
conformal factor is essential for making the Ricci tensor proportional to the metric.
This fact is relevant to our present work in the following interesting ways.
On the one hand, we have explained in this paper that a similar rationality result holds for a
general time-dependent conformal perturbation of the triaxial Bianchi IX metric
treated in \cite{FanFatMar1}. The result is proved by employing our method
based on Wodzicki's noncommutative residue \cite{Wod1, Wod2} and the K\"unneth formula.
On the other hand, it is necessary to involve the correct conformal factor in our calculations, in
order to obtain the modular transformation properties that we discussed. These properties
add to the many interesting and special features of the Bianchi IX gravitational instantons.

\smallskip

Modular forms appear in a variety of areas in mathematics and physics.
Since modular forms of a certain weight form a finite dimensional linear space and can be
computed with algorithmic methods, they have a wide range of applications. Thus,
it is of great importance in general to find an explicit way of relating any modular function
or modular form that arise from a mathematical structure or from physical problems to
well-known modular forms, whose Fourier expansion, for example, is known.
We have accomplished this task for the modular functions arising from the spectral action
for Bianchi IX metrics, in this paper, by exploring their intimate connection with modular
forms of weight $14$ and cusp forms of weight $18$, both of which form $1$-dimensional
linear spaces. That is, we have shown that,  when the two parameters of a gravitational
instanton are rational, belonging to two different general families,
there is a finite orbit of the parameters, for each case, such that summation
over the orbits leads to the following. In the first case, after multiplication by the cusp form
$\Delta$ of weight 12, each modular function arising from
the Seeley-de Witt coefficient $\tilde a_{2n}$ lands in the space
of modular forms of weight $14$. This indicates that each modular function arising
in this case has only one simple pole, which is located at infinity. In the second
case, after multiplication by $G_4^4$, where $G_4$ is the Eisenstein series of
weight 4, the modular functions arising from the Seeley-de Witt coefficients
land in the 1-dimensional space of cusp forms of weight 18.

\smallskip

In order to illustrate how the present work fits in the general panorama
of other rich arithmetic and number theoretic structures in theoretical
physics,
let us mention the following examples.
A first example is the setting in which Feynman integrals are interpreted
as periods,
see \cite{MarBook} for an overview.  In this case, the relevant amplitude
forms
and domains of integration are algebraic over the rationals or integers
and this fact
has direct implications on the class of numbers that arise as periods.
In particular, an interesting connection to modular forms also arises in this
setting \cite{BrSch}.
A second example in which rational coefficients play an important role is
in the
zero temperature KMS states of quantum statistical mechanical systems. For
example,
in the construction in \cite{ConMarLattices}, which is explained also in
Chapter 3 of
\cite{ConMarBook}, an arithmetic algebra of observables over the rationals is
constructed, whose link to modular functions allows to have KMS states
that take
their values in the modular field.
There are many occurrences of modular forms in
physics, especially in the context of String Theory. The literature
on the subject is extensive and we cannot mention all the relevant
results here, so we only point the reader to a couple of significant
recent examples, such as \cite{ChDuHa, DaMuZa}. The setting
we considered here is very different, as modular forms arise in the
gravity action functional (the spectral action) of a specific class of
gravitational instantons, rather than in settings such as superstring
amplitudes, or counting functions for BPS states, or mirror symmetry.
There are many other examples in the literature of arithmetic structures
arising in
physics, see for example the contributions collected in the volume
\cite{KirWill}. As it is noticeable from the present work as well, it is
in general a challenging
and promising task to further explore the hidden arithmetic structures in
different areas
of physics, including gravity and quantum cosmology.

\smallskip

\appendix

\smallskip

\section{
Proofs of Lemma \ref{transformationsw_22} and Lemma \ref{transformationsw_32}
}

\label{transformationsw_j2appendix}

\smallskip

For the sake of completeness, we provide here the proofs of the modular transformation 
properties of the functions $w_2$, $w_3$ and their derivatives that were stated 
in lemmas  \ref{transformationsw_22} and \ref{transformationsw_32}. In fact the 
proofs are very similar to that of Lemma \ref{transformationsw_12}. For the function $w_2$, using lemmas 
\ref{transformationsvarthetapq} and \ref{transformationsvartheta234}, we can 
write

\begin{eqnarray*}
w_{2}[p,q](\frac{i}{\mu}) &=& \frac{i}{2}\vartheta_{2}(\frac{i}{\mu})\vartheta_{4}(\frac{i}{\mu})\frac{\partial_{q}\vartheta[p+\frac{1}{2},q+\frac{1}{2}](\frac{i}{\mu})}{e^{\pi ip}\vartheta[p,q](\frac{i}{\mu})} \\
&=&\mu^{2}\frac{i}{2}\vartheta_{4}(i\mu)\vartheta_{2}(i\mu)\frac{\partial_{q}\vartheta[-q+\frac{1}{2},p+\frac{1}{2}](i\mu)}{e^{-\pi iq}\vartheta[-q,p](i\mu)}\\
&=&\mu^{2}w_{2}[-q,p](i\mu). 
\end{eqnarray*}
By taking derivatives of the latter with respect to $\mu$ consecutively we obtain: 
\begin{eqnarray*}
w_{2}^{'}[p,q](\frac{i}{\mu}) &=&\frac{d\mu}{d\frac{1}{\mu}}\partial_{\mu}w_{2}[p,q](\frac{i}{\mu}) \\
&=&-\mu^{2}\partial_{\mu}(\mu^{2}w_{2}[-q,p](i\mu))=-\mu^{2}(2\mu w_{2}[-q,p](i\mu)+\mu^{2}w_{2}^{'}[-q,p](i\mu)) \\
&=&-\mu^{4}w_{2}^{'}[-q,p](i\mu)-2\mu^{3}w_{2}[-q,p](i\mu),
\end{eqnarray*}

\begin{eqnarray*}
w_{2}^{''}[p,q](\frac{i}{\mu}) &=&\frac{d\mu}{d\frac{1}{\mu}}\partial_{\mu}w_{2}^{'}[p,q](\frac{i}{\mu}) \\
&=&-\mu^{2}\partial_{\mu}(-\mu^{4}w_{2}^{'}[-q,p](i\mu)-2\mu^{3}w_{2}[-q,p](i\mu))\\
&=&\mu^{6}w_{2}^{''}[-q,p](i\mu)+6\mu^{5}w_{2}^{'}[-q,p](i\mu)+6\mu^{4}w_{2}[-q,p](i\mu),
\end{eqnarray*}

\begin{eqnarray*}
w_{2}^{(3)}[p,q](\frac{i}{\mu}) &=& \frac{d\mu}{d\frac{1}{\mu}}\partial_{\mu}w_{2}^{''}[p,q](\frac{i}{\mu}) \\
&=&-\mu^{2}\partial_{\mu}(\mu^{6}w_{2}^{''}[-q,p](i\mu)+6\mu^{5}w_{2}^{'}[-q,p](i\mu)+6\mu^{4}w_{2}[-q,p](i\mu)) \\
&=&-\mu^{8}w_{2}^{(3)}[-q,p](i\mu)-12\mu^{7}w_{2}^{''}[-q,p](i\mu)-36\mu^{6}w_{2}^{'}[-q,p](i\mu) \\
&&-24\mu^{5}w_{2}[-q,p](i\mu),
\end{eqnarray*}

\begin{eqnarray*}
w_{2}^{(4)}[p,q](\frac{i}{\mu}) &=&-\mu^{2}\partial_{\mu}(-\mu^{8}w_{2}^{(3)}[-q,p](i\mu)-12\mu^{7}w_{2}^{''}[-q,p](i\mu)-36\mu^{6}w_{2}^{'}[-q,p](i\mu) \\
&&-24\mu^{5}w_{2}[-q,p](i\mu)) \\
&=&\mu^{10}w_{2}^{(4)}[-q,p](i\mu)+20\mu^{9}w_{2}^{(3)}[-q,p](i\mu)+120\mu^{8}w_{2}^{''}[-q,p](i\mu) \\
&&+240\mu^{7}w_{2}^{'}[-q,p](i\mu) +120\mu^{6}w_{2}[-q,p](i\mu). 
\end{eqnarray*}

\smallskip

Similarly, for the function $w_3$ we have: 
\begin{eqnarray*}
w_{3}[p,q](\frac{i}{\mu}) &=&-\frac{1}{2}\vartheta_{2}(\frac{i}{\mu})\vartheta_{3}(\frac{i}{\mu})\frac{\partial_{q}\vartheta[p+\frac{1}{2},q](\frac{i}{\mu})}{\vartheta[p,q](\frac{i}{\mu})}\\ &=&\frac{i}{2}\mu^{2}\vartheta_{4}(i\mu)\vartheta_{3}(i\mu)\frac{\partial_{q}\vartheta[-q,p+\frac{1}{2}](i\mu)}{e^{-\pi iq}\vartheta[-q,p](i\mu)} \\
&=&-\mu^{2}w_{1}[-q,p](i\mu),
\end{eqnarray*}
\begin{eqnarray*}
w_{3}^{'}[p,q](\frac{i}{\mu}) &=&
\frac{d\mu}{d\frac{1}{\mu}}\partial_{\mu}w_{3}[p,q](\frac{i}{\mu})
= -\mu^{2}\partial_{\mu}(-\mu^{2}w_{1}[-q,p](i\mu)) \\
&=&-\mu^{2}(-2\mu w_{1}[-q,p](i\mu)-\mu^{2}w_{1}^{'}[-q,p](i\mu)) \\
&=&\mu^{4}w_{1}^{'}[-q,p](i\mu)+2\mu^{3}w_{1}[-q,p](i\mu),
\end{eqnarray*}
\begin{eqnarray*}
w_{3}^{''}[p,q](\frac{i}{\mu}) &=& \frac{d\mu}{d\frac{1}{\mu}}\partial_{\mu}w_{3}^{'}[p,q](\frac{i}{\mu}) \\
&=& -\mu^{2}\partial_{\mu}(\mu^{4}w_{1}^{'}[-q,p](i\mu)+2\mu^{3}w_{1}[-q,p](i\mu)) \\
&=&-\mu^{6}w_{1}^{''}[-q,p](i\mu)-6\mu^{5}w_{1}^{'}[-q,p](i\mu)-6\mu^{4}w_{1}[-q,p](i\mu),
\end{eqnarray*}
\begin{eqnarray*}
w_{3}^{(3)}[p,q](\frac{i}{\mu}) &=& \frac{d\mu}{d\frac{1}{\mu}}\partial_{\mu}w_{3}^{''}[p,q](\frac{i}{\mu}) \\
&=&-\mu^{2}\partial_{\mu}(-\mu^{6}w_{1}^{''}[-q,p](i\mu)-6\mu^{5}w_{1}^{'}[-q,p](i\mu)-6\mu^{4}w_{1}[-q,p](i\mu)) \\
&=&\mu^{8}w_{1}^{(3)}[-q,p](i\mu)+12\mu^{7}w_{1}^{''}[-q,p](i\mu)+36\mu^{6}w_{1}^{'}[-q,p](i\mu) \\
&& +24\mu^{5}w_{1}[-q,p](i\mu),
\end{eqnarray*}
\begin{eqnarray*}
w_{3}^{(4)}[p,q](\frac{i}{\mu})&=&-\mu^{2}\partial_{\mu}(\mu^{8}w_{1}^{(3)}[-q,p](i\mu)+12\mu^{7}w_{1}^{''}[-q,p](i\mu)+36\mu^{6}w_{1}^{'}[-q,p](i\mu) \\
&&+24\mu^{5}w_{1}[-q,p](i\mu)) \\
&=&-\mu^{10}w_{1}^{(4)}[-q,p](i\mu)-20\mu^{9}w_{1}^{(3)}[-q,p](i\mu)-120\mu^{8}w_{1}^{''}[-q,p](i\mu) \\
&&-240\mu^{7}w_{1}^{'}[-q,p](i\mu)-120\mu^{6}w_{1}[-q,p](i\mu).
\end{eqnarray*}

\smallskip

\section{Full expression of the term $\tilde a_4$}
\label{fulla_4appendix}

\smallskip

The expressions for the terms $\tilde a_0$ and $\tilde a_{2}$ appearing 
in the asymptotic expansion \eqref{ExpAsympConformalEq} in which 
$\tilde D$ is the Dirac operator 
of the metric \eqref{ConformalBianchiIXMetricEq1}, 
were given at the end of Section \ref{RationalitySec}.   
Since the term $\tilde a_4$ has a lengthy expression, it is recorded in this 
appendix. It is clear that this term also, in accordance with Theorem 
\ref{ConformalRationlaityThm}, possesses only rational coefficients: 
\begin{eqnarray*}
\tilde{a}_{4} &=&-\frac{w_{1}^{3}w_{2}^{3}}{15w_{3}^{5}}-\frac{w_{1}^{3}w_{3}^{3}}{15w_{2}^{5}}-\frac{w_{2}^{3}w_{3}^{3}}{15w_{1}^{5}}+\frac{w_{1}^{3}w_{2}}{15w_{3}^{3}}+\frac{w_{1}w_{2}^{3}}{15w_{3}^{3}}+\frac{w_{1}^{3}w_{3}}{15w_{2}^{3}}+\frac{w_{2}^{3}w_{3}}{15w_{1}^{3}}+\frac{w_{1}w_{3}^{3}}{15w_{2}^{3}} \\
&&+\frac{w_{2}w_{3}^{3}}{15w_{1}^{3}}-\frac{w_{1}w_{2}}{15w_{3}}-\frac{w_{1}w_{3}}{15w_{2}}-\frac{w_{2}w_{3}}{15w_{1}}-\frac{w_{2}\left(w_{1}'\right){}^{2}}{15w_{1}w_{3}^{3}}-\frac{w_{3}\left(w_{1}'\right){}^{2}}{15w_{1}w_{2}^{3}}-\frac{w_{3}\left(w_{2}'\right){}^{2}}{15w_{1}^{3}w_{2}}
\end{eqnarray*}
\begin{eqnarray*}
&&-\frac{w_{1}\left(w_{2}'\right){}^{2}}{15w_{2}w_{3}^{3}}-\frac{w_{1}\left(w_{3}'\right){}^{2}}{15w_{2}^{3}w_{3}}-\frac{w_{2}\left(w_{3}'\right){}^{2}}{15w_{1}^{3}w_{3}}+\frac{2\left(w_{1}'\right){}^{2}}{15w_{1}w_{2}w_{3}}+\frac{2\left(w_{2}'\right){}^{2}}{15w_{1}w_{2}w_{3}} \\
&&+\frac{2\left(w_{3}'\right){}^{2}}{15w_{1}w_{2}w_{3}}-\frac{w_{2}\left(w_{1}'\right){}^{2}}{18w_{1}^{3}w_{3}}-\frac{w_{3}\left(w_{1}'\right){}^{2}}{18w_{1}^{3}w_{2}}-\frac{w_{1}\left(w_{2}'\right){}^{2}}{18w_{2}^{3}w_{3}}-\frac{w_{3}\left(w_{2}'\right){}^{2}}{18w_{1}w_{2}^{3}}
\end{eqnarray*}
\begin{eqnarray*}
&& -\frac{w_{1}\left(w_{3}'\right){}^{2}}{18w_{2}w_{3}^{3}}-\frac{w_{2}\left(w_{3}'\right){}^{2}}{18w_{1}w_{3}^{3}}-\frac{w_{2}w_{3}\left(w_{1}'\right){}^{2}}{18w_{1}^{5}}-\frac{w_{1}w_{3}\left(w_{2}'\right){}^{2}}{18w_{2}^{5}}-\frac{w_{1}w_{2}\left(w_{3}'\right){}^{2}}{18w_{3}^{5}} \\
&&-\frac{31\left(w_{1}'\right){}^{4}}{90w_{1}^{5}w_{2}w_{3}}-\frac{31\left(w_{2}'\right){}^{4}}{90w_{1}w_{2}^{5}w_{3}}-\frac{31\left(w_{3}'\right){}^{4}}{90w_{1}w_{2}w_{3}^{5}}-\frac{7w_{1}'w_{2}'}{60w_{3}^{3}}-\frac{7w_{1}'w_{3}'}{60w_{2}^{3}}-\frac{7w_{2}'w_{3}'}{60w_{1}^{3}} \\
&&-\frac{w_{1}'w_{2}'}{45w_{1}^{2}w_{3}}-\frac{w_{1}'w_{2}'}{45w_{2}^{2}w_{3}}-\frac{w_{2}'w_{3}'}{45w_{1}w_{3}^{2}}+\frac{5w_{3}w_{1}'w_{2}'}{36w_{1}^{4}}+\frac{5w_{3}w_{1}'w_{2}'}{36w_{2}^{4}}+\frac{5w_{2}w_{1}'w_{3}'}{36w_{1}^{4}}
\end{eqnarray*}
\begin{eqnarray*}
&& +\frac{5w_{2}w_{1}'w_{3}'}{36w_{3}^{4}}+\frac{5w_{1}w_{2}'w_{3}'}{36w_{2}^{4}}+\frac{5w_{1}w_{2}'w_{3}'}{36w_{3}^{4}}+\frac{7w_{3}w_{1}'w_{2}'}{90w_{1}^{2}w_{2}^{2}}+\frac{7w_{2}w_{1}'w_{3}'}{90w_{1}^{2}w_{3}^{2}}+\frac{7w_{1}w_{2}'w_{3}'}{90w_{2}^{2}w_{3}^{2}} \\
&&-\frac{41\left(w_{1}'\right){}^{3}w_{2}'}{180w_{1}^{4}w_{2}^{2}w_{3}}-\frac{41w_{1}'\left(w_{2}'\right){}^{3}}{180w_{1}^{2}w_{2}^{4}w_{3}}-\frac{41\left(w_{1}'\right){}^{3}w_{3}'}{180w_{1}^{4}w_{2}w_{3}^{2}}-\frac{41w_{1}'\left(w_{3}'\right){}^{3}}{180w_{1}^{2}w_{2}w_{3}^{4}}-\frac{41w_{2}'\left(w_{3}'\right){}^{3}}{180w_{1}w_{2}^{2}w_{3}^{4}} \\
&&-\frac{41\left(w_{2}'\right){}^{3}w_{3}'}{180w_{1}w_{2}^{4}w_{3}^{2}}-\frac{23\left(w_{1}'\right){}^{2}\left(w_{2}'\right){}^{2}}{90w_{1}^{3}w_{2}^{3}w_{3}}-\frac{23\left(w_{1}'\right){}^{2}\left(w_{3}'\right){}^{2}}{90w_{1}^{3}w_{2}w_{3}^{3}}-\frac{23\left(w_{2}'\right){}^{2}\left(w_{3}'\right){}^{2}}{90w_{1}w_{2}^{3}w_{3}^{3}} \\
&&-\frac{w_{1}'w_{3}'}{45w_{1}^{2}w_{2}}-\frac{w_{1}'w_{3}'}{45w_{2}w_{3}^{2}}-\frac{w_{2}'w_{3}'}{45w_{1}w_{2}^{2}}-\frac{91\left(w_{1}'\right){}^{2}w_{2}'w_{3}'}{180w_{1}^{3}w_{2}^{2}w_{3}^{2}}-\frac{91w_{1}'\left(w_{2}'\right){}^{2}w_{3}'}{180w_{1}^{2}w_{2}^{3}w_{3}^{2}}
\end{eqnarray*}
\begin{eqnarray*}
&&-\frac{91w_{1}'w_{2}'\left(w_{3}'\right){}^{2}}{180w_{1}^{2}w_{2}^{2}w_{3}^{3}}+\frac{w_{2}w_{1}''}{24w_{3}^{3}}+\frac{w_{3}w_{1}''}{24w_{2}^{3}}+\frac{w_{1}w_{2}''}{24w_{3}^{3}}+\frac{w_{3}w_{2}''}{24w_{1}^{3}}+\frac{w_{1}w_{3}''}{24w_{2}^{3}}+\frac{w_{2}w_{3}''}{24w_{1}^{3}} \\
&&-\frac{w_{1}''}{12w_{2}w_{3}}-\frac{w_{2}''}{12w_{1}w_{3}}-\frac{w_{3}''}{12w_{1}w_{2}}+\frac{w_{2}w_{1}''}{36w_{1}^{2}w_{3}}+\frac{w_{3}w_{1}''}{36w_{1}^{2}w_{2}}+\frac{w_{1}w_{2}''}{36w_{2}^{2}w_{3}} \\
&& -\frac{5w_{2}w_{3}w_{1}''}{72w_{1}^{4}}-\frac{5w_{1}w_{3}w_{2}''}{72w_{2}^{4}}-\frac{5w_{1}w_{2}w_{3}''}{72w_{3}^{4}}+\frac{5\left(w_{1}'\right){}^{2}w_{1}''}{8w_{1}^{4}w_{2}w_{3}}+\frac{5\left(w_{2}'\right){}^{2}w_{2}''}{8w_{1}w_{2}^{4}w_{3}}
\end{eqnarray*}
\begin{eqnarray*}
&&+\frac{5\left(w_{3}'\right){}^{2}w_{3}''}{8w_{1}w_{2}w_{3}^{4}}+\frac{71w_{1}'w_{2}'w_{1}''}{180w_{1}^{3}w_{2}^{2}w_{3}}+\frac{71w_{1}'w_{2}'w_{2}''}{180w_{1}^{2}w_{2}^{3}w_{3}}+\frac{71w_{1}'w_{3}'w_{1}''}{180w_{1}^{3}w_{2}w_{3}^{2}}+\frac{71w_{1}'w_{3}'w_{3}''}{180w_{1}^{2}w_{2}w_{3}^{3}} \\
&& +\frac{71w_{2}'w_{3}'w_{3}''}{180w_{1}w_{2}^{2}w_{3}^{3}}+\frac{71w_{2}'w_{3}'w_{2}''}{180w_{1}w_{2}^{3}w_{3}^{2}}+\frac{41\left(w_{2}'\right){}^{2}w_{1}''}{360w_{1}^{2}w_{2}^{3}w_{3}}+\frac{41\left(w_{3}'\right){}^{2}w_{1}''}{360w_{1}^{2}w_{2}w_{3}^{3}}+\frac{41\left(w_{2}'\right){}^{2}w_{3}''}{360w_{1}w_{2}^{3}w_{3}^{2}} \\
&&+\frac{41\left(w_{3}'\right){}^{2}w_{2}''}{360w_{1}w_{2}^{2}w_{3}^{3}}+\frac{41\left(w_{1}'\right){}^{2}w_{2}''}{360w_{1}^{3}w_{2}^{2}w_{3}}+\frac{41\left(w_{1}'\right){}^{2}w_{3}''}{360w_{1}^{3}w_{2}w_{3}^{2}}+\frac{11w_{2}'w_{3}'w_{1}''}{36w_{1}^{2}w_{2}^{2}w_{3}^{2}}+\frac{11w_{1}'w_{3}'w_{2}''}{36w_{1}^{2}w_{2}^{2}w_{3}^{2}}\\
&& +\frac{11w_{1}'w_{2}'w_{3}''}{36w_{1}^{2}w_{2}^{2}w_{3}^{2}}-\frac{\left(w_{1}''\right){}^{2}}{6w_{1}^{3}w_{2}w_{3}}-\frac{\left(w_{2}''\right){}^{2}}{6w_{1}w_{2}^{3}w_{3}}-\frac{\left(w_{3}''\right){}^{2}}{6w_{1}w_{2}w_{3}^{3}}+\frac{w_{3}w_{2}''}{36w_{1}w_{2}^{2}}
\end{eqnarray*}
\begin{eqnarray*}
&& +\frac{w_{1}w_{3}''}{36w_{2}w_{3}^{2}}+\frac{w_{2}w_{3}''}{36w_{1}w_{3}^{2}}-\frac{w_{1}''w_{2}''}{15w_{1}^{2}w_{2}^{2}w_{3}}-\frac{w_{2}''w_{3}''}{15w_{1}w_{2}^{2}w_{3}^{2}}-\frac{w_{1}''w_{3}''}{15w_{1}^{2}w_{2}w_{3}^{2}}-\frac{w_{1}'w_{1}{}^{(3)}}{6w_{1}^{3}w_{2}w_{3}}\\
&& -\frac{w_{2}'w_{2}{}^{(3)}}{6w_{1}w_{2}^{3}w_{3}}-\frac{w_{3}'w_{3}{}^{(3)}}{6w_{1}w_{2}w_{3}^{3}}-\frac{w_{2}'w_{1}{}^{(3)}}{10w_{1}^{2}w_{2}^{2}w_{3}}-\frac{w_{3}'w_{1}{}^{(3)}}{10w_{1}^{2}w_{2}w_{3}^{2}}-\frac{w_{1}'w_{2}{}^{(3)}}{10w_{1}^{2}w_{2}^{2}w_{3}}-\frac{w_{3}'w_{2}{}^{(3)}}{10w_{1}w_{2}^{2}w_{3}^{2}} \\
&& -\frac{w_{1}'w_{3}{}^{(3)}}{10w_{1}^{2}w_{2}w_{3}^{2}}-\frac{w_{2}'w_{3}{}^{(3)}}{10w_{1}w_{2}^{2}w_{3}^{2}}+\frac{w_{1}{}^{(4)}}{30w_{1}^{2}w_{2}w_{3}}+\frac{w_{2}{}^{(4)}}{30w_{1}w_{2}^{2}w_{3}}+\frac{w_{3}{}^{(4)}}{30w_{1}w_{2}w_{3}^{2}}
\end{eqnarray*}
\begin{eqnarray*}
&& -\frac{w_{1}w_{2}\left(F'\right)^{2}}{72F^{2}w_{3}^{3}}+\frac{w_{1}\left(F'\right)^{2}}{36F^{2}w_{2}w_{3}}+\frac{w_{2}\left(F'\right)^{2}}{36F^{2}w_{1}w_{3}}-\frac{w_{1}w_{3}\left(F'\right)^{2}}{72F^{2}w_{2}^{3}}+\frac{w_{3}\left(F'\right)^{2}}{36F^{2}w_{1}w_{2}}
\\&&
-\frac{w_{2}w_{3}\left(F'\right)^{2}}{72F^{2}w_{1}^{3}}
 -\frac{13\left(F'\right)^{4}}{24F^{4}w_{1}w_{2}w_{3}}+\frac{F'w_{2}w_{1}'}{72Fw_{3}^{3}}-\frac{F'w_{1}'}{36Fw_{2}w_{3}}+\frac{F'w_{2}w_{1}'}{36Fw_{1}^{2}w_{3}}+\frac{F'w_{3}w_{1}'}{72Fw_{2}^{3}} \\
 &&+\frac{F'w_{3}w_{1}'}{36Fw_{1}^{2}w_{2}}-\frac{F'w_{2}w_{3}w_{1}'}{24Fw_{1}^{4}}
 -\frac{41\left(F'\right)^{3}w_{1}'}{120F^{3}w_{1}^{2}w_{2}w_{3}}-\frac{53\left(F'\right)^{2}\left(w_{1}'\right){}^{2}}{360F^{2}w_{1}^{3}w_{2}w_{3}}
 +\frac{F'\left(w_{1}'\right){}^{3}}{24Fw_{1}^{4}w_{2}w_{3}} \\&&
 +\frac{F'w_{1}w_{2}'}{72Fw_{3}^{3}}-\frac{F'w_{2}'}{36Fw_{1}w_{3}}+\frac{F'w_{1}w_{2}'}{36Fw_{2}^{2}w_{3}} +\frac{F'w_{3}w_{2}'}{72Fw_{1}^{3}}-\frac{F'w_{1}w_{3}w_{2}'}{24Fw_{2}^{4}}+\frac{F'w_{3}w_{2}'}{36Fw_{1}w_{2}^{2}}
 \end{eqnarray*}
 \begin{eqnarray*}
 &&-\frac{41\left(F'\right)^{3}w_{2}'}{120F^{3}w_{1}w_{2}^{2}w_{3}}-\frac{23\left(F'\right)^{2}w_{1}'w_{2}'}{90F^{2}w_{1}^{2}w_{2}^{2}w_{3}}-\frac{7F'\left(w_{1}'\right){}^{2}w_{2}'}{40Fw_{1}^{3}w_{2}^{2}w_{3}}-\frac{53\left(F'\right)^{2}\left(w_{2}'\right){}^{2}}{360F^{2}w_{1}w_{2}^{3}w_{3}} \\
 &&-\frac{7F'w_{1}'\left(w_{2}'\right){}^{2}}{40Fw_{1}^{2}w_{2}^{3}w_{3}}+\frac{F'\left(w_{2}'\right){}^{3}}{24Fw_{1}w_{2}^{4}w_{3}}+\frac{F'w_{1}w_{3}'}{72Fw_{2}^{3}}-\frac{F'w_{3}'}{36Fw_{1}w_{2}}+\frac{F'w_{2}w_{3}'}{72Fw_{1}^{3}} \\
&& -\frac{F'w_{1}w_{2}w_{3}'}{24Fw_{3}^{4}}+\frac{F'w_{1}w_{3}'}{36Fw_{2}w_{3}^{2}}+\frac{F'w_{2}w_{3}'}{36Fw_{1}w_{3}^{2}}-\frac{41\left(F'\right)^{3}w_{3}'}{120F^{3}w_{1}w_{2}w_{3}^{2}}-\frac{23\left(F'\right)^{2}w_{1}'w_{3}'}{90F^{2}w_{1}^{2}w_{2}w_{3}^{2}}
\end{eqnarray*}
\begin{eqnarray*}
&&-\frac{7F'\left(w_{1}'\right){}^{2}w_{3}'}{40Fw_{1}^{3}w_{2}w_{3}^{2}}
-\frac{23\left(F'\right)^{2}w_{2}'w_{3}'}{90F^{2}w_{1}w_{2}^{2}w_{3}^{2}}-\frac{17F'w_{1}'w_{2}'w_{3}'}{60Fw_{1}^{2}w_{2}^{2}w_{3}^{2}}-\frac{7F'\left(w_{2}'\right){}^{2}w_{3}'}{40Fw_{1}w_{2}^{3}w_{3}^{2}} \\
&&-\frac{53\left(F'\right)^{2}\left(w_{3}'\right){}^{2}}{360F^{2}w_{1}w_{2}w_{3}^{3}}
 -\frac{7F'w_{1}'\left(w_{3}'\right){}^{2}}{40Fw_{1}^{2}w_{2}w_{3}^{3}}-\frac{7F'w_{2}'\left(w_{3}'\right){}^{2}}{40Fw_{1}w_{2}^{2}w_{3}^{3}}+\frac{F'\left(w_{3}'\right){}^{3}}{24Fw_{1}w_{2}w_{3}^{4}}\\
 &&+\frac{w_{1}w_{2}F''}{72Fw_{3}^{3}}-\frac{w_{1}F''}{36Fw_{2}w_{3}}-\frac{w_{2}F''}{36Fw_{1}w_{3}}
+\frac{w_{1}w_{3}F''}{72Fw_{2}^{3}}-\frac{w_{3}F''}{36Fw_{1}w_{2}}
+\frac{w_{2}w_{3}F''}{72Fw_{1}^{3}}
\end{eqnarray*}
\begin{eqnarray*}
&& +\frac{137\left(F'\right)^{2}F''}{120F^{3}w_{1}w_{2}w_{3}}+\frac{101F'F''w_{1}'}{180F^{2}w_{1}^{2}w_{2}w_{3}}
+\frac{67F''\left(w_{1}'\right){}^{2}}{360Fw_{1}^{3}w_{2}w_{3}}+\frac{101F'F''w_{2}'}{180F^{2}w_{1}w_{2}^{2}w_{3}} \\
&&+\frac{53w_{1}'w_{2}'F''}{180Fw_{1}^{2}w_{2}^{2}w_{3}}+\frac{67\left(w_{2}'\right){}^{2}F''}{360Fw_{1}w_{2}^{3}w_{3}}
+\frac{101F'F''w_{3}'}{180F^{2}w_{1}w_{2}w_{3}^{2}}+\frac{53w_{1}'w_{3}'F''}{180Fw_{1}^{2}w_{2}w_{3}^{2}} \\
&&+\frac{53w_{2}'w_{3}'F''}{180Fw_{1}w_{2}^{2}w_{3}^{2}}
+\frac{67\left(w_{3}'\right){}^{2}F''}{360Fw_{1}w_{2}w_{3}^{3}}-\frac{3\left(F''\right)^{2}}{10F^{2}w_{1}w_{2}w_{3}}
+\frac{41\left(F'\right)^{2}w_{1}''}{360F^{2}w_{1}^{2}w_{2}w_{3}}
\end{eqnarray*}
\begin{eqnarray*}
&&+\frac{7F'w_{1}'w_{1}''}{180Fw_{1}^{3}w_{2}w_{3}}+\frac{23F'w_{2}'w_{1}''}{180Fw_{1}^{2}w_{2}^{2}w_{3}}+\frac{23F'w_{3}'w_{1}''}{180Fw_{1}^{2}w_{2}w_{3}^{2}}-\frac{2F''w_{1}''}{15Fw_{1}^{2}w_{2}w_{3}}  \\ && +\frac{41\left(F'\right)^{2}w_{2}''}{360F^{2}w_{1}w_{2}^{2}w_{3}}+\frac{23F'w_{1}'w_{2}''}{180Fw_{1}^{2}w_{2}^{2}w_{3}}+\frac{7F'w_{2}'w_{2}''}{180Fw_{1}w_{2}^{3}w_{3}}+\frac{23F'w_{3}'w_{2}''}{180Fw_{1}w_{2}^{2}w_{3}^{2}} \\
&&-\frac{2F''w_{2}''}{15Fw_{1}w_{2}^{2}w_{3}}+\frac{41\left(F'\right)^{2}w_{3}''}{360F^{2}w_{1}w_{2}w_{3}^{2}}+\frac{23F'w_{1}'w_{3}''}{180Fw_{1}^{2}w_{2}w_{3}^{2}}+\frac{23F'w_{2}'w_{3}''}{180Fw_{1}w_{2}^{2}w_{3}^{2}}
\end{eqnarray*}
\begin{eqnarray*}
&&+\frac{7F'w_{3}'w_{3}''}{180Fw_{1}w_{2}w_{3}^{3}}-\frac{2F''w_{3}''}{15Fw_{1}w_{2}w_{3}^{2}}-\frac{2F'F^{(3)}}{5F^{2}w_{1}w_{2}w_{3}}-\frac{w_{1}'F^{(3)}(t)}{5Fw_{1}^{2}w_{2}w_{3}}-\frac{w_{2}'F^{(3)}}{5Fw_{1}w_{2}^{2}w_{3}}\\
&&-\frac{w_{3}'F^{(3)}}{5Fw_{1}w_{2}w_{3}^{2}}-\frac{F'w_{1}{}^{(3)}}{30Fw_{1}^{2}w_{2}w_{3}}-\frac{F'w_{2}{}^{(3)}}{30Fw_{1}w_{2}^{2}w_{3}}-\frac{F'w_{3}{}^{(3)}}{30Fw_{1}w_{2}w_{3}^{2}}+\frac{F^{(4)}}{10Fw_{1}w_{2}w_{3}}.
\end{eqnarray*}

\smallskip

Although it seems lengthy, the above expression for $\tilde a_4$ is the result of a significant amount of simplifications and 
cancellations. 

\smallskip

\section{Proof of Theorem \ref{IsoSpecOneParaThm}} \label{IsoSpecOneParaThmPfappendix}

\smallskip

The isospectrality of the Dirac operators mentioned in Theorem \ref{IsoSpecOneParaThm} 
can be proved by using a similar method as in the proof of Theorem \ref{IsoSpecDiracs2paraThm}. 
That is, 
suppose that $u_{n}(\mu,\eta,\phi,\psi)$
is a section of the spin bundle such that $\tilde{D}[q_0]u_{n}(\mu,\eta,\phi,\psi)=\lambda_{n}u_{n}(\mu,\eta,\phi,\psi)$. Then, using the explicit expression \eqref{DiracConfBianchiIXEq} for the Dirac operator 
we can write: 
{\small 
\[
\left(iq_{0}^{-1}\tilde{D}[\frac{1}{q_{0}}]\left(-\gamma^{0}u_{n}(\frac{1}{\mu})\right)\right)\mid_{\mu=\mu_{0}}=
\]
\[
iq_{0}^{-1}(F[\frac{1}{q_{0}}](i\mu_{0})w_{1}[\frac{1}{q_{0}}](i\mu_{0})w_{2}[\frac{1}{q_{0}}](i\mu_{0})w_{3}[\frac{1}{q_{0}}](i\mu_{0}))^{-\frac{1}{2}}\gamma^{0}\partial_{\text{\ensuremath{\mu}}}\left(-\gamma^{0}u_{n}(\frac{1}{\mu})\right)\mid_{\mu=\mu_{0}}
\]
\[
-iq_{0}^{-1}\sin\psi\cdot\left(\frac{F[\frac{1}{q_{0}}](i\mu_{0})w_{2}[\frac{1}{q_{0}}](i\mu_{0})w_{3}[\frac{1}{q_{0}}](i\mu_{0})}{w_{1}[\frac{1}{q_{0}}](i\mu_{0})}\right)^{-\frac{1}{2}}\gamma^{1}\partial_{\eta}\left(-\gamma^{0}u_{n}(\frac{1}{\mu})\right)\mid_{\mu=\mu_{0}}
\]
\[
+iq_{0}^{-1}\cos\psi\cdot\left(\frac{F[\frac{1}{q_{0}}](i\mu_{0})w_{1}[\frac{1}{q_{0}}](i\mu_{0})w_{3}[\frac{1}{q_{0}}](i\mu_{0})}{w_{2}[\frac{1}{q_{0}}](i\mu_{0})}\right)^{-\frac{1}{2}}\gamma^{2}\partial_{\eta}\left(-\gamma^{0}u_{n}(\frac{1}{\mu})\right)\mid_{\mu=\mu_{0}}
\]
\[
+iq_{0}^{-1}\cos\psi\csc\eta\cdot\left(\frac{F[\frac{1}{q_{0}}](i\mu_{0})w_{2}[\frac{1}{q_{0}}](i\mu_{0})w_{3}[\frac{1}{q_{0}}](i\mu_{0})}{w_{1}[\frac{1}{q_{0}}](i\mu_{0})}\right)^{-\frac{1}{2}}\gamma^{1}\partial_{\phi}\left(-\gamma^{0}u_{n}(\frac{1}{\mu})\right)\mid_{\mu=\mu_{0}}\]
\[
+iq_{0}^{-1}\sin\psi\csc\eta\cdot\left(\frac{F[\frac{1}{q_{0}}](i\mu_{0})w_{1}[\frac{1}{q_{0}}](i\mu_{0})w_{3}[\frac{1}{q_{0}}](i\mu_{0})}{w_{2}[\frac{1}{q_{0}}](i\mu_{0})}\right)^{-\frac{1}{2}}\gamma^{2}\partial_{\phi}\left(-\gamma^{0}u_{n}(\frac{1}{\mu})\right)\mid_{\mu=\mu_{0}}
\]
\[
-iq_{0}^{-1}\cos\psi\cot\eta\cdot\left(\frac{F[\frac{1}{q_{0}}](i\mu_{0})w_{2}[\frac{1}{q_{0}}](i\mu_{0})w_{3}[\frac{1}{q_{0}}](i\mu_{0})}{w_{1}[\frac{1}{q_{0}}](i\mu_{0})}\right)^{-\frac{1}{2}}\gamma^{1}\partial_{\psi}\left(-\gamma^{0}u_{n}(\frac{1}{\mu})\right)\mid_{\mu=\mu_{0}}
\]
\[
-iq_{0}^{-1}\sin\psi\cot\eta\cdot\left(\frac{F[\frac{1}{q_{0}}](i\mu_{0})w_{1}[\frac{1}{q_{0}}](i\mu_{0})w_{3}[\frac{1}{q_{0}}](i\mu_{0})}{w_{2}[\frac{1}{q_{0}}](i\mu_{0})}\right)^{-\frac{1}{2}}\gamma^{2}\partial_{\psi}\left(-\gamma^{0}u_{n}(\frac{1}{\mu})\right)\mid_{\mu=\mu_{0}}
\]
\[
+iq_{0}^{-1}\left(\frac{F[\frac{1}{q_{0}}](i\mu_{0})w_{1}[\frac{1}{q_{0}}](i\mu_{0})w_{2}[\frac{1}{q_{0}}](i\mu_{0})}{w_{3}[\frac{1}{q_{0}}](i\mu_{0})}\right)^{-\frac{1}{2}}\gamma^{3}\partial_{\psi}\left(-\gamma^{0}u_{n}(\frac{1}{\mu})\right)\mid_{\mu=\mu_{0}}
\]
\[
 +\frac{iq_{0}^{-1}}{4}\frac{\frac{w_{1}^{'}[\frac{1}{q_{0}}](i\mu_{0})}{w_{1}[\frac{1}{q_{0}}](i\mu_{0})}+\frac{w_{2}^{'}[\frac{1}{q_{0}}](i\mu_{0})}{w_{2}[\frac{1}{q_{0}}](i\mu_{0})}+\frac{w_{3}^{'}[\frac{1}{q_{0}}](i\mu_{0})}{w_{3}[\frac{1}{q_{0}}](i\mu_{0})}+3\frac{F^{'}[\frac{1}{q_{0}}](i\mu_{0})}{F[\frac{1}{q_{0}}](i\mu_{0})}}{\left(F[\frac{1}{q_{0}}](i\mu_{0})w_{1}[\frac{1}{q_{0}}](i\mu_{0})w_{2}[\frac{1}{q_{0}}](i\mu_{0})w_{3}[\frac{1}{q_{0}}](i\mu_{0})\right)^{\frac{1}{2}}}\gamma^{0}\left(-\gamma^{0}u_{n}(\frac{1}{\mu_{0}})\right)
\]
\[
 -\frac{iq_{0}^{-1}}{4}\left(\frac{w_{1}[\frac{1}{q_{0}}](i\mu_{0})w_{2}[\frac{1}{q_{0}}](i\mu_{0})w_{3}[\frac{1}{q_{0}}](i\mu_{0})}{F[\frac{1}{q_{0}}](i\mu_{0})}\right)^{\frac{1}{2}} \times
\]
\[
\qquad \qquad \qquad  \left(\frac{1}{w_{1}^{2}[\frac{1}{q_{0}}](i\mu_{0})}+\frac{1}{w_{2}^{2}[\frac{1}{q_{0}}](i\mu_{0})}+\frac{1}{w_{3}^{2}[\frac{1}{q_{0}}](i\mu_{0})}\right)\gamma^{1}\gamma^{2}\gamma^{3}\left(-\gamma^{0}u_{n}(\frac{1}{\mu_{0}})\right). 
\]
}

\smallskip

Now we can use lemmas \ref{w_1q_0SLem}, \ref{w_2q_0SLem}, \ref{w_3q_0SLem} and \ref{transformationsFlemma}, which explain the modular transformation properties of the 
functions $w_j$ and $F$ in the one-parametric case \eqref{one-parametric} with respect to the 
modular transformation  $S(i \mu) = i/\mu$. Using these lemmas in the above we expression, we 
can write
{\small
\[
\left(iq_{0}^{-1}\tilde{D}[\frac{1}{q_{0}}]\left(-\gamma^{0}u_{n}(\frac{1}{\mu})\right)\right)\mid_{\mu=\mu_{0}}=
\]
\[
-(F[q_{0}](\frac{i}{\mu_{0}})w_{1}[q_{0}](\frac{i}{\mu_{0}})w_{2}[q_{0}](\frac{i}{\mu_{0}})w_{3}[q_{0}](\frac{i}{\mu_{0}}))^{-\frac{1}{2}}\cdot\left(-\partial_{\text{\ensuremath{\frac{1}{\mu}}}}\gamma^{0}\left(-\gamma^{0}u_{n}(\frac{1}{\mu})\right)\mid_{\mu=\mu_{0}}\right)
\]
\[ +\sin\psi\cdot\left(\frac{F[q_{0}](\frac{i}{\mu_{0}})w_{2}[q_{0}](\frac{i}{\mu_{0}})w_{3}[p,q](\frac{i}{\mu_{0}})}{w_{1}[q_{)}](\frac{i}{\mu_{0}})}\right)^{-\frac{1}{2}}\gamma^{0}\gamma^{1}\gamma^{0}\partial_{\eta}\left(-\gamma^{0}u_{n}(\frac{1}{\mu})\right)\mid_{\mu=\mu_{0}}
\]
\[
 -\cos\psi\cdot\left(\frac{F[q_{0}](\frac{i}{\mu_{0}})w_{1}[q_{0}](\frac{i}{\mu_{0}})w_{3}[q_{0}](\frac{i}{\mu_{0}})}{w_{2}[q_{0}](\frac{i}{\mu_{0}})}\right)^{-\frac{1}{2}}\gamma^{0}\gamma^{2}\gamma^{0}\partial_{\eta}\left(-\gamma^{0}u_{n}(\frac{1}{\mu})\right)\mid_{\mu=\mu_{0}}
\]
\[
-\cos\psi\csc\eta\cdot\left(\frac{F[q_{0}](\frac{i}{\mu_{0}})w_{2}[q_{0}](\frac{i}{\mu_{0}})w_{3}[q_{0}](\frac{i}{\mu_{0}})}{w_{1}[q_{0}](\frac{i}{\mu_{0}})}\right)^{-\frac{1}{2}}\gamma^{0}\gamma^{1}\gamma^{0}\partial_{\phi}\left(-\gamma^{0}u_{n}(\frac{1}{\mu})\right)\mid_{\mu=\mu_{0}}
\]
\[
-\sin\psi\csc\eta\cdot\left(\frac{F[q_{0}](\frac{i}{\mu_{0}})w_{1}[q_{0}](\frac{i}{\mu_{0}})w_{3}[q_{0}](\frac{i}{\mu_{0}})}{w_{2}[q_{0}](\frac{i}{\mu_{0}})}\right)^{-\frac{1}{2}}\gamma^{0}\gamma^{2}\gamma^{0}\partial_{\phi}\left(-\gamma^{0}u_{n}(\frac{1}{\mu})\right)\mid_{\mu=\mu_{0}}
\]
\[
+\cos\psi\cot\eta\cdot\left(\frac{F[q_{0}](\frac{i}{\mu_{0}})w_{2}[q_{0}](\frac{i}{\mu_{0}})w_{3}[q_{0}](\frac{i}{\mu_{0}})}{w_{1}[q_{0}](\frac{i}{\mu_{0}})}\right)^{-\frac{1}{2}}\gamma^{0}\gamma^{1}\gamma^{0}\partial_{\psi}\left(-\gamma^{0}u_{n}(\frac{1}{\mu})\right)\mid_{\mu=\mu_{0}}
\]
\[
+\sin\psi\cot\eta\cdot\left(\frac{F[q_{0}](\frac{i}{\mu_{0}})w_{1}[q_{0}](\frac{i}{\mu_{0}})w_{3}[q_{0}](\frac{i}{\mu_{0}})}{w_{2}[q_{0}](\frac{i}{\mu_{0}})}\right)^{-\frac{1}{2}}\gamma^{0}\gamma^{2}\gamma^{0}\partial_{\psi}\left(-\gamma^{0}u_{n}(\frac{1}{\mu})\right)\mid_{\mu=\mu_{0}}
\]
\[
 -\left(\frac{F[q_{0}](\frac{i}{\mu_{0}})w_{1}[q_{0}](\frac{i}{\mu_{0}})w_{2}[q_{0}](\frac{i}{\mu_{0}})}{w_{3}[q_{0}](\frac{i}{\mu_{0}})}\right)^{-\frac{1}{2}}\gamma^{0}\gamma^{3}\gamma^{0}\partial_{\psi}\left(-\gamma^{0}u_{n}(\frac{1}{\mu})\right)\mid_{\mu=\mu_{0}}
\]
\[
 -\frac{1}{4}\frac{-\frac{w_{3}^{'}[q_{0}](\frac{i}{\mu})}{w_{3}[q_{0}](\frac{i}{\mu})}-\frac{2}{\mu}-\frac{w_{2}^{'}[q_{0}](\frac{i}{\mu})}{w_{2}[q_{0}](\frac{i}{\mu})}-\frac{2}{\mu}-\frac{w_{1}^{'}[q_{0}](\frac{i}{\mu})}{w_{1}[q_{0}](\frac{i}{\mu})}-\frac{2}{\mu}-3\frac{F^{'}[q_{0}](\frac{i}{\mu})}{F[q_{0}](\frac{i}{\mu})}+\frac{6}{\mu}}{\left(F[q_{0}](i\mu_{0})w_{1}[q_{0}](i\mu_{0})w_{2}[q_{0}](i\mu_{0})w_{3}[q_{0}](i\mu_{0})\right)^{\frac{1}{2}}}\gamma^{0}\left(-\gamma^{0}u_{n}(\frac{1}{\mu_{0}})\right)
\]
\[
+\frac{1}{4}\left(\frac{w_{1}[q_{0}](\frac{i}{\mu_{0}})w_{2}[q_{0}](\frac{i}{\mu_{0}})w_{3}[q_{0}](\frac{i}{\mu_{0}})}{F[q_{0}](\frac{i}{\mu_{0}})}\right)^{\frac{1}{2}}\times
\]
\[
\qquad \qquad \qquad \left(\frac{1}{w_{1}^{2}[q_{0}](\frac{i}{\mu_{0}})}+\frac{1}{w_{2}^{2}[q_{0}](\frac{i}{\mu_{0}})}+\frac{1}{w_{3}^{2}[q_{0}](\frac{i}{\mu_{0}})}\right)\gamma^{1}\gamma^{2}\gamma^{3}\left(-\gamma^{0}u_{n}(\frac{1}{\mu_{0}})\right)
\]
\begin{eqnarray*}
&=&-\gamma^{0}\left(\tilde{D}[q_{0}]u_{n}(\mu)\right)\mid_{\mu=\frac{1}{\mu_{0}}} \\
&=&\lambda_{n}\left(-\gamma^{0}u_{n}(\frac{1}{\mu}\text{）}\right)\mid_{\mu=\mu_{0}}.
\end{eqnarray*}
}

\smallskip 

Therefore, for any eigenspinor $u_{n}(\mu)$ of $\tilde{D}[q_{0}]$ with the 
eigenvalue $\lambda_{n}$, the spinor $-\gamma^{0}u_{n}(\frac{1}{\mu})$
is an eigenspinor of $iq_{0}^{-1}\tilde{D}[\frac{1}{q_{0}}]$ with
the same eigenvalue $\lambda_{n}$. Hence $\tilde{D}^{2}[q_{0}]$ and $-q_{0}^{-2}\tilde{D}^{2}[\frac{1}{q_{0}}]$
are isospectral. In a similar manner, we can verify that $\tilde{D}^{2}[q_{0}]$
and $\tilde{D}^{2}[q_{0}-i]$ are also isospectral. Notice that in
the proof we implicitly used the fact that $\tilde{D}^{2}[p,q]$ and
$\tilde{D}^{2}[\frac{1}{q_{0}}]$ act on spin bundle sections over
$M_{0}=(a,b)\times\mathbb{S}^{3}$, $\tilde{D}^{2}[p,q+p+\frac{1}{2}]$ and $\tilde{D}^{2}[q_{0}+i]$
act on those over $M_{1}=(a+i,b+i)\times\mathbb{S}^{3}$, whereas
$\tilde{D}^{2}[-q,p]$ and $\tilde{D}^{2}[\frac{1}{q_{0}}]$ act on those
over $M_{2}=(\frac{1}{b},\frac{1}{a})\times\mathbb{S}^{3}$. These
facts ensure that the spinors $u_{n}(\mu-i)$ and $u_{n}(\frac{1}{\mu})$
are well-defined over the base manifolds of the spin bundle sections
 on which the corresponding operators act.

\smallskip

\section*{Acknowledgments}

\smallskip

The first author is supported by a Summer Undergraduate
Research Fellowship at Caltech. The second author thanks the
Institut des Hautes \'Etudes Scientifiques (I.H.E.S) for an excellent
environment and their hospitality in the Summer of 2015, where
this work was partially carried out. The third author is partially
supported by NSF grants DMS-1201512 and PHY-1205440
and by the Perimeter Institute for Theoretical Physics.

\end{document}